\documentclass[a4paper]{amsart}
\usepackage[utf8x]{inputenc}%
\usepackage{hyperref}%
\usepackage{dsfont}%
\usepackage{amssymb}%
\usepackage{esint}%
\usepackage{mathptmx,times}%
\usepackage[all]{xy}%
\usepackage{mathtools}%
\usepackage{enumitem}%
\setitemize[0]{leftmargin=15pt}%

\newtheorem{thm}{Theorem}
\newtheorem{lemma}[thm]{Lemma}
\newtheorem{prop}[thm]{Proposition}
\newtheorem{cor}[thm]{Corollary}

\theoremstyle{definition}
\newtheorem{remark}[thm]{Remark}

\newtheorem{definition}[thm]{Definition}

\newtheorem{example}[thm]{Example}

\newcommand{\ausruf}{!}

\newcommand{\ti}{\breve{\i}}
\newcommand{\vds}{\breve{p}}
\newcommand{\N}{{\mathbb N}}

\newcommand{\R}{{\mathbb R}}
\newcommand{\Z}{{\mathbb Z}}
\newcommand{\HH}{\mathcal{H}}
\newcommand{\ZZ}{\mathcal{Z}}
\newcommand{\PB}{\mathrm{PB}}

\newcommand{\id}{\operatorname{id}}
\newcommand{\im}{{\operatorname{im}}}
\newcommand{\ev}{{\operatorname{ev}}}
\newcommand{\cov}{{\operatorname{cov}}}
\newcommand{\curv}{{\operatorname{curv}}}
\newcommand{\Ad}{\operatorname{Ad}}
\newcommand{\Hom}{\operatorname{Hom}}

\newcommand{\Tor}{\operatorname{Tor}}
\newcommand{\pr}{\operatorname{pr}}

\newcommand{\Ul}{{\mathrm{U}(1)}}       
\newcommand{\g}{\mathfrak{g}}

\makeindex

\begin{document}

\title{Relative differential cohomology}
\author{Christian Becker}
\address{Universit\"at Potsdam, Institut f\"ur Mathematik, Am Neuen Palais 10, 14469 Potsdam, Germany}
\email[C.~Becker]{becker@math.uni-potsdam.de}

\maketitle

\begin{abstract}
We study two notions of relative differential cohomology, using the model of differential characters.
The two notions arise from the two options to construct relative homology, either by cycles of a quotient complex or of a mapping cone complex.  
We discuss the relation of the two notions of relative differential cohomology to each other.
We discuss long exact sequences for both notions, thereby clarifying their relation to absolute differential cohomology.
We construct the external and internal product of relative and absolute characters and show that relative differential cohomology is a right module over the absolute differential cohomology ring.
Finally we construct fiber integration and transgression for relative differential characters.
\end{abstract}

\tableofcontents

\section{Introduction}\label{sec:intro}

Differential cohomology is a refinement of integral cohomology by differential forms.
The first model for differential cohomology is the graded group $\widehat H^*(X;\Z)$ of differential characters, defined by J.~Cheeger and J.~Simons in \cite{CS83}.
Differential characters of degree $k$ are certain homomorphisms $h:Z_{k-1}(X;\Z) \to \Ul$ on the abelian group of smooth singular $(k-1)$-cycles on $X$.
By now there exist lots of different models for differential cohomology, formulated in terms of smooth Deligne (hyper-)cohomology \cite{B93}, \cite{B12}, \cite{CJM04}, de Rham-Feder currents \cite{HL01}, \cite{HLZ03}, \cite{HL06}, \cite{HL08}, differential cocycles \cite{HS05}, \cite{BKS10} or simplicial forms \cite{DL05}.
In low degrees there exist more special models like gerbes, Hitchin gerbes \cite{H01}, and bundle gerbes \cite{M96} for $k=3$ and bundle $2$-gerbes \cite{S04} for $k=4$.
Axiomatic definitions have been provided in \cite{SS08} and \cite{BB13} for differential refinements of ordinary cohomology and in \cite{BS09}, \cite{BS10} for differential refinements of generalized cohomology theories.
Constructions of generalized differential cohomology theories have appeared in \cite{HS05} and \cite{B12}, \cite{BG13}.
As a particular case of generalized cohomology, there are several models of differential $K$-theory \cite{L94}, \cite{FL10}, \cite{BS09}, \cite{SS10}.
Most of these treatments do not cover relative differential cohomology.

In analogy to the case for absolute cohomology, we may define relative differential cohomology as a refinement of relative integral cohomology by differential forms.
Relative differential cohomology groups have been considered in several contexts like differential characters \cite{HL01}, \cite{BT06}, \cite{Z03}, differential cocyles \cite{Z03}, \cite{BT06}, \v{C}ech cocycles \cite{Z03}.
They have also appeared in more special models like relative gerbes \cite{S06a}, and for differential extensions of generalized cohomology theories \cite{U11}.
Relative algebraic differential characters have been studied in \cite{BE98}. 
Relative differential cohomology groups are closely related to trivializations of differential cohomology as considered in \cite{W13}, \cite{R12}.
Physical applications of relative differential cohomology groups have been sketched in \cite{K05} and \cite{S06b}.
Applications to Chern-Simons theory are discussed in \cite{W13} and \cite{B13}.

It seems that a systematic discussion of notions and models for relative differential cohomology including e.g.~uniqueness, module structures, long exact sequences etc.~is still missing.
In the present paper we treat the case of relative differential (ordinary) cohomology.
As in the preceding paper \cite{BB13} we work with the group $\widehat H^*(X;\Z)$ of differential characters as a model for differential cohomology.
The definition and elementary properties of differential characters are easily transferred from absolute to relative cohomology.

We consider the following situation: let $X$ be a smooth manifold and $i_A:A \to X$ the embedding of a smooth submanifold. 
There are two ways to define the relative singular homology: either as the homology of the (smooth singular) mapping cone complex $C_*(i_A;\Z)$ or as the homology of the quotient complex $C_*(X,A;\Z):= C_*(X;\Z) / \im({i_A}_*)$.
There arise two different notions of relative cycles and hence two notions of relative differential characters.

The first option was treated in \cite{BT06} and will be reviewed in Section~\ref{subsec:rel_diff} below. 
The characters on $Z_{k-1}(i_A;\Z)$ thus obtained are called \emph{relative differential characters}.
We adopt this notion, although it would also be appropriate to call them \emph{mapping cone characters}.
We denote the corresponding group of relative differential characters by $\widehat H^k(i_A;\Z)$.
It is a differential refinement of the relative cohomology $H^k(i_A;\Z) \cong H^k(X,A;\Z)$. 
In fact, the notion of relative differential characters is established in \cite{BT06} not just for embeddings but for any smooth maps $\varphi:A \to X$.
This way one obtains a differential refinement $\widehat H^k(\varphi;\Z)$ of the mapping cone cohomology $H^k(\varphi;\Z)$.   
These characters are treated also in \cite[Ch.~8]{BB13}, where we derive a long exact sequences that relates the groups of relative and absolute differential characters.

The second option has appeared in \cite{HL01} for the special case of the inclusion of the boundary $i_{\partial M}:\partial M \to M$ of a smooth manifold with boundary.
We treat this version of relative differential cohomology as groups of characters on $Z_{k-1}(X,A;\Z)$ in detail in Section~\ref{subsec:rel_diff_par}.
Here $A \subset X$ is an arbitrary smooth submanifold.
We denote the corresponding group of differential characters by $\widehat H^k(X,A;\Z)$.
It yields another differential refinement of the relative cohomology $H^k(X,A;\Z)$. 
We show that the group $\widehat H^k(X,A;\Z)$ corresponds to the subgroup of \emph{parallel} characters in $\widehat H^k(i_A;\Z)$.
In this sense, $\widehat H^k(i_A;\Z)$ is finer as a refinement of $H^k(X,A;\Z)$ than $\widehat H^k(X,A;\Z)$.
We derive a long exact sequence that relates the group $\widehat H^k(X,A;\Z)$ to absolute differential cohomology groups on $X$ and $A$.

We clarifiy the relation of the two notions of relative differential cohomology above to another notion that has appeared in the literature, namely the relative Hopkins-Singer groups $\check{H}^k(\varphi;\Z)$ for a smooth map $\varphi:A \to X$ and $\check{H}^k(i_A;\Z)$ for the embedding $i_A:A\to X$ of a smooth submanifold.
These groups have been constructed in \cite{BT06}.
It is shown there that $\check{H}^k(\varphi;\Z)$ is a subquotient of $\widehat H^k(\varphi;\Z)$.
In Section~\ref{subsec:H_check} we show that $\widehat H^k(X,A;\Z)$ is a subgroup of $\check H^k(i_A;\Z)$. 

In Chapter~\ref{sec:products} we discuss internal and external products in differential cohomology.
The internal product and ring structure on $\widehat H^*(X;\Z)$ was constructed first in \cite{CS83}.
Uniqueness of the ring structure was proved in \cite{SS08} and \cite{BB13}.
The proof in \cite[Ch.~6]{BB13} starts from an axiomatic definition of internal and external products and ends up with a new formula for the latter.
In that sense it is constructive.
In the present paper we give a new proof of the key lemma in the uniqueness proof from \cite[Ch.~6]{BB13}.
This new proof starts from the original definition in \cite{CS83} and ends up with the formula in \cite[Ch.~6]{BB13}.
Further, we use the methods from \cite[Ch.~6]{BB13} to construct a product of absolute and relative differential characters.
This provides the graded group $\widehat H^*(\varphi;\Z)$ of relative differential characters with the structure of a module over the ring $\widehat H^*(X;\Z)$ of absolute differential characters.
The module structure is natural with respect to pull-back and the structure maps (curvature, covariant derivative, characteristic class and topological trivializations) are multiplicative.

Last but not least, in Chapter~\ref{sec:fiber_int_trans} we construct fiber integration of relative differential characters and transgression maps as we did for absolute differential characters in \cite[Ch.~7--9]{BB13}.
To some extent thus, the present work is a ``relativization`` of the results obtained in \cite{BB13} on the absolute differential cohomology ring $\widehat H^*(X;\Z)$ of a smooth manifold $X$.
In fact, the ''relativization'' is a generalization of those results from absolute to relative differential cohomology.
The results for absolute differential cohomology are reproduced as a special case.
We show that fiber integration in fiber products is compatible with cross products of characters, and we derive the up-down formula from this.  
We generalize in two ways a result from \cite{BB13} on integration over fibers that bound:
For integration of relative differential characters over fibers that bound, we find topological trivializations of the integrated characters as in \cite{BB13}.
For integration of absolute characters in fiber bundles that bound along a smooth map into the base -- a notion we introduce in Section~\ref{subsec:fiber_boundary} -- we show that the integrated characters admit sections along that map with covariant derivatives prescribed by fiber integration.

The methods in \cite{BB13} use representations of smooth homology classes by certain geometric cycles, namely Kreck's stratifolds \cite{K10}.
In the present paper we need to adapt these representations to mapping cone cohomology classes. 
This is done in Chapter~\ref{sec:stratifolds} below.  
It provides the necessary prerequisites from relative (or mapping cone) stratifold homology that are needed in the rest of the paper.

\bigskip
\emph{Acknowledgments.}
It is a great pleasure to thank Christian B\"ar and Matthias Kreck for very helpful discussions.
Moreover, the author thanks \emph{Sonderforschungsbereich 647} funded by \emph{Deutsche Forschungsgemeinschaft} for financial support.

\section{Stratifold homology}\label{sec:stratifolds}

In this chapter we construct relative stratifold homology as a geometric homology isomorphic to mapping cone homology of a smooth map $\varphi:A \to X$.
We first discuss the concept of thin chains from \cite{BB13}.
This yields the notion of refined fundamental classes of closed oriented smooth manifolds or stratifold.
We introduce the notion of geometric relative cycle and their refined fundamental classes.
We prove that the bordism theory of relative stratifolds in $(X,A)$ is isomorphic to the smooth singular mapping cone homology of a smooth map $\varphi:A \to X$.
Finally, we adapt the construction of the pull-back operation of geometric cycles and of the transfer map for smooth singular cycles in the base of a fiber bundle $\pi:E \to X$ with closed oriented fibers from \cite[Ch.~3]{BB13} to mapping cone relative homology.
 
\subsection{Thin chains}\label{sec:thinchains}
We briefly recall the concept of thin chains defined in \cite[Ch.~3]{BB13} and certain equivalence relations on singular chains and cycles, respectively.
Let $C_k(C;\Z)$ denote the abelian group of smooth singular $k$-chains in a smooth manifold $X$.
\index{+CXZ@$C_*(X;\Z)$, group of singular chains}
Let $Z_k(X;\Z)$ and $B_k(X;\Z)$ denote the subgroups of smooth singular $k$-cycles and $k$-boundaries, respectively.
\index{+ZXZ@$Z_*(X;\Z)$, group of singular cycles}
\index{+BXZ@$B_*(X;\Z)$, group of singular boundaries}
A \emph{thin $k$-chain} is a chain $s \in C_k(X;\Z)$ such that for every $k$-form $\omega \in \Omega^k(X)$, we have $\int_s \omega=0$.
\index{thin chain}
This happens for instance if $s$ is supported in a $(k-1)$-dimensional submanifold.
Thin chains are preserved by the boundary operator and thus form a subcomplex $S_*(X;\Z) \subset C_*(X;\Z)$.
\index{+SXZ@$S_*(X;\Z)$, complex of thin singular chains}

\subsubsection{Degenerate chains.}
Let $\Delta^k:= \{ \sum_{i=0}^k t_i e_i \,|\, \sum_{i=0}^k t_i=1 \} \subset R^{k+1}$ be the standard $k$-simplex.
Let $l_j:\Delta^{k+1} \to \Delta^k$ be the $j$-th degeneracy map.
A smooth singular $(k+1)$-simplex $\sigma:\Delta^{k+1} \to X$ is called \emph{degenerate}, it if is of the form $\sigma =\sigma'\circ l_j$ for some $k$-simplex $\sigma'$ and $j \in \{0,\ldots,k\}$.
Let $D_k(X;\Z) \subset C_k(X;\Z)$ be the submodule generated by degenerate simplexes.
Elements of $D_*(X;\Z)$ are called \emph{degenerate chains}.
\index{degenerate chain}
It is easy to see that degenerate chains are preserved by the boundary operator $\partial$ of the singular chain complex.
Thus $(D_*(X;\Z),\partial)$ is a subcomplex which is well-known to have vanishing homology \cite{T38}.
Hence the homology of the quotient complex $C_*(X;\Z) / D_*(X;\Z)$ is canonically isomorphic to the smooth singular homology $H_*(X;\Z)$. 
\index{+HXZ@$H_*(X;\Z)$, smooth singular homology}
\index{homology!singular $\sim$}

Degenerate chains are special examples of thin chains, i.e.~$D_*(X;\Z) \subset S_*(X;\Z)$, since differential forms vanish upon pull-back by degeneracy maps.
In particular, any degenerate cycle $z \in Z_*(X;\Z) \cap D_*(X;\Z)$ is the boundary of a thin chain:
since $[z]=0 \in H_*(D_*(X;\Z))$, we find a chain $c \in D_*(X;\Z) \subset S_*(X;\Z)$ such that $\partial c = z$.
This might not be the case for arbitrary thin cycles.

\subsubsection{The mapping cone complex.}
Let $\varphi:A \to X$ be a smooth map.
We denote by $C_k(\varphi;\Z) := C_k(X;\Z) \times C_{k-1}(A;\Z)$ the group of $k$-chains in the \emph{mapping cone complex} of $\varphi$.
\index{CvarphiZ@$C_*(\varphi;\Z)$, mapping cone complex}
\index{mapping cone!complex}
The differential $\partial_\varphi:C_k(\varphi;\Z) \to C_{k-1}(\varphi;\Z)$ of the mapping cone complex is defined as $\partial_\varphi(s,t):=(\partial s + \varphi_*t,-\partial t)$.
\index{+Delvarphi@$\partial_\varphi$, mapping cone differential}
\index{mapping cone!differential}
We denote by $Z_k(\varphi;\Z)$ and $B_k(\varphi;\Z)$ the $k$-cycles and $k$-boundaries of this complex.
Moreover, set $S_k(\varphi;\Z):= S_k(X;\Z) \times S_{k-1}(A;\Z)$ for the space of thin chains in the mapping cone complex.
\index{+ZvarphiZ@$Z_*(\varphi;\Z)$, group of mapping cone cycles}
\index{mapping cone!cycles}
\index{+BvarphiZ@$B_*(\varphi;\Z)$, group of mapping cone boundaries}
\index{+SvarphiZ@$S_*(\varphi;\Z)$, group of thin mapping cone chains}

The homology of the mapping cone complex is denoted by $H_*(\varphi;\Z)$.
\index{+HvarphiZ@$H_*(\varphi;\Z)$, mapping cone homology}
\index{mapping cone!homology}
\index{homology!mapping cone $\sim$}
The short exact sequence of chain complexes 
\begin{equation*}
\xymatrix{
0 \ar[r] & C_*(X;\Z) \ar[r] & C_*(\varphi;\Z) \ar[r] & C_{*-1}(A;\Z) \ar[r] & 0 
}
\end{equation*}
induces a long exact sequence of homology groups:
\begin{equation*}
\xymatrix{
\ldots \ar[r] & H_*(X;\Z) \ar[r] & H_*(\varphi;\Z) \ar[r] & H_{*-1}(A;\Z) \ar^{\varphi_*}[r] & H_{*-1}(X;\Z) \ar[r] & \ldots \,.
}
\end{equation*}
The connecting homomorphism coincides with the map on homology induced by $\varphi$.

The mapping cone cochain complex $(C^*(\varphi;\Z),\delta_\varphi)$ associated with the cochain map $\varphi^*:C^*(X;\Z) \to C^*(A;\Z)$ coincides with the dual complex to $(C_*(\varphi;\Z),\partial_\varphi)$.
\index{+Deltavarphi@$\delta_\varphi$, mapping cone codifferential}
\index{mapping cone!codifferential}
The cohomology of this complex is denoted by $H^*(\varphi;\Z)$ and will be referred to as the \emph{mapping cone cohomology}.
\index{+HuvarphiZ@$H^*(\varphi;\Z)$, mapping cone cohomology}
\index{mapping cone!cohomology}
\index{cohomology!mapping cone $\sim$}
We obtain the corresponding long exact sequence:
\begin{equation*}
\xymatrix{
\ldots \ar[r] & H^{*-1}(A;\Z) \ar[r] & H^*(\varphi;\Z) \ar[r] & H^*(X;\Z) \ar^{\varphi^*}[r] & H^*(A;\Z) \ar[r] & \ldots \,. 
}
\end{equation*}

In case $\varphi = i_A:A \hookrightarrow X$ is the inclusion of a subset, we have a natural chain map $q:C_*(i_A;\Z) \to C_*(X,A;\Z)$, $(v,w) \mapsto v + \im({i_A}_*)$.
Here $C_k(X,A;\Z) := C_k(X;\Z) / {i_A}_*(C_k(A;\Z))$ is the \emph{relative chain complex}.
The long exact sequences together with the five lemma provide identifications $H_*(\varphi;\Z) \cong H_*(X,A;\Z)$ and $H^*(\varphi;\Z) \cong H^*(X,A;\Z)$.

Let $(\Omega^*(\varphi),d_\varphi)$ be the relative or mapping cone de Rham complex as defined in \cite[p.~78]{BT82}.
\index{+Omegavarphi@$\Omega^*(\varphi)$, mapping cone de Rham complex}
\index{mapping cone!de Rham complex}
Thus $\Omega^k(\varphi) := \Omega^k(X) \times \Omega^{k-1}(A)$ with the differential $d_\varphi(\omega,\vartheta) := (d\omega,\varphi^*\omega - d\vartheta)$.
\index{+Dvarphi@$d_ \varphi$, mapping cone de Rham differential}
\index{mapping cone!de Rham differential}
We denote the cohomology of this complex by $H^*_\mathrm{dR}(\varphi)$ and call it the \emph{mapping cone de Rham cohomology}. 
\index{+Hdrvarphi@$H^*_\mathrm{dR}(\varphi)$, mapping cone de Rham cohomology}
\index{mapping cone!de Rham cohomology}

Integration of a pair $(\omega,\vartheta) \in \Omega^k(\varphi)$ over a chain $(a,b) \in C_k(\varphi;\Z)$ is defined in the obvious manner:
\index{integration!of mapping cone forms}
\index{mapping cone!integration of forms}
\begin{equation*}
\int_{(a,b)} (\omega,\vartheta) 
:=
\int_a \omega + \int_b \vartheta \,.  
\end{equation*}
Thus pairs of differential forms $(\omega,\varphi) \in \Omega^k(\varphi)$ can be considered as differential cochains in $C^k(\varphi;\R)$. 
Moreover, by the mapping cone Stokes theorem
\begin{equation}\label{eq:stokes}
\int_{\partial_\varphi(a,b)} (\omega,\vartheta)
=
\int_{(\partial a + \varphi_*b,-\partial b)} (\omega,\vartheta)
=
\int_{(a,b)} (d\omega,\varphi^*\omega - d\vartheta) 
=
\int_{(a,b)} d_\varphi(\omega,\vartheta)
\end{equation}
the inclusion $(\Omega^*(\varphi);d_\varphi) \hookrightarrow (C^*(\varphi;\R),\delta_\varphi)$ is a chain map.

The short exact sequence of de Rham complexes 
\begin{equation*}
\xymatrix{
0 \ar[r] & \Omega^{*-1}(A) \ar[r] & \Omega^*(\varphi) \ar[r] & \Omega^*(X) \ar[r] & 0
}
\end{equation*}
gives rise to the long exact sequence 
\begin{equation*}
\xymatrix{
\ldots \ar[r] & H^{k-1}_\mathrm{dR}(A) \ar[r] & H^k_\mathrm{dR}(\varphi) \ar[r] & H^k_\mathrm{dR}(X) \ar[r] & H^k_\mathrm{dR}(A) \ar[r] & \ldots
}
\end{equation*}
in de Rham cohomology.
The de Rham theorem together with the five Lemma yields the identification $H^*_\mathrm{dR}(\varphi) \cong H^*(\varphi;\R)$.

\subsubsection{Equivalence classes}
Let $c \in C_k(X;\Z)$ be a smooth singular $k$-chain in $X$.
We consider its equivalence class modulo thin chains, i.e.~its image in the quotient $C_k(X;\Z)/S_k(X;\Z)$.
We denote this class by $[c]_{S_k}$.
\index{+CSk@$[c]_{S_k}$, equivalence class mod.~thin chains}
Similarly, for a smooth singular $k$-cycle $z \in Z_k(X;\Z)$ we consider its equivalence class modulo boundaries of thin chains, i.e.~the image in the quotient $Z_k(X;\Z) / \partial S_{k+1}(X;\Z)$.
We denote this class by $[z]_{\partial S_{k+1}}$.
\index{+ZSk@$[z]_{\partial S_k}$, equivalence class mod.~boundaries of thin chains}
Finally, for a cycle $(s,t) \in Z_k(\varphi;\Z)$ of the mapping cone complex we consider its equivalence class $[s,t]_{\partial_\varphi S_{k+1}} \in Z_k(\varphi;\Z) / \partial S_{k+1}(\varphi;\Z)$ modulo boundaries of thin chains. 
\index{+StSvarphik@$[s,t]_{\partial_\varphi S_k}$, equivalence class mod.~boundaries of thin chains}

These equivalence classes show up rather naturally when considering fundamental classes of oriented closed manifolds or $p$-stratifolds, as shall be explained in the next section.

Moreover, by definition of thin chains, integration of differential forms over smooth singular cycles descends to the equivalence classes modulo boundaries of thin chains.
Thus we have well-defined integration maps
\begin{align*}
Z_k(\varphi;\Z)/\partial_\varphi S_{k+1}(\varphi;\Z) \times \Omega^k(\varphi) &\to \R\,, \\
([s,t]_{\partial_\varphi S_{k+1}},(\omega,\vartheta)) &\mapsto \int_{[s,t]_{\partial_\varphi S_{k+1}}} (\omega,\vartheta) := \int_{(s,t)} (\omega,\vartheta) \,,
\end{align*}
and similary for absolute cycles and differential forms, see \cite[Ch.~3]{BB13}. 

\subsection{Refined fundamental classes}
Let $M$ be a closed oriented $k$-dimensional smooth manifold.
Triangulation yields a fundamental cycle $z \in Z_k(M;\Z)$. 
Any two such cycles differ by a boundary $\partial a \in B_k(M;\Z)$.
For dimensional reasons, we have $C_{k+1}(M;\Z) = S_{k+1}(M;\Z)$.
Thus the fundamental class of $M$ may be regarded as equivalence class in $Z_k(M;\Z) / \partial S_{k+1}(M;\Z)$.
We denote this class by $[M]_{\partial S_{k+1}}$.
A smooth map $f:M \to X$ yields an induced class $f_*[M]_{\partial S_{k+1}} \in Z_k(X;\Z) / \partial S_{k+1}(X;\Z)$.
We refer to $[M]_{\partial S_{k+1}}$ (resp.~$f_*[M]_{\partial S_{k+1}}$) as the \emph{refined fundamental class} of $M$ (in $X$).
\index{+MSk@$[M]_{\partial S_{k+1}}$, refined fundamental class}
\index{refined fundamental class}
\index{fundamental class!refined $\sim$}

Now let $M$ be a compact oriented smooth $k$-dimensional manifold with boundary and $i_{\partial M}: \partial M \to M$ the inclusion of the boundary.
By triangulation we obtain a smooth singular chain $x \in C_k(M;\Z)$ together with a smooth singular cycle $y \in Z_{k-1}(\partial M;\Z)$ such that $\partial x = y$.
Thus the pair $(x,y) \in C_k(M;\Z) \times C_{k-1}(\partial M;\Z)$ is a cycle in $Z_k(i_{\partial M};\Z)$.
Moreover, $y$ is a fundamental cycle of $\partial M$.

For any two such chains $x,x' \in C_k(M;\Z)$, obtained from triangulations of $M$, we find a chain $a \in C_{k+1}(M;\Z) = S_{k+1}(M;\Z)$ such that $x-x'-\partial a$ is supported in $\partial M$.
Since $M$ is supposed to be $k$-dimensional, we have \mbox{$x-x'-\partial a=:b \in C_k(\partial M;\Z)= S_k(\partial M;\Z)$}.
Thus $M$ comes together with a well-defined equivalence class in $C_k(M;\Z)/S_k(M;\Z)$.
We denote this class by $[M]_{S_k}$.

We may also collect the data on $\partial M$ into the equivalence class:
any two pairs $(x,y)$ and $(x',y')$, obtained as above from traingulations, differ by the relative boundary $\partial (a,b) = (\partial a + b,-\partial b)$ of a pair $(a,b) \in C_{k+1}(M;\Z) \times C_k(\partial M;\Z) = C_{k+1}(i_{\partial M};\Z)$. 
For dimensional reasons, we then have $C_{k+1}(\varphi;\Z) = S_{k+1}(\varphi;\Z)$.  
Thus the pair $(M,\partial M)$ comes together with a well-defined equivalence class in $Z_k(i_{\partial M};\Z)/\partial S_{k+1}(i_{\partial M};\Z)$.
We denote this class by $[M,\partial M]_{\partial S_{k+1}}$.

A commutative diagram of smooth maps
\begin{equation*}
\xymatrix{
\partial M  \ar[rr] \ar_(0.4)g[dr] && M \ar^(0.58)f[dr] & \\
& A \ar_\varphi[rr] && X
} 
\end{equation*}
yields an induced class $(f,g)_*[M,\partial M]_{\partial S_{k+1}} \in Z_k(\varphi;\Z)/\partial_\varphi S_{k+1}(\varphi;\Z)$.
We refer to $[M,\partial M]_{\partial S_{k+1}}$ (resp.~$(f,g)_*[M,\partial M]_{\partial S_{k+1}}$) as the \emph{refined fundamental class} of $(M,\partial M)$ (in $(X,A)$).
\index{+MdelMSk@$[M,\partial M]_{\partial S_{k+1}}$, refined fundamental class}

Restriction to the boundary maps the refined fundamental class 
\mbox{$[M,\partial M]_{\partial S_{k+1}} \in Z_k(i_{\partial M};\Z)/\partial S_{k+1}(i_{\partial M};\Z)$} of $(M,\partial M)$ to the refined fundamental class \mbox{$[\partial M]_{\partial S_k} \in Z_{k-1}(\partial M;\Z)/\partial S_k(\partial M;\Z)$} of the boundary (and similarly for the classes in $(X,A)$).

\begin{remark}
In the same way as explained for a closed oriented smooth manifold, we can associate to a closed oriented $p$-stratifold $M$ a refined fundamental class $[M]_{\partial S_{k+1}} \in Z_k(M;\Z)/\partial S_{k+1}(M;\Z)$ obtained from a triangulation of the top-dimensional stratum.  
Similarly, to a compact oriented $p$-stratifold $M$ with closed boundary $N = \partial M$ we associate an equivalence class $[M,N]_{\partial S_{k+1}} \in Z_k(i_{\partial M};\Z)/\partial S_{k+1}(i_{\partial M};\Z)$.
Restriction to the boundary maps the refined fundamental class $[M,N]_{\partial S_{k+1}}$ to the refined fundamental class $[\partial M]_{\partial S_k}$ of the boundary.  
\end{remark}

\subsection{Geometric cycles}
We use a notion of geometric cycles similar to the one in \cite[Chap.~4]{BB13}.
In contrast to the cycles and boundaries approach formulated there, in the present work we use the original construction from \cite{K10} of geometric or stratifold homology $\HH_*(X)$ as the bordism theory of stratifolds in $X$.

The concept of geometric cycles is motivated by the aim to represent singular homology classes in a smooth manifold $X$ by smooth submanifolds.
By work of Thom \cite{T54} this is not allways possible: in general, there are homology classes not representable as fundamental classes of submanifolds, see e.g.~\cite{BHK02}.
Replacing smooth manifolds by certain types of singular manifolds, it is possible to represent all homology classes by certain geometric spaces.
This is achieved either by pseudomanifolds in the sense of Baas and Sullivan, or by stratifolds in the sense of Kreck \cite{K10}.
Here we use stratifolds, as we did in \cite{BB13}.

Let $n\in \N_0$.
The abelian semigroup $\ZZ_k(X)$ of \emph{geometric $k$-cycles} is the set of smooth maps $f:M \to X$, where $M$ is a $k$-dimensional oriented compact $p$-stratifold without boundary (see \cite[pp.~35 and 43]{K10} for the definition of stratifolds).
The semigroup structure is defined by disjoint union.  
For $n<0$ put $\ZZ_n(X):=\{0\}$.
For a smooth map $f:X \to Y$ we define $f_* : \ZZ_n(X) \to \ZZ_n(Y)$ by concatenation, i.e.~$f_*(M\xrightarrow{g}X) := M\xrightarrow{f\circ g}Y$.
\index{+ZX@$\ZZ_*(X)$, group of geometric cycles}
\index{geometric cycle}

For an oriented stratifold $S$ we denote by $\overline{S}$ the same stratifold $S$ with reversed orientation.
A bordism between geometric $k$-cycles $f:S \to X$ and $f':S' \to X$ in $X$ is a smooth map $F:W \to X$ from a $(k+1)$-dimensional compact oriented $p$-stratifold with boundary $\partial W = \overline{S} \sqcup S'$ such that $F|_{S} = f$ and $F|_{S'} = f'$.
Geometric $k$-cycles $S \xrightarrow{f} X$ and $S' \xrightarrow{f'} X$ are called \emph{bordant} if there exists a bordism between them.
This defines an equivalence relation on $\ZZ_k(X)$.
For transience note that stratifolds with boundary can be glued along their boundary.
 
The bordism class of a geometric $k$-cycle $S \xrightarrow{f} X$ is denoted by $[S \xrightarrow{f} X]$.
The \emph{$k$-th stratifold homology of $X$} is the set of bordism classes $\HH_k(X):= \{[S \xrightarrow{f} X] \,|\, f \in \ZZ_k(X) \}$ of geometric $k$-cycles in $X$.
\index{+HX@$\HH_*(X)$, stratifold homology}
\index{stratifold homology}
\index{homology!stratifold $\sim$}
Orientation reversal defines an involution $\overline{\phantom{X}}$ on $\ZZ_k(X)$ which maps $S \xrightarrow{f} X$ to $\overline{S} \xrightarrow{f} X$.
The involution is compatible with the bordism relation and thus yields a well-defined involution on $\HH_k(X)$.
The stratifold homology $\HH_k(X)$ is an abelian group where the inverse of $[S \xrightarrow{f} X]$ is given by $[\overline{S} \xrightarrow{f} X]$.
A null-bordism of $S \sqcup \overline{S} \xrightarrow{f \sqcup f} X$ is given by $S \times [0,1] \xrightarrow{F} X$, $F(s,t):=f(s)$. 

The stratifold homology group $\HH_k(X)$ can equivalently be defined as the quotient of the semigroup $\ZZ_k(X)$ of geometric cycles by a sub semigroup of geometric boundaries, as in \cite[Ch.~3]{BB13}.
In the present work we use the bordism theory formulation, since this seems more suitable for generalization to relative homology. 

A $k$-dimensional closed $p$-stratifold $M$ has a fundamental class $[M] \in H_k(M;\Z)$, see \cite[p.~186]{K10}.
More precisely, triangulation of the top dimensional stratum yields a smooth singular cycle in $Z_k(M;\Z)$.
Any two such cycles differ by the boundary of a smooth singular chain in $C_{k+1}(M;\Z) = S_{k+1}(M;\Z)$.
Thus we have a well-defined equivalence class $[M]_{\partial S_{k+1}} \in Z_k(M;\Z) / \partial S_{k+1}(M;\Z)$.
As in the previous section, we call it the \emph{refined fundamental class} of $M$.  
Composition with smooth maps yields a well-defined semigroup homomorphism $\ZZ_k(X) \to Z_k(X;\Z) / \partial S_{k+1}(X;\Z)$ mapping the geometric cycle $M\xrightarrow{f}X$ to the equivalence class $f_*[M]_{\partial S_{k+1}}$.
It descends to a group isomorphism $\HH_k(X) \to H_k(X;\Z)$, $[M\xrightarrow{f}X] \mapsto \big[f_*[M]_{\partial S_{k+1}}\big]$, see \cite[p.~186]{K10}.

Differential forms in $\Omega^*(X)$ can be pulled back to a stratifold $S$ along a smooth map $f:S \to X$.
Integration of differential forms over (refined fundamental classes of) compact oriented stratifolds is well-defined and the Stokes theorem holds \cite{E05}.
For a geometric cycle $\zeta \in \ZZ_k(X)$, represented by $S\xrightarrow{f}X$, and a differential form $\omega \in \Omega^k(X)$, we write:
$$
\int_{[\zeta]_{\partial S_{k+1}}} \omega 
=
\int_{[S]_{\partial S_{k+1}}} f^*\omega
=
\int_S f^*\omega \,.
$$

\subsection{Relative stratifold homology}
In this section we introduce relative stratifold homology by adapting the well-known definition of relative bordism groups to stratifolds.
More precisely we modify the classical notion in order to represent the mapping cone homology $H_*(\varphi;\Z)$ of a smooth map as a bordism theory of stratifolds.
The standard construction yields a long exact sequence that relates the absolute and relative stratifold homology groups.

\subsubsection{Geometric relative cycles}
Let $k \geq 0$.
Let $S$ be a $k$-dimensional compact oriented regular $p$-stratifold with boundary $\partial S = T$, and let $\partial T=0$.
By a smooth map $(S,T) \xrightarrow{(f,g)} (X,A)$ we understand a pair of smooth maps $f:S \to X$ and $g:T \to A$ such that the diagram  
\begin{equation*}
\xymatrix{
T \ar[rr] \ar_(0.4)g[dr] && S \ar^(0.58)f[dr] & \\
& A \ar_\varphi[rr] && X
} 
\end{equation*}
commutes.
We define the abelian semigroup 
$$
\ZZ_k(\varphi) := \{ (S,T) \xrightarrow{(f,g)} (X,A) \,|\, \partial S = T, \partial T= \emptyset\}
$$ 
of \emph{geometric relative cycles}, where the semigroup structure is given by disjoint union.
For $k=0$, we have $\ZZ_0(\varphi) = \ZZ_0(X)$.
\index{+Zvarphi@$\ZZ_*(\varphi)$, group of geometric relative cycles}
\index{relative geometric cycle}
\index{geometric cycle!relative $\sim$}

A \emph{bordism} from $(S_0,T_0) \xrightarrow{(f_0,g_0)} (X,A)$ to $(S_1,T_1) \xrightarrow{(f_1,g_1)} (X,A)$ is a smooth map $(W,M) \xrightarrow{(F,G)} (X,A)$ with the following properties:
\index{bordism}
\begin{itemize}
\item[$\bullet$] $W$ is a $(k+1)$-dimensional compact oriented regular $p$-stratifold with boundary $\partial W$.
\item[$\bullet$] The boundary $\partial W$ is the union of compact oriented stratifolds (diffeomorphic to) $\overline{S_0}$, $S_1$ and $M$ such that $\partial M = \partial S_1 \sqcup \partial \overline{S_0}$ and $\overline{S_0} \cap M = \partial \overline{S_0}$ and $S_1 \cap M = \partial S_1$.  
\item[$\bullet$] On the boundary components $S_i$, $i=1,2$, we have $F|_{S_i} = f_i$.
\end{itemize}
Geometric relative $k$-cycles in $(X,A)$ are called \emph{bordant} if there exists a bordism between them.
\index{bordant}
This defines an equivalence relation on $\ZZ_k(\varphi)$. 
For transience of the relation, we note that bordisms can be glued along parts of the boundary, as explained in \cite[Sec.~A.2]{K10}.
The bordism class of a geometric relative $k$-cycle $(f,g) \in \ZZ_k(\varphi)$ is denoted by $[f,g]$.
\index{+Fg@$[f,g]$, bordism class of a smooth map}

For a pair of compact oriented $p$-stratifolds $(S,T)$ as above, we denote by $(\overline{S},\overline{T})$ the same stratifolds with reversed orientation.
Orientation reversal defines an involution on geometric relative cycles which maps $(S,T) \xrightarrow{(f,g)} (X,A)$ to $(\overline{S},\overline{T})\xrightarrow{(f,g)}(X,A)$.
The involution is compatible with the bordism relation in the sense that orientation reversal on bordisms mapping $(W,M) \xrightarrow{(F,G)} (X,A)$ to $(\overline{W},\overline{M}) \xrightarrow{(F,G)} (X,A)$ induces orientation reversal on the cycles related by the bordisms.
In other words, we have a well-defined involution on bordism classes $\overline{\phantom{X}}:[f,g] \mapsto \overline{[f,g]} := [\overline{(f,g)}]$. 

We define the \emph{$k$-th relative stratifold homology group} $\HH_k(\varphi)$ as the set of all bordism classes $[f,g]$ of geometric relative $k$-cycles $(f,g) \in \ZZ_k(\varphi)$.
\index{+Hkvarphi@$\HH_*(\varphi)$ relative stratifold homology}
\index{stratifold homology!relative $\sim$}
\index{homology! relative stratifold $\sim$} 
Given a geometric relative $k$-cycle $(S,T) \xrightarrow{(f,g)} (X,A)$, we find a null-bordism $(W,M) \xrightarrow{(F,G)} (X,A)$ of $(f,g) \sqcup \overline{(f,g)}$ by setting $W:= S \times [0,1]$, $M:= \partial S \times [0,1]$ and $F(s,t):= f(s)$, $G(s,t) := g(s)$.
The semigroup structure on $\ZZ_k(\varphi)$ thus yields an abelian group structure on $\HH_k(\varphi)$, where the inverse of a bordism class $[f,g]$ is given by $\overline{[f,g]}$.
In the following we will write $[f,g] + [f',g']$ instead of $[(f,g) \sqcup (f',g')]$ and correspondingly $-[f,g]$ instead of $\overline{[f,g]}$. 

\subsubsection{Long exact sequence.}
The smooth map $\varphi:A \to X$ induces a semigroup homomorphism $\varphi_*:\ZZ_k(A) \to \ZZ_k(X)$, $f \mapsto \varphi \circ f$.
Moreover, we have the canonical semigroup homomorphisms $i:\ZZ_k(X) \to \ZZ_k(\varphi)$, $f \mapsto (f,\emptyset)$,
and $p:\ZZ_k(\varphi) \to \ZZ_{k-1}(A)$, 
$(f,g) \mapsto g$.
To simplify notation, we write geometric cycles in $X$ as $\zeta \in \ZZ_k(X)$ and cycles relative to $\varphi$ as pairs $(\zeta,\tau) \in \ZZ_k(\varphi)$.
The bordism class of a geometric relative $k$-cycle $(\zeta,\tau) \in \ZZ_k(\varphi)$ is denoted by $[\zeta,\tau]$.
Instead of the empty map $\emptyset$ we write $0$ for the neutral element in the semigroup $\ZZ_{k-1}(A)$ (and similary for the other semigroups of geometric cycles).  
The maps just defined now read $\varphi_*: \zeta \mapsto \varphi_*\zeta$ and $i:\zeta \mapsto (\zeta,0)$ and $p:(\zeta,\tau) \mapsto \tau$. 

These semigroup homomorphisms are bordism invariant and hence descend to group homomorphisms on stratifold homology.
They fit into the following long exact sequence:

\begin{prop}[Long exact sequence]
Let $\varphi:A \to X$ be a smooth map and $k \geq 0$.
Then we have the following exact sequence relating absolute and relative stratifold homology groups:
\begin{equation*}
\xymatrix{
\ldots \ar[r] &
\HH_k(X) \ar^{j_*}[r] & \HH_k(\varphi) \ar^{p_*}[r] & \HH_{k-1}(A) \ar^{\quad\varphi_*}[r] & \ldots \ar[r] & \HH_0(\varphi) \ar[r] & 0 \,.
} 
\end{equation*}
\end{prop}
\index{stratifold homology!long exact sequence}

\begin{proof}
Conceptually, the proof of exactness of the sequence is the same as for oriented smooth bordism homology, see e.g.~\cite[Ch.~21]{tD08}: 

\emph{Exactness at $\HH_k(X)$:}
Let $g:T \to A$ be a geometric cycle in $\ZZ_k(A)$.
Then we have $j(\varphi_*g)=(\varphi \circ g,\emptyset):(T,\emptyset) \to (X,A)$.
We set $W:=T \times [0,1]$ and $F:W \to X$, $F(x,t):=\varphi(g(x))$.
Then we have $\partial W = \overline{T} \times \{0\} \sqcup T \times \{1\}$.
Moreover, we set $M := T \times \{1\}$ and $G:M \to A$, $G(x,1):=g(x)$.
This defines a null bordism $(W,M) \xrightarrow{(F,G)} (X,A)$ of $j(g)=(\varphi \circ g,\emptyset)$.
Thus the composition $\HH_k(A) \xrightarrow{\varphi_*} \HH_k(X) \xrightarrow{j_*} \HH_k(\varphi)$ is the trivial map.

Now let $f:S \to X$ be a geometric cycle in $X$ such that $j_*([S\xrightarrow{f}X])=0 \in \HH_k(\varphi)$. 
Choose a null bordism $(W,M) \xrightarrow{(F,G)} (X,A)$ of $j(f) = (f,\emptyset):(S,\emptyset) \to (X,A)$.
Then we have $\partial M = \partial \overline{S} \sqcup \emptyset = \emptyset$ and $S \cap M = \partial S = \emptyset$, thus $\partial W = \overline{S} \sqcup M$.
Moreover, $F|_S = f$ and $F|_M = \varphi \circ G$.
Thus we have a geometric cycle $G:M \to A$, and a bordism $F:W \to X$ from $f$ to $\varphi \circ g$.
This shows $[S \xrightarrow{f} X] = \varphi_*[M \xrightarrow{g} A]$.  

\emph{Exactness at $\HH_k(\varphi)$:}
By definition, the image of the composition $\ZZ_k(X) \xrightarrow{j} \ZZ_k(\varphi) \xrightarrow{p} \ZZ_{k-1}(A)$ is the empty map, which represents the trivial bordism class. 

Now let $[(S,T) \xrightarrow{(f,g)} (X,A)] \in \HH_k(\varphi)$ be a relative bordism class with $p_*[f,g]=[g]=0 \in \HH_{k-1}(A)$.
Then we find a $k$-dimensional compact oriented stratifold $Q$ with boundary $\partial Q = T$ and a smooth map $G:Q \to A$ such that $G|_{\partial Q} = g$.
Glueing $S$ and $\overline{Q}$ along $T = \partial Q = \partial S$, we obtain a $k$-dimensional oriented closed stratifold $N:=S \cup_T \overline{Q}$.
We extend the maps $f$ and $\varphi \circ G$ to a smooth map $r: N \to X$.

It remains to construct a bordism $(W,M) \xrightarrow{(F,G)} (X,A)$ from $(S,T) \xrightarrow{(f,g)} (X,A)$ to $(N,\emptyset) \xrightarrow{(r,\emptyset)} (X,A)$.
To this end, set $W:= N \times [0,1]$ and $F:W \to X$, $F(n,t):= r(n)$.
Thus $\partial W = \overline{N} \times \{0\} \sqcup N \times \{1\}$. 
We set $M:= Q \times \{0\}$.
This yields a smooth map $(W,M) \xrightarrow{(F,G)} (X,A)$.

We have $\partial W = \overline{S} \times \{0\} \cup M \cup N \times \{1\}$.
Moreover $\partial M = \partial Q \times \{0\} = \partial S \times \{0\} = T$ and $S \times \{0\} \cap M = \partial S = T$ and $N \times \{1\} \cap N = \emptyset$. 
By construction, we have $F|_{S \times \{0\}} = f$ and $F|_{N \times \{1\}} = r$.
Thus we have constructed a bordism $(W,M) \xrightarrow{(F,G)} (X,A)$ from $(S,T) \xrightarrow{(f,g)} (X,A)$ to $(N,\emptyset) \xrightarrow{(r,\emptyset)} (X,A)$.
In other words, we have shown that $[f,g] = j_*[r,\emptyset]$.

\emph{Exactness at $\HH_{k-1}(A)$:}
Let $(S,T) \xrightarrow{(f,g)} (X,A) \in \ZZ_k(\varphi)$ be a geometric relative cycle.
Then we have $\partial S=T$ and $f|_{\partial S} = \varphi \circ g$.
Thus $S \xrightarrow{f} X$ is a null bordism of $\varphi_*(p(f,g)) = \varphi \circ g: T \to X$.   
In other words, the composition $\HH_k(\varphi) \xrightarrow{p_*} \HH_{k-1}(A) \xrightarrow{\varphi_*} \HH_{k-1}(X)$ is trivial.

Now let $T \xrightarrow{g} A$ be a geometric cycle in $\ZZ_{k-1}(A)$ such that $\varphi_*[g] = [T \xrightarrow{\varphi \circ g} X] =0 \in \HH_{k-1}(X)$.
Choose a null bordism $S \xrightarrow{F} X$ of $T \xrightarrow{\varphi \circ g} X$.
Then $(S,T) \xrightarrow{(F,g)} (X,A)$ is a relative geometric cycle in $\ZZ_k(\varphi)$ and $p(F,g) = g$.

\emph{Exactness at $\HH_0(\varphi)$:}
The map $j_*:\HH_0(X) \to \HH_0(\varphi)$ is induced by the isomorphism $j:\ZZ_0(X) \xrightarrow{\cong} \ZZ_0(\varphi)$.
Hence it is surjective.
\end{proof}

\subsubsection{Relative stratifold homology and mapping cone homology}
In \cite{BB13} we used geometric cycles in $\ZZ_k(X)$ to represent singular homology classes in $X$.
A geometric cycle $\zeta \in \ZZ_k(X)$, given by a smooth map $M\xrightarrow{f}X$, yields an equivalence class $f_*[M]_{\partial S_{k+1}} \in Z_k(X;\Z) / \partial S_{k+1}(X;\Z)$.
By a slight abuse of notation, we denote this class as $[\zeta]_{\partial S_{k+1}}$ and refer to it as the \emph{refined fundamental class} of $\zeta$.
\index{refined fundamental class}
\index{fundamental class!refined $\sim$}
The map $\ZZ_k(X) \to Z_k(X;\Z)/\partial S_{k+1}(X;\Z)$, $\zeta \mapsto [\zeta]_{\partial S_{k+1}}$, is a semigroup homomorphism and commutes with the boundary operators.
By \cite[p.~186]{K10}, the induced map $\HH_k(X) \to H_k(X;\Z)$, $[\zeta] \mapsto f_*[M]$, is a group isomorphism. 
Similarly, a geometric relative cycle $(\zeta,\tau) \in \ZZ_k(\varphi)$, given by a smooth map $(S,T)\xrightarrow{(f,g)}(X,A)$, has a \emph{refined fundamental class} $[\zeta,\tau]_{\partial_\varphi S_{k+1}}:=(f,g)_*[S,T]_{\partial S_{k+1}} \in Z_k(\varphi;\Z) / \partial_\varphi S_{k+1}(\varphi;\Z)$.

Taking refined fundamental classes commutes with maps $i:\ZZ_k(X) \to \ZZ_k(\varphi)$ and $p:\ZZ_k(\varphi) \to \ZZ_{k-1}(A)$ defined above:
Restriction to the boundary maps the refined fundamental class $[\zeta,\tau]_{\partial_\varphi S_{k+1}} \in Z_k(\varphi;\Z) / \partial_\varphi S_{k+1}(\varphi;\Z)$ to the refined fundamental class $[\tau]_{\partial S_k} \in Z_{k-1}(A;\Z) / \partial S_k(A;\Z)$ of the boundary.
Similarly, under the map $i:\ZZ_k(X) \to \ZZ_k(\varphi)$ the refined fundamental class $[\zeta]_{\partial S_{k+1}} \in \Z_k(X;\Z)/\partial S_{k+1}(X;\Z)$ of a geometric cycle is mapped to the refined fundamental class $[\zeta,\emptyset]_{\partial_\varphi S_{k+1}} \in \Z_k(\varphi;\Z) / \partial_\varphi S_{k+1}(\varphi;\Z)$ of the corresponding relative cycle.

Let $(\zeta,\tau),(\zeta',\tau') \in \ZZ_k(\varphi)$ be represented by smooth maps $(S,T)\xrightarrow{(f,g)}(X,A)$ and $(S',T')\xrightarrow{(f',g')}(X,A)$.
Let $(W,M)\xrightarrow{(F,G)}(X,A)$ be a bordism from $(S,T)\xrightarrow{(f,g)}(X,A)$ to $(S',T')\xrightarrow{(f',g')}(X,A)$.
Choose a triangulation of $W$ and the induced triangulations of $S, S', M \subset \partial W$.
We thus obtain a chain $(w,m) \in C_{k+1}(W;\Z) \times C_k(M;\Z)$.
Denote the corresponding fundamental cycles of $(S,T)$ and $(S',T')$ by $(s,t)$ and $(s',t')$, respectively.
By definition of the bordism relation, we find 
$$
(f',g')_*(s',t') - (f,g)_*(s,t) = \partial_\varphi\big((F,G)_*(w,-m)\big) \,.
$$
This yields for the refined fundamental classes: 
\begin{equation}\label{eq:ffggwm}
[\zeta',\tau']_{\partial_\varphi S_{k+1}} - [\zeta,\tau]_{\partial_\varphi S_{k+1}}
=
\partial_\varphi \big((F,G)_*[W,\overline{M}]_{S_{k+1}}\big) \,.
\end{equation}
In particular, the fundamental classes coincide: $(f',g')_*[S',T'] = (f,g)_*[S,T] \in H_k(\varphi;\Z)$.

Hence the refined fundamental class homomorphism 
$\ZZ_k(\varphi) \to Z_k(\varphi;\Z)/\partial_\varphi S_{k+1}(\varphi;\Z)$, $(\zeta,\tau) \mapsto [\zeta,\tau]_{\partial_\varphi S_{k+1}}:=(f,g)_*[S,T]_{\partial S_{k+1}}$, descends to a group homomorphism $\HH_k(\varphi) \to H_k(\varphi;\Z)$, $[\zeta,\tau] \mapsto \big[[\zeta,\tau]_{\partial_\varphi S_{k+1}}\big] = (f,g)_*[S,T]$.
Here $[S,T] \in H_k(S,T;\Z)$ denotes the usual fundamental class of the stratifold $S$ with boundary $T$ and $\big[[\zeta,\tau]_{\partial_\varphi S_{k+1}}\big]$ denotes the image of the refined fundamental class $[\zeta,\tau]_{\partial_\varphi S_{k+1}} \in Z_k(\varphi;\Z) / \partial_\varphi S_{k+1}(\varphi;\Z)$ in the mapping cone homology $H_k(\varphi;\Z) = Z_k(\varphi;\Z) / \partial_\varphi C_{k+1}(\varphi;\Z)$.

In fact, this map is a group isomorphism.
Thus geometric relative cycles represent homology classes of the mapping cone:
 
\begin{thm}[Relative stratifold homology]\label{prop:rel_homology}
Let $\varphi:A \to X$ be a smooth map and $k \geq 0$.
Then the map $\HH_k(\varphi) \to H_k(\varphi;\Z)$, $[\zeta,\tau] \mapsto \big[[\zeta,\tau]_{\partial_\varphi S_{k+1}}\big]$, is a group isomorphism. 
\end{thm}
\index{Theorem!Relative stratifold homology}

\begin{proof}
The refined fundamental classes of absolute and relative geometric cycles yield a commutative diagram
\begin{equation}
\xymatrix{
\ZZ_k(A) \ar^{\varphi_*}[r] \ar[d] 
& 
\ZZ_k(X) \ar^i[r] \ar[d] 
& 
\ZZ_k(\varphi) \ar^p[r] \ar[d] 
& 
\ZZ_{k-1}(A) \ar[d] \ar^{\varphi_*}[r] \ar[d]
&
\ZZ_{k-1}(X) \ar[d] \\
\frac{Z_k(A;\Z)}{\partial S_{k+1}(A;\Z)} \ar[r]
&
\frac{Z_k(X;\Z)}{\partial S_{k+1}(X;\Z)} \ar[r] 
& 
\frac{Z_k(\varphi;\Z)}{\partial_\varphi S_{k+1}(\varphi;\Z)} \ar[r]  
& \frac{Z_{k-1}(A;\Z)}{\partial S_k(A;\Z)}   \ar[r]
&
\frac{Z_{k-1}(X;\Z)}{\partial S_k(X;\Z)} \,.
}
\label{eq:diag_ref_fund}
\end{equation}
In the induced diagram on homology, the two left as well as the two right vertical maps are group isomorphisms by \cite[p.~186]{K10}.
By the five lemma, so is the middle vertical map.
\end{proof}

\subsubsection{Integration of differential forms.}
As above let $\Omega^*(\varphi)$ be the mapping cone de Rham complex of a smooth map $\varphi:A \to X$.
Integration of differential forms in $\Omega^*(\varphi)$ over refined fundamental classes of geometric relative cycles is well-defined.
For $(\omega,\vartheta) \in \Omega^k(\varphi)$ and $(\zeta,\tau) \in \ZZ_k(\varphi)$, represented by $(S,T)\xrightarrow{(f,g)}(X,A)$, we write:
$$
\int_{[\zeta,\tau]_{\partial_\varphi S_{k+1}}} (\omega,\vartheta)
=
\int_{[S,T]_{\partial S_{k+1}}} (f,g)^*(\omega,\vartheta) 
=
\int_{(S,T)} (f,g)^*(\omega,\vartheta) \,.
$$  

\subsection{The cross product}\label{subsec:cross}
For geometric cycles $\zeta \in \ZZ_k(X)$ and $\zeta' \in \ZZ_{k'}(X')$ the cartesian product of the corresponding stratifolds defines a cross product on stratifold homology, see \cite[Ch.~10]{K10} and \cite[Ch.~6]{BB13}:
if $\zeta$ is represented by $M\xrightarrow{f}X$ and $\zeta'$ is represented by $M'\xrightarrow{f'}X'$ then the cross product $\zeta \times \zeta'$ is the stratifold represented by $M\times M'\xrightarrow{f \times f'}X \times X'$.
This cartesian product of stratifolds is compatible with bordism: if $W\xrightarrow{F}X$ is a bordism from $\zeta_0 \to \zeta_1$, then $W \times S'\xrightarrow{F \times f'}X \times X'$ is a bordism from $\zeta_0 \times \zeta'$ to $\zeta_1 \times \zeta'$, and similarly for bordisms of the second factor.
Thus the cartesian product descends to a product of stratifold bordism groups.
This coincides with the homology cross product under the isomorphism $\HH_*(X) \xrightarrow{\cong} H_*(X;\Z)$.

\subsubsection{Cross products of geometric cycles.}
Analogously, we define the cross product of a geometric relative cycle $(\zeta,\tau) \in \ZZ_k(\varphi)$, represented by $(S,T)\xrightarrow{(f,g)}(X,A)$, with a geometric cycle $\zeta' \in \ZZ_{k'}(X')$, represented by $S'\rightarrow{f'}X'$:
the stratifold
$$
(S,T) \times S'\xrightarrow{(f,g) \times f'}(X,A) \times X'
$$
represents a geometric relative cycle $(\zeta,\tau) \times \zeta' \in \ZZ_{k+k'}(\varphi\times \id_{X'})$.
\index{cross product!of geometric cycles}

The cartesian product is compatible with the bordism relation:
If $(W,M)\xrightarrow{(F,G)}(X,A)$ is a bordism from $(\zeta_0,\tau_0)$ to $(\zeta_1,\tau_1)$, then $(W,M) \times S'\xrightarrow{(F,G) \times f'}(X,A) \times X'$ is a bordism from $(\zeta_0,\tau_0) \times \zeta'$ to $(\zeta_1,\tau_1) \times \zeta'$.
Likewise, if $W'\xrightarrow{F'}X'$ is a bordism from $\zeta'_0$ to $\zeta'_1$, then $(S,T) \times W'\xrightarrow{(f,g) \times F'}(X,A) \times X'$ is a bordism from $(\zeta,\tau) \times \zeta'_0$ to $(\zeta,\tau) \times \zeta'_1$.
Thus the cartesian product descends to a cross product 
$$
\times: \HH_*(\varphi) \otimes \HH_*(X') \to \HH_*(\varphi \times \id_{X'})
$$
on stratifold homology.

Choosing triangulations of the stratifolds involved and refining them to triangulations of the various cartesian products, it is easy to see that the cross product on stratifold bordism groups coincides with the ordinary homology cross product.
In the same way, the cross product 
$\times: \HH_*(\varphi) \otimes \HH_*(X') \to \HH_*(\varphi \times \id_{X'})$
on stratifold bordism groups is identified with the homology cross product
$\times: H_*(\varphi;\Z) \otimes H_*(X';\Z) \to H_*(\varphi \times \id_{X'}\;Z)$.
\index{cross product!on stratifold homology}
\index{+X@$\times$, cross product}

\subsection{The pull-back operation}\label{sec:PB}
Let $\pi:E \to X$ be a fiber bundle with closed oriented fibers.  
Let $\varphi:A \to X$ be a smooth map and $\Phi:\varphi^*E \to E$ the induced fiber bundle map in the pull-back diagram
\begin{equation*}
\xymatrix{
\varphi^*E \ar^\Phi[rr] \ar_\pi[d] && E \ar^\pi[d] \\
A \ar_\varphi[rr] && X \,.
} 
\end{equation*}
We adapt the pull-back operation $\PB_\bullet$ on geometric cycles in the base of a fiber bundle from \cite[Ch.~4]{BB13}.
In the notation of the present paper, we define $\PB_E:\ZZ_k(X) \to \ZZ_{k+\dim F}(E)$ by mapping the geometric cycle $M\xrightarrow{f}X$ to $f^*E\xrightarrow{F}E$.
\index{+PBE@$\PB_E$, pull-back operation}
\index{pull-back operation}
Here $F:f^*E \to E$ denotes the induced bundle map on the total space of the pull-back bundle $\pi:f^*E \to M$.
Similarly, we may define a pull-back operation $\PB_{E,\varphi^*E}:\ZZ_k(\varphi) \to \ZZ_{k+\dim F}(\Phi)$ by mapping the relative cycle $(S,T)\xrightarrow{(f,g)}(X,A)$ to $(f^*E,g^*(\varphi^*E))\xrightarrow{(F,G)}(E,\varphi^*E)$.
Here $G:g^*(\varphi^*E) \to \varphi^*E$ is the bundle map in the pull-back diagram induced by $g:T \to A$ and the bundle $\varphi^*E \to A$.
\index{+PBEE@$\PB_{E,\varphi^* E}$, pull-back operation}
  
These maps fit into the following commutative diagram of pull-back bundles:  
\begin{equation}
\xymatrix{
g^*(\varphi^*E) \ar[rrr] \ar[dd] \ar^G[dr] &&& f^*E \ar'[d][dd] \ar^F[dr] \\
& \varphi^*E \ar^(0.4)\Phi[rrr] \ar[dd] &&& E \ar[dd] \\ 
T \ar'[r][rrr] \ar_g[dr] &&& S \ar^f[dr] \\
& A \ar_(0.4)\varphi[rrr] &&& X \\ 
} \label{eq:pull-back_diag}
\end{equation}
Since $T = \partial S$ and $\varphi \circ g = f|_{\partial S}$, we have in particular $g^*(\varphi^*E) = (f|_{\partial S})^*E = \partial(f^*E) $ and $\Phi \circ G = F$.
Thus $(f^*E,g^*(\varphi^*E))\xrightarrow{(F,G)}(E,\varphi^*E)$ indeed defines a geometric cycle in $\ZZ_{k+\dim F}(\Phi)$. 
\index{diagram!pull-back bundles}

The pull-back operation $\PB_\bullet$ on geometric relative cycles is compatible with the maps $i:\ZZ_k(X) \to \ZZ_k(\varphi)$ and $p:\ZZ_k(\varphi) \to \ZZ_{k-1}(A)$ and the pull-back operations on geometric cycles in $X$ and $A$, respectively.  
Thus we have a commutative diagram of pull-back operations: 
\begin{equation}
\xymatrix{
\ZZ_{k+\dim F}(E) \ar^i[rr] && \ZZ_{k+\dim F}(\Phi) \ar^p[rr] && \ZZ_{k-1+\dim F}(\varphi^*E) \\
\ZZ_k(X) \ar_i[rr] \ar_{\PB_E}[u] && \ZZ_k(\varphi) \ar_p[rr] \ar_{\PB_{E,\varphi^*E}}[u] && \ZZ_{k-1}(A) \ar_{\PB_{\varphi^*E}}[u] \,.
}
\label{eq:diag_PB}
\end{equation}

\subsubsection{Compatibility with fiber integration of differential forms.}
As above let $\Omega^*(\Phi)$ be the mapping cone de Rham complex for the induced bundle map $\Phi:\varphi^*E \to E$. 
Fiber integration on the relative de Rham complex is defined componentwise:
for $(\omega,\vartheta) \in \Omega^k(\Phi)$ put 
\begin{equation*}
\fint_F (\omega,\vartheta) 
:=
\left( \fint_F \omega \,,\, \fint_F \vartheta \right) \,.
\end{equation*}
This obviously defines a map $\fint_F:\Omega^k(\Phi) \to \Omega^{k-\dim F}(\varphi)$.
\index{fiber integration!of mapping cone forms}

Fiber integration of differential forms is natural with respect to pull-back along the induced bundle maps in the pull-back diagram~\eqref{eq:pull-back_diag}.
In other words, for a smooth map $(Y,B)\xrightarrow{(f,g)}(X,A)$ and the corresponding map $(f^*E,g^*(\varphi^*E))\xrightarrow{(F,G)}(E,\varphi^*E)$, we have:
\begin{equation}
\fint_F (F,G)^*(\omega,\vartheta) 
=
(f,g)^* \fint_F (\omega,\vartheta) \,. \label{eq:fint_nat}
\end{equation}
Moreover, fiber integration is compatible with the mapping cone de Rham differentials:  
\index{fiber integration!compatibility with mapping cone de Rham differential}
\begin{align}
d_\varphi \fint_F (\omega,\vartheta)
&= 
\left( d\fint_F \omega \,,\, \varphi^* \fint_F \omega + d \fint_F \vartheta \right) \notag \\
&=
\left( \fint_F d\omega \,,\, \fint_F \Phi^*\omega + d\vartheta \right) \notag \\
&=
\fint_F d_\Phi (\omega,\vartheta) \,. \label{eq:fiber_d_commute}
\end{align}
Thus fiber integration of differential forms descends to a well-defined homomorphism $H^k_\mathrm{dR}(\Phi) \to H^{k-\dim F}_\mathrm{dR}(\varphi)$ of the mapping cone de Rham cohomologies.

The pull-back operation $\PB_\bullet$ is compatible with fiber integration of differential forms in the following sense: 
Let $(\zeta,\tau) \in \ZZ_k(\varphi)$ be a geometric relative cycle, represented by $(S,T)\xrightarrow{(f,g)}(X,A)$, and $(\omega,\vartheta) \in \Omega^{k+\dim F}(\Phi)$.
Then we have:
\begin{align}
\int_{[\PB_{E,\varphi^*E}(\zeta,\tau)]_{\partial_\Phi S_{k+\dim F+1}}} (\omega,\vartheta)
&=
\int_{(f^*E,g^*(\varphi^*E))} (F,G)^*(\omega,\vartheta) \notag \\
&\stackrel{\eqref{eq:fint_nat}}{=}
\int_{(S,T)} (f,g)^* \fint_F (\omega,\vartheta) \notag \\ 
&=
\int_{[\zeta,\tau]_{\partial_\varphi S_{k+1}}} \fint_F(\omega,\vartheta) \,. \label{eq:PB_int}
\end{align}
\index{pull-back operation!compatibility with fiber integration}

\subsubsection{Compatibility with bordism and refined fundamental classes.}
The pull-back operation $\PB_\bullet$ is compatible with the bordism relation in the following sense:
Let $(\zeta,\tau), (\zeta',\tau') \in \ZZ_k(\varphi)$ be geometric relative cycles, represented by $(S,T)\xrightarrow{(f,g)}(X,A)$ and $(S',T')\xrightarrow{(f',g')}(X,A)$, respectively. 
Let $(W,M)\xrightarrow{(F,G)}(X,A)$ be a bordism from $(\zeta,\tau)$ to $(\zeta',\tau')$.
Then the induced bundle map $(F^*E,G^*(\varphi^*E))\xrightarrow{(\bf{F},\bf{G})}(E,\varphi^*E)$ defines a bordism from $\PB_{E,\varphi^*E}(\zeta,\tau)$ to $\PB_{E,\varphi^*E}(\zeta',\tau')$.
Consequently the pull-back operation yields a homomorphism on relative stratifold homology $\HH_k(\varphi) \to \HH_{k+\dim F}(\Phi)$, $[(\zeta,\tau)] \mapsto [\PB_{E,\varphi^*E}(\zeta',\tau')]$.
This may be considered as a transfer map on relative stratifold homology. 
\index{pull-back operation!compatibility with bordism}

As above let $(W,M)\xrightarrow{(F,G)}(X,A)$ be a bordism from $(\zeta,\tau)$ to $(\zeta',\tau')$.
Choose a triangulation of the stratifold $F^*E$ and the induced triangulations of $G^*(\varphi^*E)$.
This yields a smooth singular chain $(a,b) \in C_k(F^*E;\Z) \times C_{k-1}(G^*(\varphi^*E);\Z)$.
Restricting the bundle $G^*(\varphi^*E) \to M$ to the subspaces $S$ and $S'$ of $\partial W$, we obtain induced triangulations of $f^*E$ and ${f'}^*E$.
Let $(x,y), (x',y') \in Z_{k+\dim F}(\Phi;\Z)$ be the induced fundamental cycles.
Since $(F^*E,G^*(\varphi^*E))\xrightarrow{(\bf{F},\bf{G})}(E,\varphi^*E)$ is a bordism from $\PB_{E,\varphi^*E}(\zeta,\tau)$ to $\PB_{E,\varphi^*E}(\zeta',\tau')$, we obtain from the bordism relation:
$$
(x',y') - (x,y) = \partial_\Phi \big(({\bf F},{\bf G})_*(a,-b)\big) \,.
$$  
In particular, we obtain for the refined fundamental classes:
\begin{align}
[\PB_{E,\varphi^*E}(\zeta',\tau')]_{\partial_\Phi S_{k+\dim F+1}} - [\PB_{E,\varphi^*E}(\zeta,\tau)&]_{\partial_\Phi S_{k+\dim F+1}} \notag \\
&= 
\partial_\Phi \big(({\bf F},{\bf G})_*[(F^*E,\overline{G^*(\varphi^*E)})]_{S_{k+\dim F+1}}\big) \,. \label{eq:bord_fund_cycle}
\end{align}
We use this relation in the following section to construct transfer maps on the level of cycles.
\index{pull-back operation!compatibility with refined fundamental class}

\subsubsection{Compatibility with fiber products.}
The pull-back operation for fiber bundles \mbox{$\pi:E \to X$} and $\pi':E' \to X'$ with closed oriented fibers is compatible with the cross product of geometric cycles and the pull-back operation for the fiber product $\pi\times\pi':E \times E' \to X \times X'$.
This means that the following diagram is graded commutative:
\index{pull-back operation!compatibility fiber products}
\begin{equation}
\xymatrix{
\ZZ_{k+\dim F}(\Phi) \otimes \ZZ_{k'+\dim F'}(E') \ar^\times[rr] && \ZZ_{k+k'+\dim(F \times F')}(\Phi\times \id_{E'}) \\
\ZZ_k(\varphi)  \otimes  \ZZ_{k'}(X') \ar_\times[rr] \ar@<-25pt>[u]^{\PB_{E'}} \ar@<+16pt>[u]^{\PB_{E,\varphi^*E}} && \ZZ_{k+k'}(\varphi\times \id_{X'}) \ar@<-5pt>[u]_{\PB_{(E,\varphi^*E)\times E'}} \,.
}
\end{equation}
The graded commutativity is caused by orientation conventions: cartesian products carry the ordinary product orientation while fiber bundles are oriented like products of first the base and then the fiber.
The fiber product $\pi \times \pi':E \times E' \to X \times X'$ carries the orientation of a fiber bundle over $X \times X'$ with fiber $F \times F'$.
This orientation might differ from the product orientation of the total spaces (in case the total spaces carry an orientation).
Explicitly, for cycles $(\zeta,\tau) \in \ZZ_k(\varphi)$ and $\zeta' \in \ZZ_{k'}(X')$, we have:
\begin{equation}\label{eq:prod_PB}
\PB_{(E,\varphi^*E) \times E'}((\zeta,\tau) \times \zeta')
=
(-1)^{k' \cdot \dim F} \cdot \PB_{E,\varphi^*E}(\zeta,\tau) \times \PB_{E'}(\zeta') \,.
\end{equation}
We use this relation in the following section to construct transfer maps compatible with the cross product.

\subsection{Transfer maps}\label{sec:transfer}
As above let $\pi:E \to X$ be a fiber bundle with closed oriented fibers and $\varphi:A \to X$ a smooth map.  
In \cite[Ch.~4]{BB13} we used geometric cycles to construct a transfer map $\lambda:C_k(X;\Z) \to C_{k+\dim F}(E;\Z)$ that commutes with the boundary operator.
Moreover, it is compatible with fiber integration of differential forms in the sense that for any smooth singular cycle $z \in Z_k(X;\Z)$ and any closed differential form $\omega \in \Omega^{k+\dim F}(E)$, we have $\int_z \fint_F \omega = \int_{\lambda(z)} \omega$.
Here we construct a transfer map $\lambda_{\varphi}:Z_k(\varphi;\Z) \to Z_{k+\dim F}(\Phi;\Z)$ that commutes with the maps $i:Z_k(X;\Z) \to Z_k(\varphi;\Z)$ and $p:Z_k(\varphi;\Z) \to Z_{k-1}(A;\Z)$.
These transfer maps are used in Section~\ref{subsec_fiber_int} to construct fiber integration on relative differential cohomology.

\subsubsection{Representation by geometric cycles.}
The construction of the transfer map in \cite[Ch.~4]{BB13} is based on the pull-back operation $\PB_\bullet$ on geometric cycles and a homomorphism that chooses geometric cycles to represent homomology classes.
More precisely, for any singular cycle $z \in Z_k(X;\Z)$ choose a geometric cycle $\zeta(z)$ and a singular chain $a(z)$ such that $[z-\partial a(z)]_{\partial S_{k+1}} = [\zeta(z)]_{\partial S_{k+1}}$.
In other words, the refined fundamental class of $\zeta(z)$ represents the homology class of $z$.
These choices can be made into homomorphisms $\zeta:Z_k(X;\Z) \to \ZZ_k(X)$ and $a:Z_k(X;\Z) \to C_{k+1}(X;\Z)$ be first defining them on a basis and then extending linearly.

Now we do the same for cycles of the mapping cone complex of a smooth map $\varphi:A \to X$.
By Theorem~\ref{prop:rel_homology}, the relative stratifold homology $\HH_k(\varphi)$ is isomorphic to the mapping cone homology $H_k(\varphi;\Z)$. 
Thus for any relative cycle $(s,t) \in Z_k(\varphi;\Z)$ we may choose a geometric relative cycle $(\zeta,\tau) \in \ZZ_k(\varphi)$ such that its bordims class $[\zeta,\tau] \in \HH_k(\varphi)$ maps the mapping cone cohomology class $[s,t] \in H_k(\varphi;\Z)$ under the isomorphism $\HH_k(\varphi) \to H_k(\varphi;\Z)$ from Theorem~\ref{prop:rel_homology}. 
In particular, we find a singular chain $(a,b) \in C_{k+1}(\varphi;\Z)$  such that $[\zeta,\tau]_{\partial_\varphi S_{k+1}} = [(s,t) - \partial_\varphi(a,b)]_{\partial_\varphi S_{k+1}}$. 
We say that the geometric relative cycle $(\zeta,\tau) \in \ZZ_k(\varphi)$ \emph{represents} the homology class $[s,t] \in H_k(\varphi;\Z)$.
 
We may organize the choice of geometric relative cycles $(\zeta,\tau)$ and singular chains $(a,b)$ 
satisfying $[\zeta,\tau]_{\partial S_{k+1}} = [(s,t) - \partial_\varphi(a,b)]_{\partial S_{k+1}}$ into group homomorphisms 
\index{+Zetatauvarphi@$(\zeta,\tau)_\varphi$, group homomorphism}
\index{+Abvarphi@$(a,b)_\varphi$, group homomorphism}
\begin{align*}
(\zeta,\tau)_\varphi: Z_k(\varphi;\Z) &\to \ZZ_k(\varphi) \,, \quad (s,t) \mapsto (\zeta,\tau)_\varphi(s,t) \,, \\
(a,b)_\varphi: Z_k(\varphi;\Z) &\to C_{k+1}(\varphi;\Z) \,, \quad  (s,t) \mapsto (a,b)_\varphi(s,t) =(a(s,t),b(s,t)) \,, 
\end{align*}
by first defining them on a basis of $\Z_k(\varphi;\Z)$ and then extending linearly.
These homomorphisms can be made compatible with the maps $i$ and $p$ as follows:  
The group $Z_k(\varphi;\Z)$ of relative cycles sits in the split exact sequence
\begin{equation}\label{eq:ex_seq_Z}
\xymatrix{
0
\ar[r] &
Z_k(X;\Z)
\ar[r]_(0.45)i &
Z_k(\varphi;\Z)
\ar[r]_p  &
Z_{k-1}(A,\Z) \ar@<-4pt>@{.>}[l]_\sigma
\ar[r] &
0 
}
\end{equation}
where $i:s \mapsto (s,0)$ and $p:(s,t) \mapsto t$. 
Choose a splitting $\sigma:Z_{k-1}(A;\Z) \to Z_k(\varphi;\Z)$. 
From bases of $Z_k(X;\Z)$ and $Z_{k-1}(A;\Z)$ and the splitting $\sigma$ we obtain a basis of $Z_k(\varphi;\Z)$.
We may thus choose the homomorphism $(\zeta,\tau)_\varphi:Z_k(\varphi;\Z) \to \ZZ_k(\varphi)$ compatible with the maps $p$ and $i$ and the homomorphism $\zeta$ defined on absolute cycles as follows:
For basis elements $(s,t)$ in the image of $i:Z_k(X;\Z) \to Z_k(\varphi;\Z))$, put $(\zeta,\tau)_\varphi(s,t) := (\zeta(s),0)$.
For complementary basis elements $(s,t)=\sigma(t)$, obtained from a basis of $Z_{k-1}(A;\Z)$, choose $(\zeta,\tau)_\varphi(s,t) \in \ZZ_k(\varphi)$ such that $p((\zeta,\tau)_\varphi(s,t)) = \tau(s,t) = \zeta(t) \in \ZZ_{k-1}(A)$.
Then extend linearly.
This yields a commutative diagram
\begin{equation}\label{eq:zetazetatauzeta}
\xymatrix{
Z_k(X;\Z) \ar^i[rr] \ar^\zeta[d] && Z_k(\varphi;\Z) \ar^p[rr] \ar^{(\zeta,\tau)_\varphi}[d] && Z_{k-1}(A;\Z) \ar^\zeta[d] \\
\ZZ_k(X) \ar_i[rr] && \ZZ_k(\varphi) \ar_p[rr] && \ZZ_{k-1}(A) \,. 
}
\end{equation}
Similarly, we may choose the homomorphism $(a,b)_\varphi:Z_k(\varphi;\Z) \to C_{k+1}(\varphi;\Z)$ compatible with the maps $i$ and $p$ and the homomorphism $a$ defined on absolute cycles.
Using the splitting $\sigma$, we write a cycle $(s,t) \in Z_k(\varphi;\Z)$ as $(s,t) = (z,0) + \sigma(p(s,t)) = i(z) + \sigma(t)$.
For compatibility with the map $i$ we may simply put $(a,b)(i(z)) := (a(z),0)$.
However, compatibility with the map $p$ involves a sign:
Since 
$$
(s,t) - \partial_\varphi (a,b)_\varphi(s,t) = (s - \partial a(s,t) - \varphi_*a(s,t),t+\partial b(s,t))
$$ represents the fundamental class of $(\zeta,\tau)_\varphi(s,t)$ and $t-\partial a(t)$ represents the fundamental class of $\zeta(t) = \zeta(p(s,t)) = p((\zeta,\tau)(s,t))$, we are forced to put $b(\sigma(t)) := -a(t) \in C_k(A;\Z)$.

\subsubsection{Compatibility of transfer maps.}
We define the mapping cone transfer map $\lambda_\varphi:Z_k(\varphi;\Z) \to Z_{k+\dim F}(\Phi;\Z)$ as follows: for any cycle $(s,t)$ in a basis of $Z_k(\varphi;\Z)$ choose a cycle $\lambda_\varphi(s,t) \in Z_{k+\dim F}(\Phi;\Z)$ such that the equivalence classes modulo boundaries of thin chains satisfy the equation $[\lambda_\varphi(s,t)]_{\partial_\Phi S_{k+\dim F+1}} = [\PB_{E,\varphi^*E}(\zeta,\tau)_\varphi(s,t)]_{\partial_\Phi S_{k+\dim F+1}}$. 
Then extend $\lambda_\varphi$ as a homomorphism.
\index{transfer map!mapping cone $\sim$}
\index{mapping cone!transfer map}
\index{+Lambdavarphi@$\lambda_\varphi$, transfer map}

By \eqref{eq:diag_ref_fund} and \eqref{eq:diag_PB}, the refined fundamental classes and the pull-back operations are compatible with the maps $i$ and $p$ that relate absolute and relative cycles.
By the choice of homomorphisms $(\zeta,\tau)_\varphi$ and $(a,b)_\varphi$ above, we may also choose the transfer map $\lambda_\varphi$ compatible with $i$ and $p$.
We thus obtain a commutative diagram of transfer maps:
\begin{equation}\label{eq:diag_lambda_commute}
\xymatrix{
Z_{k+\dim F}(E;\Z) \ar^i[rr] && Z_{k+\dim F}(\Phi;\Z) \ar^p[rr] && Z_{k+\dim F-1}(\varphi^*E;\Z) \\
Z_k(X;\Z) \ar_i[rr] \ar_\lambda[u] && Z_k(\varphi;\Z) \ar_p[rr] \ar_{\lambda_\varphi}[u] && Z_{k-1}(A;\Z) \ar_\lambda[u] \,.
}
\end{equation}
\index{diagram!of transfer maps}

Like for absolute cycles, the transfer map is compatible with fiber integration of forms in the mapping cone de Rham complex of the induced bundle map:
\index{transfer map!compatibility with fiber integration}
Let $(\omega,\vartheta) \in \Omega^{k+\dim F}(\Phi)$ and $(s,t) \in Z_k(\varphi;\Z)$.
Then we have:
\allowdisplaybreaks{
\begin{align}
\int_{\lambda_\varphi(s,t)} (\omega,\vartheta)
&=
\int_{[\PB_{E,\varphi^*E}\big((\zeta,\tau)_\varphi(s,t)\big)]_{\partial_\Phi S_{k+\dim F+1}}} (\omega,\vartheta) \notag \\
&\stackrel{\eqref{eq:PB_int}}{=}
\int_{[(\zeta,\tau)_\varphi(s,t)]_{\partial_\varphi S_{k+1}}} \fint_F (\omega,\vartheta) \notag \\
&= 
\int_{(s,t) - \partial_\varphi((a,b)_\varphi(s,t))} \fint_F (\omega,\vartheta) \notag \\
&=
\int_{(s,t)} \fint_F (\omega,\vartheta) - \int_{(a,b)_\varphi(s,t)} d_\varphi \fint_F (\omega,\vartheta) \,. \label{eq:lambda_fint_1}
\end{align}
}
In particular, for a $d_\Phi$-closed pair $(\omega,\vartheta)$ we have:  
\begin{equation*}
\int_{\lambda_\varphi(s,t)} (\omega,\vartheta)
=
\int_{(s,t)} \fint_F (\omega,\vartheta) \,. 
\end{equation*}

As in \cite[Ch.~4]{BB13}, we extend the transfer map $\lambda_\varphi:Z_k(\varphi;\Z) \to Z_{k+\dim F}(\Phi;\Z)$ to a homomorphism $\lambda_\varphi:C_k(\varphi;\Z) \to C_{k+\dim F}(\Phi;\Z)$ such that 
\begin{equation}\label{eq:del_lambda}
\partial_\Phi \circ \lambda_\varphi 
= 
\lambda_\varphi \circ \partial_\varphi \,.
\end{equation}
This is done by appropriate choices on a basis of $C_k(\varphi;\Z)$. 
For basis elements in $Z_k(\varphi;\Z)$ we choose $\lambda_\varphi$ as before.
For complementary basis elements $(x,y) \in C_k(\varphi;\Z)$ we choose $\lambda_\varphi(x,y) \in C_{k+\dim F}(\Phi;\Z)$ such that \eqref{eq:del_lambda} holds.
By changing $\lambda_\varphi$ on the complementary basis elements if necessary\footnote{This is explained in detail in \cite[Ch.~4]{BB13} for the case of absolute chains.}, we may as well assume that for any $d_\Phi$-closed $(\omega,\vartheta) \in \Omega^{k+\dim F}(\Phi)$, we have: 
\begin{equation}\label{eq:lambda_fint}
\int_{\lambda_{\varphi}(x,y)} (\omega,\vartheta)
=
\int_{(x,y) - (a,b)(\partial_\varphi(x,y))} \fint_F (\omega,\vartheta) \,.
\end{equation}
The transfer map $\lambda_\varphi$ will be used to define fiber integration for relative differential characters.
\index{transfer map!extension to chains}

\subsubsection{Multiplicativity of transfer maps.}
Let $\pi:E \to X$ and $\pi':E' \to X'$ be fiber bundles with compact oriented fibers $F$ and $F$' and let $\pi \times \pi':E \times E' \to X \times X'$ be the fiber product.
It carries the orientation of a fiber bundle over $X \times X'$ with fiber $F \times F'$.
This orientation might differ from the product orientation of the total spaces.

Using multiplicativity of the pull-back operation \eqref{eq:prod_PB} and a splitting of the K\"unneth sequence as constructed in Section~\ref{subsec:module} below, we may choose the transfer map for the product bundle in such a way that we obtain the following graded commutative diagram: 
\begin{equation}\label{eq:diag_lambda_mult}
\xymatrix{
Z_{k+\dim F}(\Phi;\Z) \otimes Z_{*+\dim F'}(E') \ar^\times[rr] && Z_{k+k'+\dim(F\times F')}(\Phi\times\id_{E'};\Z) \\
Z_k(\varphi;\Z) \otimes Z_{k'}(X';\Z) \ar^\times[rr] \ar@<-21pt>[u]^{\lambda'} \ar@<+22pt>[u]^{\lambda_\varphi} && Z_{k+k'}(\varphi\times\id_{X'};\Z) \ar@<-7pt>[u]_{\lambda_{\varphi\times\id_{X'}}} \,. \\
} 
\end{equation}
More precisely, for cycles $(s,t) \in Z_k(\varphi;\Z)$ and $z' \in Z_{k'}(X';\Z)$, we have: 
\begin{equation}\label{eq:lambda_cross}
\lambda_\varphi(s,t) \times \lambda'(z')
=
(-1)^{k' \cdot \dim F} \cdot \lambda_{\varphi \times \id_{X'}}((s,t) \times z') \,. 
\end{equation}
This relation is used in the proof of the up-down formula in Section~\ref{subsec:up-down} below. 
\index{transfer map!multplicativity}

\section{Differential characters}\label{sec:diff_charact}
In this chapter we discuss (absolute and relative) differential characters as models for (absolute and relative) differential cohomology classes.
Differential characters had been introduced in \cite{CS83} as certain homomorphisms $h:Z_{k-1}(X;\Z) \to \Ul$ on the group of smooth singular cycles in a smooth manifold $X$.
The graded abelian group $\widehat H^*(X;\Z)$ of differential characters was the first model for what is now called differential cohomology.\footnote{It is convenient to shift the degree of the differential characters by $+1$ as compared to the original definition from \cite{CS83}. 
Thus a degree $k$ differential character has curvature and characteristic class of degree $k$.}

The definition from \cite{CS83} can be easily adapted to homomorphisms on relative cycles.
As explained in the introduction, there are two ways to define relative singular homology $H_*(X,A;\Z)$, either as the homology of the mapping cone complex of the inclusion $i_A:A \to X$ or as the homology of the quotient complex $C_*(X,A;\Z) := C_*(X,\Z) / \im({i_A}_*)$. 
Hence there arise two notions of relative cycles and thus two ways two adapt the notion of differential characters.
The corresponding groups of differential characters are both refinements of the relative cohomology $H^*(X,A;\Z)$ by differential forms.
Both notions appear rather naturally. 

In Section~\ref{subsec:rel_diff} we review the notion and elementary properties of \emph{relative differential characters} as introduced in \cite{BT06}.
These are characters on the group of cycles $Z_*(\varphi;\Z)$ in the mapping cone complex of a smooth map $\varphi:A \to X$.
The graded abelian group of those characters is denoted by $\widehat H^*(\varphi;\Z)$.
We review the results from \cite[Ch.~8]{BB13}, including a long exact secquence for $\widehat H^*(\varphi;\Z)$ and the groups of absolute differential characters on $X$ and $A$.

In Section~\ref{subsec:rel_diff_par} we discuss differential characters on the group $Z_*(X,A;\Z)$ of relative cycles.
Here $A \subset X$ is an embedded smooth submanifold.
The graded abelian group of differential characters on $Z_*(X,A;\Z)$ is denoted by $\widehat H^*(X;A;\Z)$.  
We prove a long exact sequence that relates the group $\widehat H^*(X;A;\Z)$ to the groups of absolute differential characters on $X$ and $A$.
Further, we show that $\widehat H^*(X,A;\Z)$ coincides with the subgroup of \emph{parallel} characters in $\widehat H^*(i_A;\Z)$ . 
In Section~\ref{subsec:H_check} we clarify the relation o the groups $\widehat H^k(\varphi;\Z)$ and $\widehat H^k(X,A;\Z)$ to another notion of relative differential cohomology that has appeared in the literature: the relative Hopkins-Singer groups $\check{H}^k(\varphi;\Z)$ for a smooth map $\varphi:A \to X$ and $\check{H}^k(i_A;\Z)$ for the embedding $i_A:A\to X$ of a smooth manifold.
These groups have been constructed in \cite{BT06}.

\subsection{Relative differential characters}\label{subsec:rel_diff}
Differential characters on a smooth manifold $X$ were introduced by Cheeger and Simons in \cite{CS83}.
Differential characters relative to a smooth map $\varphi:A \to X$ were introduced in \cite{BT06}.
We briefly review the definition and elementary properties of relative differential characters, thereby treating the absolute differential characters of \cite{CS83} as a special case.

\subsubsection{Characters on mapping cone cycles.}
Let $(C_*(X;\Z),\partial)$ be the complex of smooth singular chains in $X$. 
The {\em mapping cone complex} of a smooth map $\varphi: A \to X$ is the complex $C_k(\varphi;\Z) := C_k(X;\Z) \times C_{k-1}(A;\Z)$ of pairs of smooth singular chains with the differential $\partial_\phi (s,t) := (\partial s + \varphi_* t,-\partial t)$.
The homology $H_k(\varphi;\Z)$ of this complex coincides with the homology of the mapping cone of $\varphi$ in the topological sense.
For the special case of an embedding $i_A:A \hookrightarrow X$ it coincides with the relative homology $H_k(X,A;\Z)$.

As above, let $\Omega^*(\varphi)$ be the mapping cone de Rham complex with the differential $d_\varphi(\omega,\vartheta) := (d\omega,\varphi^*\omega - d\vartheta)$.
The mapping cone de Rham cohomology $H^*_\mathrm{dR}(\varphi)$ is canonically identified with the real mapping cone cohomology $H^*(\varphi;\R)$.

We denote by $Z_k(\varphi;\Z)$ the group of $k$-cycles of the mapping cone complex and by $B_k(\varphi;\Z)$ the group of $k$-boundaries. 
Let $k \geq 1$.
The group of degree-$k$ relative differential characters is defined as follows: 
\index{+HvarphiZ@$\widehat H^*(\varphi;\Z)$, group of relative differential characters}
\index{differential character!relative $\sim$}
\index{relative differential character}
\begin{equation*}
\widehat H^k(\varphi;\Z) 
:= 
\big\{\, h \in \Hom(Z_{k-1}(\varphi;\Z),\Ul) \,\big|\, h \circ \partial_\varphi \in \Omega^k(\varphi) \,\big\} \,.  
\end{equation*}
The notation $h \circ \partial_\varphi \in \Omega^k(\varphi)$ means that there exists $(\omega,\vartheta) \in \Omega^k(\varphi)$ such that for every smooth singular chain $(a,b) \in C_k(\varphi;\Z)$ we have
\begin{equation}
h(\partial_\varphi(a,b))
= 
\exp \Big( 2 \pi i \int_{(a,b)} (\omega,\vartheta) \Big) \,. \label{eq:def_rel_diff_charact_2}
\end{equation}
The form $\omega=: \curv(h) \in \Omega^k(X)$ is called the {\em curvature} of the relative differential character $h$.
\index{curvature}
\index{+Curv@$\curv$, curvature}
The form $\vartheta =: \cov(h) \in \Omega^{k-1}(A)$ is called its {\em covariant derivative}.
\index{covariant derivative}
\index{+Cov@$\cov$, covariant derivative}
The curvature is uniquely determined by the differential character.
For $k\ge2$, this is also true for the covariant derivative.
For $k=1$, the function $\vartheta$ is unique only up to addition of a locally constant integer valued function, see \cite[Ex.~8.3]{BB13}.

We denote by $\Omega^k_0(\varphi)$ the set of all $d_\varphi$-closed forms $(\omega,\vartheta) \in \Omega^k(\varphi)$ with integral periods, i.e., such that $\int_{(s,t)} (\omega,\vartheta) \in \Z$ holds for all $(s,t) \in Z_k(\varphi;\Z)$.
Since $h \in \widehat H^k(\varphi;\Z)$ is a homomorphism, condition \eqref{eq:def_rel_diff_charact_2} implies that 
$$
\int_{(s,t)} (\curv,\cov)(h) \in \Z
$$
for any cycle $(s,t) \in Z_k(\varphi;\Z)$.
Moreover, since 
$$
\int_{\partial_\varphi(a,b)} (\curv,\cov)(h)
=
\int_{(a,b)} d_\varphi(\curv,\cov)(h) \in \Z
$$
holds for all chains $(a,b) \in C_{k+1}(\varphi;\Z)$, it follows that $(\curv,\cov)(h)$ is $d_\varphi$-closed.
Thus $(\curv,\cov)(h) \in \Omega^k_0(\varphi)$.

Differential characters $h \in \widehat H^k(\varphi;\Z)$ with $\curv(h)=0$ are called \emph{flat}, while characters with $\cov(h)=0$ are called \emph{parallel}.
\index{differential character!flat $\sim$}
\index{differential character!parallel $\sim$}
The condition $\varphi^*\curv(h) = d\cov(h)$ implies that parallel characters are in particular \emph{flat along} $\varphi$.
\index{differential character!flat along $\varphi$}

It is shown in \cite[p.~273f.]{BT06} that relative differential characters $h\in\widehat H^k(\varphi;\Z)$ have characteristic classes $c(h)$ in the mapping cone cohomology $H^k(\varphi;\Z)$. 
The class $c(h)$ is defined as follows: 
Let $\tilde h \in C^{k-1}(\varphi;\Z)$ be a real lift of $h$.
Thus $h(s,t) = \exp ( 2\pi i \cdot \tilde h(s,t))$ holds for all cycles $(s,t) \in Z_{k-1}(\varphi;\Z)$.
By \eqref{eq:def_rel_diff_charact_2}, the cocycle 
\begin{equation}\label{eq:cocycle_c(h)}
\mu^{\tilde h}
:= 
(\curv,\cov)(h) - \delta_\varphi \tilde h  
\end{equation}
satisfies $\exp (2\pi i \mu^{\tilde h}(a,b)) =1$ for all $(a,b) \in C_k(\varphi;\Z)$.
Thus $\mu^{\tilde h} \in C^k(\varphi;\Z)$.
The \emph{characteristic class} of $h$ is defined as $c(h):= [\mu^{\tilde h}] \in H^k(\varphi;\Z)$.
\index{characteristic class}
\index{+Ch@$c(h)$, characteristic class}
\index{differential character!characteristic class} 
Characters $h \in \widehat H^k(\varphi;\Z)$ with $c(h)$ are called \emph{topologically trivial}.
\index{differential character!topologically trivial}
\index{topologically trivial}

\subsubsection{Exact sequences.}
By \cite[Thm.~2.4]{BT06}, the group $\widehat H^k(\varphi;\Z)$ fits into the following short exact sequences:
\begin{equation}
\xymatrix@R=5mm{
0 \ar[r] &
\frac{\Omega^{k-1}(\varphi)}{\Omega^{k-1}_0(\varphi)} \ar[rr]^(0.45){\iota_\varphi} &&
\widehat H^k(\varphi;\Z) \ar^c[rr] &&
H^k(\varphi;\Z) \ar[r] \ar[r] &
0 \\
0 \ar[r] &
H^{k-1}(\varphi;\Ul) \ar^j[rr] &&
\widehat H^k(\varphi;\Z) \ar^(0.55){(\curv,\cov)}[rr] &&
\Omega^k_0(\varphi) \ar[r] &
0 
}
\label{eq:rel_sequences}
\end{equation}
The map $j:H^{k-1}(\varphi;\Ul) \to \widehat H^k_\varphi(X,A;\Z)$ is defined by $j(\tilde u)(s,t) := \langle \tilde u,[s,t] \rangle$.
This is well defined and injective, since $H^{k-1}(\varphi;\Ul) \cong \Hom(H_{k-1}(\varphi;\Z),\Ul)$.
The map $\iota_\varphi:\Omega^{k-1}(\varphi) \to \widehat H^k(\varphi;\Z)$ is defined by $\iota(\omega,\vartheta)(s,t) := \exp \Big(2\pi i \int_{(s,t)} (\omega,\vartheta) \Big)$.
From the relative Stokes theorem \eqref{eq:stokes} we conclude $(\curv,\cov) \circ \iota_\varphi = d_\varphi$. 
A form $(\omega,\vartheta) \in \Omega^{k-1}(\varphi)$ such that $\iota_\varphi(\omega,\vartheta) = h$ is called a \emph{topological trivialization} of $h$.
\index{topological trivialization}
\index{differential character!topological trivialization}
\index{+Iotavarphi@$\iota_\varphi$, topological trivialization}
Thus the map $\iota_\varphi:\Omega^{k-1}(\varphi) \to \widehat H^k(\varphi;\Z)$  provides topological trivializations, and the first sequence in \eqref{eq:rel_sequences} tells us that a character $h \in \widehat H^k(\varphi;\Z)$ admits topological trivializations if and only if it is topologically trivial. 
\index{+J@$j$, representation of flat characters}

We denote by
\begin{equation}\label{eq:def_R}
R^k(\varphi;\Z)
:= 
\{(\omega,\vartheta,u) \in \Omega^k_0(\varphi) \times H^k(\varphi;\Z) \,|\, [\omega,\vartheta]_\mathrm{dR} = u_\R \} 
\end{equation}
the set of pairs of $d_\varphi$-closed differential forms with integral periods and integral mapping cone classes that match in the real mapping cone cohomology $H^k(\varphi;\R)$. 
\index{+RvarphiZ@$R^*(\varphi;\Z)$, group of pairs of forms and cohomology classes}
By definition of the characteristic class of a character $h \in \widehat H^k(\varphi;\Z)$ we have $((\curv,\cov)(h),c(h)) \in R^k(\varphi;\Z)$.
Moreover, the exact sequences above may be joined to the exact sequence
\begin{equation}\label{eq:sequ_R}
\xymatrix{
0 \ar[r] 
& 
\frac{H^{k-1}(\varphi;\R)}{H^{k-1}(\varphi;\Z)_\R} \ar[rr] 
&&
\widehat H^k_\varphi(X,A;\Z) \ar^{(\curv,\cov,c)}[rr] 
&& 
R^k(\varphi;\Z) \ar[r]
& 
0 \,.
} 
\end{equation}
Here $H^{k-1}(\varphi;\Z)_\R$ denotes the image of $H^{k-1}(\varphi;\Z)$ in $H^{k-1}(\varphi;\R)$ under the change of coefficients homomorphism induced by $\Z \hookrightarrow \R$.

\subsubsection{Naturality, thin invariance, torsion cycles.}
The following properties are used in several constructions throughout the paper.
 
\begin{remark}[Pull-back of relative differential characters]
Let $\psi:B \to Y$ be another smooth map.
A smooth map $(Y,B) \xrightarrow{(f,g)}(X,A)$ of pairs is a pair of smooth maps such that $\varphi \circ g = f \circ \psi$. 
Thus we have the commutative diagram:
\begin{equation*}
\xymatrix{
B \ar^\psi[rr] \ar_g[dr] && \ar^f[dr] Y & \\
& A \ar_\varphi[rr] && X 
} 
\end{equation*}
We define the pull-back of relative characters along a smooth map $(Y,B) \xrightarrow{(f,g)}(X,A)$ by:
\index{pull-back!of relative differential characters}
\index{differential character!pull-back}
\index{relative differential character!pull-back}
\begin{equation*}
(f,g)^*:\widehat H^k(\varphi;\Z) 
\to 
\widehat H^k(\psi;\Z) \,, \quad 
h 
\mapsto
h \circ (f,g)_* \,.
\end{equation*}
Here $(f,g)_*$ denotes the induced map on relative cycles:
for $(s,t) \in Z_{k-1}(\psi;\Z)$, we have $(f,g)_*(s,t) := (f_*s,g_*t)$ and hence $\big((f,g)^*h\big)(s,t) = h(f_*s,g_*t)$.
\end{remark}

\begin{remark}[Thin invariance]\label{rem:thin_inv}
By definition of thin chains, relative differential characters vanish on boundaries of thin chains of the mapping cone complex.
We term this property the \emph{thin invariance} of differential characters.
\index{thin invariance}
\index{differential character!thin invariance}
\index{relative differential character!thin invariance}
In particular, we have a well-defined evaluation of characters $h \in \widehat H^k(\varphi;\Z)$ on the refined fundamental class $(f,g)_*[M,N]_{\partial_\varphi S_k} \in Z_{k-1}(\varphi;\Z)/\partial_\varphi S_k(\varphi;\Z)$ of a geometric relative cycle $(M,N)\xrightarrow{(f,g)}(X,A)$. 
\end{remark}
 
\begin{remark}[Evaluation on torsion cycles]\label{rem:ev_torsion}
Let $z \in Z_{k-1}(X;\Z)$ be a cycle that represents a torsion class in $H_{k-1}(X;\Z)$ -- a \emph{torsion cycle}, for short.
If $z$ is a boundary then by definition, the evaluation of a differential character $h \in \widehat H^k(X;\Z)$ on $z$ only depends upon $\curv(h)$.
In \cite[Ch.~5]{BB13} we show that the evaluation of $h \in \widehat H^k(X;\Z)$ on a torsion cycle $z$ only depends upon $\curv(h)$ and $c(h)$.

An analogous statement holds for relative characters and mapping cone cycles:
Let $h \in \widehat H^k(\varphi;\Z)$ and let $\tilde h \in C^{k-1}(\varphi;\R)$ be a real lift as in the definition of the characteristic class.  
Suppose that $(s,t) \in Z_{k-1}(\varphi;\Z)$ is a torsion cycle, i.e.~it represents a torsion class in $H_{k-1}(\varphi;\Z)$.
Thus we find an integer $N \in \N$ and a chain $(a,b) \in C_k(\varphi;\Z)$ such that $N \cdot (s,t) = \partial_\varphi (a,b)$.
Then we have:
\index{differential character!evaluation on torsion cycle}
\index{relative differential character!evaluation on torsion cycle}
\begin{align}
h(s,t)
&=
\exp \Big( 2\pi i \cdot \tilde h \Big(\frac{1}{N} \partial_\varphi (a,b) \Big) \Big) \notag \\
&=
\exp \Big( \frac{2\pi i}{N} \cdot \big(\delta_\varphi \tilde h \big)(a,b) \Big) \notag \\
&=
\exp \Big( \frac{2\pi i}{N} \cdot \big( (\curv,\cov)(h) - \mu^{\tilde h} \big)(a,b) \Big) \notag \\ 
&=
\exp \Big( \frac{2\pi i}{N} \cdot \Big( \int_{(a,b)}(\curv,\cov)(h) - \langle c(h),(a,b) \rangle \Big)\Big) \,. \label{eq:ev_torsion}
\end{align}
Note that the evaluation of the characteristic class $c(h)$ on the chain $(a,b)$ is not well-defined.
But the term in \eqref{eq:ev_torsion} is well-defined: any two cocycles that represent $c(h)$ differ by an integral coboundary $\delta_\varphi \ell$ for $\ell \in C^{k-1}(\varphi;\Z)$ and $\frac{1}{N} \langle \delta_\varphi \ell,(a,b) \rangle = \langle \ell,(s,t) \rangle \in  \Z$.  
\end{remark}

\subsubsection{Absolute differential characters.}
Let $x \in X$ be any point.
We may consider $x$ as a smooth map \mbox{$\varphi=x: \{*\} \to X$}, $* \mapsto x$.
We have the canonical identification $C_k(\{*\};\Z) \cong \Z$ for $k \geq 0$. 
The boundary map $\partial:C_k(\{*\};\Z) \to C_{k-1}(\{*\};\Z)$ is the identity for positive even $k$ and identically $0$ else.
For $k \geq 2$, we obtain canonical identifications
$$
Z_{k-1}(x;\Z) 
\cong
Z_{k-1}(X;\Z) \oplus Z_{k-2}(\{*\};\Z) 
\cong
\begin{cases}
 Z_{k-1}(X;\Z) & \mbox{\quad if $k$ even} \\
 Z_{k-1}(X;\Z) \oplus \Z & \mbox{\quad if $k$ odd.} \\
\end{cases}
$$ 
For relative differential forms, we have $\Omega^k_0(x) = \Omega^k_0(X) \times \{0\}$ for any $k \geq 2$.
 
\begin{remark}[Absolute differential characters]
Let $k \geq 2$.
Let $h:Z_{k-1}(x;\Z) \to \Ul$ be a relative differential character. 
Then we have $\cov(h)=0$ for dimensional reasons.
In particular, $\curv(h) \in \Omega^k_0(X)$. 
For even $k$, the relative character $h$ is a homomorphism $h:Z_{k-1}(X;\Z) \to \Ul$.
For odd $k$, condition \eqref{eq:def_rel_diff_charact_2} implies that the homomorphism $h:Z_{k-1}(X;\Z) \oplus \Z \to \Ul$ vanishes on the second factor, since any $(0,t) \in Z_{k-1}(x;\Z)$ is a boundary. 
Thus $h$ induces a homomorphism $h:Z_{k-1}(X;\Z) \to \Ul$ that satisfies $h(\partial a) = \exp( 2\pi i \int_a \curv(h))$.

We thus obtain a canonical identification of $\widehat H^k(x;\Z)$ with the group 
\begin{equation}
\widehat H^k(X;\Z)
:=  \big\{\, h \in \Hom(Z_{k-1}(X;\Z),\Ul) \,\big|\, h \circ \partial \in \Omega^k(X) \,\big\} 
\end{equation}
of dabsolute differential characters on $X$, as defined in \cite{CS83}.
The group $\widehat H^k(X;\Z)$ fits into short exact sequences analogous to the sequences in \eqref{eq:rel_sequences} with mapping cone cohomology groups replaced by the corresponding absolute cohomology groups, and similarly for the spaces of differential forms.
\end{remark}

\subsubsection{Long exact sequence.}
Pre-composition with the maps $i$ and $p$ in the exact sequence \eqref{eq:ex_seq_Z} induces homomorphisms $\vds_\varphi$ and $\ti_\varphi$ on differential characters groups
\begin{equation}\label{eq:natural_maps_rel_diff_charact}
\xymatrix{
\widehat H^{k-1}(A;\Z) 
\ar^{\ti_\varphi}[r] &
\widehat H^k_\varphi(X,A;\Z)
\ar^{\vds_\varphi}[r] &
\widehat H^k(X;\Z) \,.
}
\end{equation}
Thus for a character $h \in \widehat H^{k-1}(A;\Z)$ and a relative cycle $(s,t) \in Z_{k-1}(\varphi;\Z)$, we have $\ti_\varphi(h)(s,t) := h(t)$.
Likewise for a relative character $h \in \widehat H^k(\varphi;\Z)$ and a cycle $z \in Z_{k-1}(X;\Z)$, we have $\vds_\varphi(h)(z):= h(z,0)$.
One easily checks that $\curv \circ \ti_\varphi \equiv 0$ whereas $\cov \circ \ti_\varphi \equiv -\curv$ and $\curv \circ \vds_\varphi \equiv \curv$.

Let $\psi:B \to Y$ be another smooth map.
The homomorphisms $\ti$ and $\vds$ are natural with respect pull-back along smooth maps $(Y,B)\xrightarrow{(f,g)}(X,A)$:
For a character $h \in \widehat H^{k-1}(A;\Z)$ and a relative cycle $(s,t) \in Z_{k-1}(\psi;\Z)$ we have 
$
\big((f,g)^*\ti_\varphi(h)\big)(s,t) 
= 
\ti_\varphi(h)((f,g)_*(s,t)) 
= 
h(g_*t)
=
g^*h(t)
$
and hence
\begin{equation}\label{eq:ti_nat}
(f,g)^*\ti_\varphi(h) 
=  
\ti_\psi(g^*h) \,.
\end{equation}
Similarly, for a relative character $h \in \widehat H^k(\varphi;\Z)$ and a cycle $z \in Z_{k-1}(Y;\Z)$ we have 
$
\vds_\psi((f,g)^*h)(z)
=
(f,g)^*h(z,0)
=
h(f_*z,0)
=
\vds_\varphi(h)(f_*z)
=
f^*(\vds_\varphi(h))(z)
$
and hence
\begin{equation}\label{eq:vds_nat}
\vds_\psi((f,g)^*h) = f^*\vds_\varphi(h) \,.
\end{equation}

In \cite[Ch.~8]{BB13} we show that for $k \geq 2$ the absolute and relative differential characters groups fit into the following long exact sequence:
\index{long exact sequence!for differential characters groups}
\index{differential character!long exact sequence}
\index{relative differential character!long exact sequence} 
\begin{equation}
\xymatrix{
\ldots \ar[r]
& H^{k-3}(A;\Ul) \ar[r] & H^{k-2}(\varphi;\Ul) \ar[r] & H^{k-2}(X;\Ul) \ar`[r]`d[lll]`[llld]`[llldr][lld]^(0.3){j \circ \varphi^*} &  \\  
& \widehat H^{k-1}(A;\Z) \ar^\ti[r] & \widehat H^k_\varphi(X,A;\Z) \ar^\vds[r] & \widehat H^k(X;\Z) \ar`[r]`d[lll]`[llld]`[llldr][lld]^(0.3){\varphi^* \circ c} & \\
& H^k(A;\Z) \ar[r] & H^{k+1}(\varphi;\Z) \ar[r] & H^{k+1}(X;\Z) \ar[r] &  \ldots
} 
\label{eq:long_ex_sequ}
\end{equation}
The sequence proceeds as the long exact sequence for singular cohomology with $\Ul$-coefficients on the left and with integer coefficients on the right.

\subsection{Sections and topological trivializations.}
Let $h \in \widehat H^k(X;\Z)$ be a differential character and $\varphi:A \to X$ a smooth map.
As in \cite[Ch.~8]{BB13} we say that $h$ \emph{admits sections} along $\varphi$ if $h$ lies in the image of the map $\vds:H^k(\varphi;\Z) \to H^k(X;\Z)$.
Any preimage $\vds^{-1}(h)$ of $h$ is called a \emph{section} of $h$ along $\varphi$.
\index{section}
\index{section!along $\varphi$}
\index{differential character!section along $\varphi$}
From the exact sequence \eqref{eq:long_ex_sequ} we conclude that $h$ admits setions along $\varphi$ if and only if $\varphi^*c(h) =0$, i.e.~if and only if $h$ is topologically trivial along $\varphi$.
\index{differential character!topologically trivial along $\varphi$}

\subsubsection{Sections and covariant derivative}
We discuss the role that sections and their covariant derivatives play for topological trivializations. 
We briefly recall the following basic example from \cite{BB13}.

\begin{example}\label{ex:H2}
It is well-known that the group $\widehat H^2(X;\Z)$ is canonically isomorphic to the group of isomorphism classes of hermitean line bundles with connection (under connection preserving isomorphisms).
A differential character $h \in \widehat H^k(X;\Z)$ corresponds to the holonomy map of a bundle $(L,\nabla)$ under this isomorphism.
Holonomy is invariant under connection preserving isomorphisms.
Moreover, the characteristic class $c(h) \in \widehat H^2(X;\Z)$ coincides with the first Chern class of the bundle.
For the curvature form we have $\curv(h)= \frac{i}{2\pi} \cdot R^\nabla$, where $R^\nabla \in \Omega^2(X;i\R)$ is the curvature $2$-form of the connection $\nabla$. 
The image of a differential form $\omega \in \Omega^1(X)$ under the map $\iota:\Omega^1(X) \to \widehat H^2(X;\Z)$ corresponds to a topologically trivial line bundle with connection $1$-form $\omega$.
Hence the name \emph{topological trivialization} for the map $\iota$.
 
In \cite[Ch.~8]{BB13} we have shown that the group $\widehat H^2(\varphi;\Z)$ is isomorphic to the group of isomorphism classes of hermitean line bundles with connection $(L,\nabla)$ and section $\sigma:A \to \varphi^*L$ along $\varphi$.
The isomorphisms are bundle isomorphisms of $L$ that preserve both the connection and the section.  
The map $\vds$ corresponds to the forgetful map that ignores the section.
The covariant derivative of the character is related to the covariant derivative of the section by $\nabla \sigma = \cov(h) \cdot \sigma \in \Gamma(T^*A \otimes \varphi^*L)$.

Thus for a differential character $h \in \widehat H^2(X;\Z)$, any preimage $\vds^{-1}(h) \in  \widehat H^2(\varphi;\Z)$ corresponds to an isomorphism class of sections along the map $\varphi$.
Hence the name.
\end{example}

By the exact sequence \eqref{eq:long_ex_sequ}, a differential character admits sections along a smooth map if and only if it is topologically trivial along $\varphi$.
Consequently, the character $\varphi^*h$ is topologically trivial.
We show that it is trivialized by the covariant derivative of any section of $h$.
The special case $\varphi = \id_X$ was discussed in \cite[Ch.~8]{BB13}.

\begin{prop}[Topological trivialization and covariant derivative]
Let $\varphi:A \to X$ be a smooth map.
Then we have the following commutative diagram:
\begin{equation}\label{eq:diag:triv_cov}
\xymatrix{
\widehat H^k(\varphi;\Z) \ar_\cov[d] \ar^{\vds_\varphi}[r] & \widehat H^k(X;\Z) \ar^{\varphi^*}[d] \\
\Omega^{k-1}(A) \ar_\iota[r] & \widehat H^k(A;\Z) \,. 
} 
\end{equation}
Thus covariant derivatives of sections along a smooth map yield topological trivializations of the pulled back characters. 
\end{prop}
\index{Proposition!Topological trivialization and covariant derivative}

\begin{proof}
Let $h \in \widehat H^k(\varphi;\Z)$ be a relative character and $z \in Z_{k-1}(A;\Z)$ a cycle.
Then we have:
\begin{align}
(\varphi^*\vds(h))(z)
&=
(\vds(h))(\varphi_*z) \notag \\
&=
h(\varphi_*z,0) \notag \\
&=
h(\partial_\varphi (0,z)) \notag \\
&=
\exp \Big( 2\pi i \int_{(0,z)} (\curv,\cov)(h) \Big) \notag \\
&=
\exp \Big( 2\pi i \int_z \cov(h) \Big) \notag \\
&=
\iota(\cov(h))(z) \,. \qedhere \label{eq:triv_cov}
\end{align} 
\end{proof}

\subsubsection{The Cheeger-Chern-Simons construction}
A particular example of relative differential characters as sections of absolute characters along a smooth map arises by the differential character valued refinement of the Chern-Weil construction, due to Cheeger and Simons:

\begin{example}\label{ex:CCS}
Let $G$ be a compact Lie group with Lie algebra $\g$.
An invariant polynomial, homogeneous of degree $k$, is a symmetric $\Ad_G$-invariant multilinear map $q:\g^{\otimes k}\to \R$.
The Chern-Weil construction associates to any principal $G$-bundle with connection $(P,\nabla) \to X$ a closed differential form $CW(q) = q(R^\nabla) \in \Omega^{2k}(X)$ by applying the polynomial $q$ to the curvature $2$-form $R^\nabla$ of the connection $\nabla$.
Consider those polynomials $q$ for which the Chern-Weil form $CW(q)$ has integral periods. 
Let $u \in H^{2k}(X;\Z)$ be a universal characteristic class for principal $G$-bundles that coincides in $H^{2k}(X;\R)$ with the de Rham class of $CW(q)$. 
The \emph{Cheeger-Simons construction}~\cite[Thm~2.2]{CS83} associates to this setting a differential character $\widehat{CW}(q,u) \in \widehat H^{2k}(X;\Z)$ with curvature $\curv(\widehat{CW}(q,u)) = CW(q)$, the Chern-Weil form, and characteristic class $c(\widehat{CW}(q,u))=u$, the fixed universal characteristic class.
\index{Cheeger-Simons construction}
The construction is natural with respect to bundle maps.

Since the total space $EG$ of the universal principal $G$-bundle is contractible, universal characteristic classes vanish upon pull-back to the total space.
By the long exact sequence~\eqref{eq:long_ex_sequ} the Cheeger-Simons character $\widehat{CW}(q,u)$ thus admits sections along the bundle projection $\pi:P \to X$.
The so-called \emph{Cheeger-Chern-Simons construction} of \cite{B13} yields a canonical section $\widehat{CCS}(q,u) \in \widehat H^{2k}(\pi;\Z)$ with covariant derivative $\cov(\widehat{CCS}(q,u)) = CS(q) \in \Omega^{2k-1}(P)$, the Chern-Simons form.
\index{Cheeger-Chern-Simons construction}
The construction is natural with respect to bundle maps.

Thus the Cheeger-Chern-Simons construction is a relative differential character valued refinement of the Chern-Weil and Chern-Simons constructions in the same way as the Cheeger-Simons construction is a differential character valued refinement of the Chern-Weil construction alone.    
\end{example}

\subsubsection{Parallel sections.}
In general, the property for a given character to admit sections with prescribed covariant derivatives depends on the character.
For example, a hermitean line bundle with connection $(L,\nabla)$ and with sections along a smooth map $\varphi:A \to X$ admits \emph{parallel} sections if and only if the pull-back $\varphi^*(L,\nabla)$ is isomorphic to the trivial bundle with trivial connection.

The analogous statement holds for any differential characters, as we shall prove next:

\begin{thm}[Parallel sections]\label{thm:par_sec}
A differential character $h \in \widehat H^k(X;\Z)$ admits parallel sections along a smooth map $\varphi:A \to X$ if and only if $\varphi^*h=0$.
\end{thm}
\index{Theorem!Parallel sections}
\index{section!parallel $\sim$}
\index{differential character!parallel section}

\begin{proof}
Let $h \in \widehat H^k(X;\Z)$ with $\varphi^*h=0$.
Then in particular $h$ is topologically trivial along $\varphi$ and hence admits sections along $\varphi$.
By the commutative diagram \eqref{eq:diag:triv_cov} and the exact sequence \eqref{eq:rel_sequences}, the covariant derivative of any such section $h' \in \widehat H^k(\varphi;\Z)$ satisfies $\cov(h') \in \Omega^{k-1}_0(A)$, i.e.~it is closed with integral periods.  
Choose a character $h'' \in \widehat H^{k-1}(A;\Z)$ with $\curv(h'') = \cov(h')$.
Now put $h''':= h' + \ti(h'')$.
Then we have $\vds(h''') = \vds(h') = h$ and $\cov(h''') = \cov(h') - \curv(h'') =0$.
Thus $h'''$ is a parallel section of $h$.

Conversely let $h' \in \widehat H^k(\varphi;\Z)$ be a parallel section of $h \in \widehat H^k(X;\Z)$.
By the commutative diagram \eqref{eq:diag:triv_cov} we find $\varphi^*h = \iota(\cov(h')) = 0$.  
\end{proof}

\subsection{Parallel characters}\label{subsec:rel_diff_par}
Throughout this section let $i_A:A \to X$ be the embedding of a smooth submanifold.
As explained in the introduction, there is another notion of relative differential cohomology, based on homomorphisms on the group of relative cycles $Z_{*-1}(X,A;\Z)$. 
This notion has appeared in \cite{HL01} for the special case where $A = \partial X$ is the boundary of $X$.

In this section, we introduce differential characters on the group $Z_{*-1}(X,A;\Z)$ of relative cycles, where $A \subset X$ is an arbitrary submanifold.
We denote the corresponding group of differential characters by $\widehat H^*(X,A;\Z)$. 
We prove a long exact sequence that relates this group to differential characters groups on $X$ and $A$.
An analogous sequence has appeared in \cite{U11} for generalized differential cohomology. 
Further we show that $\widehat H^k(X,A;\Z)$ is in 1-1 correspondence with the subgroup of parallel characters in $\widehat H^*(i_A;\Z)$.

\subsubsection{Characters on relative cycles.}
Let $X$ be a smooth manifold and $A \subset X$ an embedded smooth submanifold.
Denote the embedding by $i_A:A \to X$.
Let $Z_*(X,A;\Z)$ be the group of relative cycles, i.e.~cycles in the quotient complex $C_*(X,A;\Z) := C_*(X;\Z) / C_*(A;\Z)$.
\index{+CXAZ@$C_*(X,A;\Z)$, group of relative chains}
\index{+ZXAZ@$Z_*(X,A;\Z)$, group of relative cycles}

We put
\begin{equation*}
\widehat H^k(X,A;\Z) 
:= 
\big\{\, h \in \Hom(Z_{k-1}(X,A;\Z),\Ul) \,\big|\, f \circ \partial \in \Omega^k(X) \,\big\} \,.  
\end{equation*}
The notation $f \circ \partial_\phi \in \Omega^k(X)$ means that there exists a differential form $\omega \in \Omega^k(X)$ such that for every relative chain $x \in C_k(X,A;\Z)$ we have
\begin{equation}
f(\partial c)
= 
\exp \Big( 2 \pi i \int_c \omega \Big) \,. \label{eq:def_rel_diff_charact_2a}
\end{equation}
The condition in particular implies that the integral of $\omega$ over chains $c \in C_k(X,A;\Z)$ is well-defined.
Hence $i_A^*\omega \equiv 0$.   
\index{+HhatXAZ@$\widehat H^*(X,A;\Z)$ group of differential characters on relative cycles}
\index{differential character!on relative cycles}
\index{relative differential character!on relative cycles}

Since condition \eqref{eq:def_rel_diff_charact_2a} holds for all chains $c \in C_k(X,A;\Z)$, the differential form $\omega$ is uniquely determined.
We call it the \emph{curvature} of $h$ and denote it by $\curv(h)$.
\index{curvature}
\index{relative differential character!curvature}
\index{+Curv@$\curv$, curvature}
A character $h \in \widehat H^k(X,A;\Z)$ with $\curv(h)=0$ is called \emph{flat}.
\index{flat}
\index{relative differential character!flat $\sim$}
 
Let $\Omega^k(X,A):= \{ \omega \in \Omega^k(X) \,|\, i_A^*\omega \equiv 0 \}$ be the space of $k$-forms on $X$ relative $A$. 
As we have seen, a character $h \in \widehat H^k(X,A;\Z)$ has curvature $\curv(h) \in \Omega^k(X,A)$.
\index{+OmegaXA@$\Omega^*(X,A)$, relative de Rham complex}
\index{relative de Rham complex}
We define the relative de Rham cohomology $H^*_\mathrm{dR}(X,A)$ as the cohomology of the de Rham subcomplex $(\Omega^*(X,A),d) \subset (\Omega^*(X),d)$.
\index{relative de Rham comomology}
\index{+HdRXA@$H^*_\mathrm{dR}(X,A)$, relative de Rham cohomology}
\index{cohomology!relative de Rahm $\sim$}
The short exact sequence
\begin{equation*}
\xymatrix{
0 \ar[r] & \Omega^k(X,A) \ar[r] & \Omega^k(X) \ar^{i_A^*}[r] & \Omega^k(A) \ar[r] & 0
}
\end{equation*}
gives rise to a long exact sequence relating absolute and relative de Rham cohomology groups.
  
Integration of differential forms over smooth singular chains in $X$ yields a well-defined homomorphism
$\Omega^k(X,A) \to C^k(X,A;\R)$. 
By the de Rham theorem and the five lemma, applied to the long exact sequences, this induces a canonical isomorphism $H^*_\mathrm{dR}(X,A) \xrightarrow{\cong} H^*(X,A;\R)$. 
Denote by $\Omega^k_0(X,A)$ the subgroup of differential forms $\omega \in \Omega^k(X,A)$ with integral periods, i.e.~such that $\int_y \omega \in \Z$ holds for any $y \in Z_k(X,A;\Z)$.
Then we have $\Omega^k_0(X,A) = \{\omega \in \Omega^k_0(X) \,|\, i_A^*\omega \equiv 0 \}$.

Since $h$ is a homomorphism, condition \eqref{eq:def_rel_diff_charact_2a} implies that $\int_z \curv(h) \in \Z$ holds for any cycle $z \in Z_k(X;\Z)$.
Stokes theorem implies that $\int_c d\curv(h) = \int_{\partial c} \curv(h) \in \Z$ holds for any chain $c \in C_{k+1}(X;\Z)$, hence $\curv(h)$ is closed.
Thus $\curv(h) \in \Omega^k_0(X,A)$.

The \emph{characteristic class} $c(h) \in H^k(X,A;\Z)$ of a character $h \in \widehat H^k(X,A;\Z)$ is defined as follows:
The curvature defines a real cocycle $\curv(h): C_k(X,A;\R) \to \R$, $c \mapsto \int_c \curv(h)$.
\index{characteristic class}
\index{+Ch@$c(h)$, characteristic class}
\index{relative differential character!characteristic class}
Choose a real lift $\tilde h \in C^{k-1}(X,A;\R)$ of $h$, i.e.~$h(z) = \exp \Big( 2\pi i \tilde h(z) \Big)$.
Put $\mu^{\tilde h} := \curv(h) - \delta \tilde h \in C^k(X,A;\Z)$.
In fact, $\mu^{\tilde h}$ is an integral cochain because of \eqref{eq:def_rel_diff_charact_2a}.
Since the curvature is a closed form, $\mu^{\tilde h}$ is a cocycle.
Now define $c(h) := [\mu^{\tilde h}] \in H^k(X,A;\Z)$.
It is easy to see that $c(h)$ does not depend upon the choice of real lift $\tilde h$:
The difference between two choices of real lifts is an integral cochain. 
Thus the cocycles for two choices of real lifts differ by an integral coboundary.
A character $h \in \widehat H^k(X,A;\Z)$ with $c(h)=0$ is called \emph{topologically trivial}. 
\index{relative differential character!topologically trivial}
\index{topologically trivial}
\index{differential character!topologically trivial}

\subsubsection{Exact sequences.}
We have a natural map $\iota:\Omega^{k-1}(X,A) \to \widehat H^k(X,A;\Z)$, defined by $\iota(\vartheta)(z) := \exp \Big( 2\pi i \int_z \vartheta \Big)$.
By the Stokes theorem, the induced character $\iota(\vartheta)$ satisfies $\curv(\iota(\vartheta)) = d\vartheta$.
The map $\iota$ descends to an injective map $\iota: \frac{\Omega^{k-1}(X,A)}{\Omega^{k-1}_0(X,A)} \to \widehat H^k(X,A;\Z)$. 
A form $\omega \in \Omega^{k-1}(X,A)$ such that $\iota(\omega)=h$ is called a \emph{topological trivialization} of $h$.
\index{+Iota@$\iota$, topological trivialization}
\index{topological trivialization}
\index{relative differential character!topological trivialization}
\index{differential character!topological trivialization}

Finally, we have an obvious injection $j: H^{k-1}(X,A;\Ul) \to \widehat H^k(X,A;\Z)$, defined by $j(u)(z) := \langle u,[z] \rangle$. 
\index{+J@$j$, representation of flat characters}

The above maps fit into the following exact sequences:
\begin{equation}\label{eq:rel_sequ_par}
\xymatrix{
0 \ar[r] & H^{k-1}(X,A;\Ul) \ar^j[r] & \widehat H^k(X,A;\Z) \ar^{\curv}[r] & \Omega^k_0(X,A) \ar[r] & 0 \\
0 \ar[r] & \frac{\Omega^{k-1}(X,A)}{\Omega^{k-1}_0(X,A)} \ar^\iota[r] & \widehat H^k(X,A;\Z) \ar^c[r] & H^k(X,A;\Z) \ar[r] & 0 \,. 
} 
\end{equation}
Exactness of the curvature sequence at $\widehat H^k(X,A;\Z)$ is clear, since by \eqref{eq:def_rel_diff_charact_2a} flat characters are precisely those that vanish on boundaries and hence descend to homomorphisms on $H_{k-1}(X;\Z)$.
Surjectivity of the curvature follows e.g.~from surjectivity of the curvature map $\curv: \widehat H^k(X;\Z) \to \Omega^k_0(X)$ and the exact sequence \eqref{eq:long_ex_sequ_par} below.

Exactness of the characteristic class sequence at $\widehat H^k(X,A;\Z)$ follows from \eqref{eq:def_rel_diff_charact_2a} and the definition of the characteristic class.
Surjectivity of the curvature follows e.g.~from surjectivity of the characteristic class map $c: \widehat H^k(X;\Z) \to H^k(X;\Z)$ and the exact sequence \eqref{eq:long_ex_sequ_par} below. 

The second sequence in \eqref{eq:rel_sequ_par} tells us that a character $h \in \widehat H^k(X,A;\Z)$ admits a topological trivilization if and only if it is topologically trivial.

Put $R^k(X,A;\Z) := \{(\omega,u)\in \Omega^k_0(X,A) \times H^k(X,A;\Z) \,|\, u_\R = [\omega]_{dR} \in H^k(X,A;\R) \}$.
\index{+RXAZ@$R^*(X,A;\Z)$, group of pairs of forms and cohomology classes}
Then the two sequences above may be joined to give the following exact sequence:
\begin{equation}
\xymatrix{
0 \ar[r] & \frac{H^{k-1}(X,A;\R)}{H^{k-1}(X,A;\R)_\Z}  \ar^j[r] & \widehat H^k(X,A;\Z) \ar^{(\curv,c)}[rr] && R^k(X,A;\Z) \ar[r] & 0 \,.
} \label{eq:R-sequence_X_A}
\end{equation}  

\subsubsection{Absolute differential characters.}
Let $x \in X$ an arbitrary point.
We write $x$ instead of $\{x\} \subset X$.
For positive even $k$, we have $Z_{k-1}(x;\Z) = \Z = B_{k-1}(x;\Z)$, while for odd $k$, we have $Z_{k-1}(x;\Z) =\{0\}$. 
Since differential forms of degree $k \geq 1$ vanish upon pull-back to $x$, we obtain a canonical identification
$$
\widehat H^k(X,x;\Z) \xrightarrow{\cong} \widehat H^k(X,\Z) \,.
$$

\subsubsection{Long exact sequence}
Pre-composition of a differential character on relative cycles \mbox{$h:Z_{k-1}(X,A;\Z) \to \Ul$} with the quotient map $Z_{k-1}(X;\Z) \to Z_{k-1}(X,A;\Z)$, \mbox{$z \mapsto z + \im({i_A}_*)$}, yields a homomorphism \mbox{$h':Z_{k-1}(X;\Z) \to \Ul$}.
This homomorphism is in fact a differential character in $\widehat H^k(X;\Z)$, since for any chain $c \in C_k(X;\Z)$, we have 
\begin{equation*}
h'(\partial c)
:=
h(\partial c + \im({i_A}_*))
\stackrel{\eqref{eq:def_rel_diff_charact_2a}}{=}
\exp \Big( 2 \pi i \int_c \curv(h) \Big) \,. 
\end{equation*}
Hence $\curv(h') = \curv(h) \in \Omega^k_0(X,A) \subset \Omega^k_0(X)$.
We thus obtain a homomorphism $\widehat H^k(X,A;\Z) \to \widehat H^k(X;\Z)$ that preserves the curvature.
Moreover, a real lift for $h \in \widehat H^k(X,A;\Z)$ also defines a real lift of its image in $\widehat H^k(X;\Z)$. 
Thus the homomorphism is also compatible with the characteristic class and we obtain the following commutative diagram:  
\begin{equation*}
\xymatrix{
\widehat H^k(X,A;\Z) \ar_c[d] \ar[r] & \widehat H^k(X;\Z) \ar^c[d] \\
H^k(X,A;\Z) \ar[r] & H^k(X;\Z) \,,
} 
\end{equation*}
Here the lower horizontal map is the usual map in the long exact sequence for absolute and relative cohomology.

We denote the connecting homomorphism in the long exact sequence for relative and absolute cohomology by $\beta: H^*(A;\Z) \to H^{*+1}(X,A;\Z)$ (and likewise for $\Ul$ coefficients). 
Concatenation with $j$ yields a map $:H^{k-2}(A;\Ul) \xrightarrow{\beta} H^{k-1}(X,A;\Ul) \xrightarrow{j} \widehat H^k(X,A;\Z)$.
Likewise, we obtain a map $\widehat H^k(A;\Z) \xrightarrow{c} H^k(A;\Z) \xrightarrow{\beta} H^{k+1}(A;\Z)$.
These maps fit into the following long exact sequence:

\index{long exact sequence!for differential characters groups}
\index{differential character!long exact sequence}
\index{relative differential character!long exact sequence}
\begin{thm}[Long exact sequence]
Let $i_A:A \hookrightarrow X$ be the embedding of a smooth submanifold.
Let $k \geq 1$.
Then we have the following long exact sequence for the groups of differential characters:
\begin{equation}
\xymatrix{
\ldots \ar[r]
& H^{k-2}(X,A;\Ul) \ar[r] & H^{k-2}(X;\Ul) \ar[r] & H^{k-2}(A;\Ul) \ar`[r]`d[lll]`[llld]`[llldr][lld]^(0.3){j \circ \beta} &  \\  
& \widehat H^k(X,A;\Z) \ar[r] & \widehat H^k(X;\Z) \ar^{i_A^*}[r] & \widehat H^k(A;\Z) \ar`[r]`d[lll]`[llld]`[llldr][lld]^(0.3){\beta \circ c} & \\
& H^{k+1}(X,A;\Z) \ar[r] & H^{k+1}(X;\Z) \ar[r] & H^{k+1}(A;\Z) \ar[r] &  \ldots
} 
\label{eq:long_ex_sequ_par}
\end{equation}
The sequence proceeds as the long exact sequence for singular cohomology with $\Ul$-coefficients on the left and with integer coefficients on the right.
\end{thm}

\begin{remark}
Exactness at $\widehat H^k(X;\Z)$ also follows from Theorem~\ref{thm:par_sec} above and Theorem~\ref{thm:rel_rel} below.  
\end{remark}

\begin{proof}
Exactness at the first two and the last two groups is of course well-known. 
We give a direct proof of the exactness at the remaining groups:

a)
Exactness of the sequence \eqref{eq:long_ex_sequ_par} at $H^{k-2}(A;\Ul)$ follows from exactness of the sequence $H^{k-2}(X;\Ul) \xrightarrow{i_A^*} H^{k-2}(A,\Ul) \xrightarrow{\beta} H^{k-1}(X,A;\Ul)$ and injectivity of the map $H^{k-1}(X,A;\Ul) \xrightarrow{j} \widehat H^k(X,A;\Z)$.

b)
We prove exactness at $\widehat H^k(X,A;\Z)$:
Let $u \in H^{k-2}(A;\Ul)$ and $z \in Z_{k-1}(X;\Z)$.
Since $\Ul$ is divisible, we have $H^{k-2}(A;\Ul) = \Hom(H_{k-2}(A;\Z),\Ul)$. 
The connecting homomorphism $\beta:H^{k-2}(A;\Ul) \to H^{k-1}(X,A;\Ul)$ is dual to the connecting homomorphism $\beta:H_{k-1}(X,A;\Z) \to H_{k-2}(A;\Z)$.
This yields:
$$
(j\circ\beta(u))(z)
=
(j \circ \beta(u))(z + \im ({i_A}_*)) 
=
\langle \beta(u) , [z + \im ({i_A}_*)] \rangle 
=
\langle u , \underbrace{\beta([z + \im ({i_A}_*)])}_{=0} \rangle 
=
1 \,.
$$
Here we use that $H_{k-1}(X,\Z) \to H_{k-1}(X,A;\Z) \xrightarrow{\beta} H_{k-2}(A;\Z)$, $[z] \mapsto \beta([z+\im({i_A}_*)])$, is the trivial map.  

Conversely, let $h \in \widehat H^k(X,A;\Z)$ such that the induced character in $\widehat H^k(X;\Z)$ vanishes.
In particular, we have $\curv(h) = 0$.
By the exact sequence \eqref{eq:rel_sequ_par} we find $\tilde u \in H^{k-1}(X,A;\Ul)$ such that $h = j(\tilde u)$.
By assumption, $h$ vanishes on cycles in $X$, hence $\tilde u$ lies in the kernel of the map $H^{k-1}(X,A;\Ul) \to H^k(X;\Ul)$.
Thus we find $u \in H^{k-2}(A;\Ul)$ such that $\tilde u = \beta(u)$ and hence $h = j \circ \beta(u)$.  

c)
We show exactness at $\widehat H^k(X;\Z)$:
The map $\widehat H^k(X,A;\Z) \to \widehat H^k(X;\Z) \xrightarrow{i_A^*} \widehat H^k(A;\Z)$ is trivial, since cycles in $Z_{k-1}(A;\Z)$ represent $0$ in $Z_{k-1}(X,A;\Z)$.
 
Conversely, let $h \in \widehat H^k(X;\Z)$ with $i_A^*h = 0$.
Let $\sigma: B_{k-2}(X;\Z) \to C_{k-1}(X;\Z)$ be a splitting of the exact sequence 
$$
0
\to
Z_{k-1}(X;\Z) 
\to
C_{k-1}(X;\Z)
\to
B_{k-2}(X;\Z)
\to
0 \,.
$$ 
Let $V:= \sigma(B_{k-2}(X;\Z)) \cap \{c \in C_{k-1}(X;\Z)  \,|\, \partial c \in \im({i_A}_*) \} \subset C_{k-1}(X;\Z)$.
Since $V$ is a submodule of $C_{k-1}(X;\Z)$, it is a free $\Z$-module, and we have the splittings 
$$
\{c \in C_{k-1}(X;\Z)  \,|\, \partial c \in \im({i_A}_*) \}
= Z_{k-1}(X;\Z) \oplus V
$$
and 
$$
Z_{k-1}(X,A;\Z) 
=
\frac{Z_{k-1}(X;\Z)}{\im({i_A}_*)} \oplus \frac{V}{\im({i_A}_*)} \,. 
$$
By assumption the character $h:Z_{k-1}(X;\Z) \to \Ul$ vanishes upon pull-back to $A$.
Hence it descends to a homomorphism $\overline{h}$ on the first factor. 
The above splitting allows us to extend $\overline{h}$ to a homomorphism $\overline{h}:Z_{k-1}(X,A;\Z) \to \Ul$.
By construction, any such extension $\overline{h}$ is a differential character in $\widehat H^k(X,A;\Z)$ which maps to $h$ under the map $\widehat H^k(X,A;\Z) \to \widehat H^k(X;\Z)$.  

d)
We show exactness at $\widehat H^k(A;\Z)$:
Since the characteristic class is natural with respect to pull-back, we have $\beta \circ c \circ i_A^* = (\beta \circ i_A^*) \circ c = 0$.

Conversely, let $h \in \widehat H^k(A;\Z)$ such that $\beta \circ c(h) =0$.
By exactness of the sequence $H^k(X;\Z) \xrightarrow{i_A^*} H^k(A;\Z) \xrightarrow{\beta} H^{k+1}(X,A;\Z)$ and surjectivity of the characteristic class we find a character $h' \in \widehat H^k(X;\Z)$ such that $i_A^*h' - h = \iota(\vartheta)$ for some differential form $\vartheta \in \Omega^{k-1}(A)$.
Choose a differential form $\vartheta' \in  \Omega^{k-1}(X)$ such that $\vartheta = i_A^*\vartheta'$.
Now put $h'' := h' + \iota(\varrho') \in \widehat H^k(X;\Z)$.
Then we have
$i_A^*h'' = i_A^*h' + \iota(\varrho) = h$. 

e) 
Finally, exactness at $H^{k+1}(X,A;\Z)$ follows from exactness of the sequence 
\mbox{$H^k(X;\Z) \xrightarrow{i_A^*} H^k(A;\Z) \xrightarrow{\beta} H^{k+1}(X,A;\Z)$}
and surjectivity of the characteristic class.
\end{proof}
 
\subsubsection{Comparison of two notions of relative differential cohomology.}
Now we compare the two notions of relative differential cohomology, based on differential characters on $Z_{k-1}(i_A;\Z)$ and $Z_{k-1}(X,A;\Z)$, respectively.  
Pre-composition of relative differential characters $h:Z_{k-1}(i_A;\Z) \to \Ul$ with the projection map $q:Z_{k-1}(i_A;\Z) \twoheadrightarrow Z_{k-1}(X,A;\Z)$, $(s,t) \mapsto s + \im({i_A}_*)$, yields a homomorphism 
$$
\widehat H^k(X,A;\Z) \to \widehat H^k(i_A;\Z)\,, \quad h \mapsto h \circ q \,.
$$ 

\begin{thm}[Comparison of relative differential cohomology groups]\label{thm:rel_rel}
Let $i_A:A \to X$ be the embedding of a smooth submanifold.
Let $k \geq 2$.
Then the homomorphism 
$$
\widehat H^k(X,A;\Z) \to \widehat H^k(i_A;\Z) \,, \quad h \mapsto h \circ q \,,
$$
commutes with curvature, characteristic class and the inclusion of cohomology classes in $H^{k-1}(X,A;\Ul)$.
It provides a 1-1 correspondence of $\widehat H^k(X,A;\Z)$ with the subgroup of parallel characters in $\widehat H^k(i_A;\Z)$.
\end{thm}
\index{Theorem!Comparison of relative differential cohomology groups}

\begin{proof}
Let $h \in \widehat H^k(X,A;\Z)$ and $(a,b) \in C_k(i_A;\Z)$.
Then we have
$$
\big(h \circ q\big) (\partial_{i_A}(a,b))
=
h(\partial a +\im({i_A}_*))
\stackrel{\eqref{eq:def_rel_diff_charact_2a}}{=}
\exp \Big( 2\pi i \int_a \curv(h) \Big)\,.
$$
Thus the composition $h \circ q$ is indeed a relative differential character in $\widehat H^k(i_A;\Z)$ and we have $(\curv,\cov)(h \circ q) = (\curv(h),0)$.

The projection $q:Z_{k-1}(i_A;\Z) \to Z_{k-1}(X,A;\Z)$ induces isomorphisms on homology and cohomology.
Moreover, since $\Ul$ is divisible, we have $H^{k-1}(X,A;\Ul) \cong \Hom(H_{k-1}(X,A;\Z),\Ul)$. 
Hence the above homomorphism commutes with the inclusion of cohomology classes $u \in \widehat H^{k-1}(X,A;\Ul)$.

Now let $\tilde h \in C^k(X,A;\Z)$ be a real lift for $h \in \widehat H^k(X,A;\Z)$.
Then $\tilde h \circ q \in C^k(i_A;\Z)$ is a real lift for $h \circ q$.
Since $\curv(h \circ q) = \curv(h) \in \Omega^k_0(X,A)$, we conclude that $\mu^{\tilde h} \circ q$ represents the characteristic class of $h \circ q$.
Hence $c(h \circ q) = c(h)$.   

Since the projection $q:Z_{k-1}(i_A;\Z) \to Z_{k-1}(X,A;\Z)$, $(s,t) \mapsto s + \im({i_A}_*)$, is surjective, the homomorphism $\widehat H^k(X,A;\Z) \to \widehat H^k(i_A;\Z)$ is injective.
As we have seen, its image is contained in the subgroup of parallel characters in $\widehat H^k(i_A;\Z)$.

It remains to show that any parallel character in $\widehat H^k(i_A;\Z)$ lies in the image.
Thus let $h' \in \widehat H^k(i_A;\Z)$ with $\cov(h') =0$.
In particular, $i_A^*\curv(h') =0$, thus $\curv(h') \in \Omega^k_0(X,A)$.
We construct a character $h \in \widehat H^k(X,A;\Z)$ such that $h' = h \circ q$.
From the exact sequences \eqref{eq:sequ_R} and \eqref{eq:R-sequence_X_A} we obtain the following commutative diagram with exact rows and injective vertical maps: 
\begin{equation*}
\xymatrix{
0 \ar[r] & \frac{H^{k-1}(X,A;\R)}{H^{k-1}(X,A;\R)_\Z} \ar^j[rr] \ar[d] && \widehat H^k(X,A;\Z) \ar^{(\curv,c)}[rr] \ar[d] &&  R^k(X,A;\Z) \ar[r] \ar[d] & 0 \\   
0 \ar[r] & \frac{H^{k-1}(X,A;\R)}{H^{k-1}(X,A;\R)_\Z} \ar^j[rr] && \widehat H^k(i_A;\Z) \ar^{((\curv,\cov),c)}[rr] && R^k(i_A;\Z) \ar[r] & 0
}
\end{equation*}
Since $\curv(h') \in \Omega^k_0(X,A)$, we have $(\curv(h'),c(h')) \in R^k(X,A;\Z)$.
By exactness of the upper row we may choose a character $h'' \in \widehat H^k(X,A;\Z)$ with $\curv(h'') = \curv(h') \in \Omega^k_0(X,A)$ and $c(h'') = \curv(h') \in H^k(X,A;\Z)$.
By exactness of the lower row we find $u \in \frac{H^{k-1}(X,A;\R)}{H^{k-1}(X,A;\R)_\Z}$ such that $h' -h''\circ q = j(u)$.
Now put $h:= h'' + j(u) \in \widehat H^k(X,A;\Z)$.
This yields $h \circ q = h'' \circ q + j(u) = h'$. 
\end{proof}

\begin{example}
Let $\varphi:A \to X$ be a smooth map. 
As in Example~\ref{ex:H2} we have the identification of the group $\widehat H^2(\varphi;\Z)$ with the group of isomorphism classes of hermitean line bundles with connection $(L,\nabla) \to X$ and a section $\sigma:A \to \varphi^*L$ along the smooth map $\varphi:A \to X$.

Now let $i_A:A \to X$ be the inclusion of a smooth submanifold.
Theorem~\ref{thm:rel_rel} yields an identification of the group $\widehat H^2(X,A;\Z)$ with the group of isomorphism classes of hermitean line bundles with connection $(L,\nabla) \to X$ and \emph{parallel} sections $\sigma:A \to L|_A$.
In both cases, the isomorphisms are bundle isomorphisms of $L$ that preserve both the connection $\nabla$ and the section $\sigma$.
In particular, they preserve the property of the section $\sigma$ to be parallel with respect to the connection $\nabla$.
\end{example}

\subsection{Relative differential cocycles.}\label{subsec:H_check}
In this section, we discuss the relation of the group $\widehat H^k(X,A;\Z)$ to another notion of relative differential cohomology that has appeared in the literature.
As above let $\varphi:A \to X$ be a smooth map.
The Hopkins-Singer complex of differential cocycles is a cochain complex, the $k$-th homology of which is isomorphic to the differential cohomology group $\widehat H^k(X;\Z)$.
\index{Hopkins-Singer complex}
\index{differential cocycle}
In the original definition of differential cocycles in \cite{HS05}, the complex that computes the $k$-th differential cohomology group $\widehat H^k(X;\Z)$ depends on the degree $k$: for each degree of differential cohomology, one has to consider a different complex.\footnote{The same holds for smooth Deligne cohomology: the smooth Deligne complex, the $k$-th homology of which is isomorphic to $\widehat H^k(X;\Z)$, is the total complex of a \v{C}ech-de Rham double complex, truncated at the de Rham order $(k-1)$.}
However, the Hopkins-Singer complex can be modified such that differential cohomology groups of all degrees arise as homology groups of a single complex, see \cite{BT06}.
 
The \emph{relative Hopkins-Singer differential cohomology group} $\check{H}^k(\varphi;\Z)$ is defined in \cite{BT06} as the $k$-th homology group of the mapping cone complex of the modified Hopkins-Singer complex.
\index{+HcheckvarphiZ@$\check H^*(\varphi;\Z)$, relative Hopkins-Singer group}
\index{relative Hopkins-Singer group}
\index{Hopkins-Singer group!relative $\sim$}
The cocycles of this mapping cone complex are referred to as \emph{relative differential cocycles}. 
\index{relative differential cocycle}
\index{differential cocycle!relative $\sim$}

The main feature of the relative Hopkins-Singer groups $\check{H}^*(\varphi;\Z)$ is the long exact sequence they fit into:  
The complex of relative differential cochains sits in the usual short exact sequence of cochain complexes which relates the modified Hopkins-Singer complexes on $X$ and $A$ to the corresponding mapping cone complex.
Thus the relative Hopkins-Singer groups fit into the following long exact sequence \cite{BT06}:
\index{long exact sequence!for relative Hopkins-Singer groups}
\index{differential cocycle!long exact sequence}
\index{Hopkins-Singer group!long exact sequence}
\begin{equation}\label{eq:long_ex_sequ_HS}
\xymatrix{
\ldots \ar[r] & \widehat H^{k-1}(A;\Z) \ar[r] & \check{H}^k(\varphi;\Z) \ar[r] & \widehat H^k(X;\Z) \ar^{\varphi^*}[r] & \widehat H^k(A;\Z)  \ar[r] & \ldots
}
\end{equation}
Comparison of \eqref{eq:long_ex_sequ_HS} with the long exact sequences \eqref{eq:long_ex_sequ} and \eqref{eq:long_ex_sequ_par} for the relative groups $\widehat H^k(\varphi;\Z)$ and $\widehat H^k(X,A;\Z)$ (for $\varphi = i_A$) shows that the relative Hopkins-Singer group $\check{H}^k(i_A;\Z)$ differs from both. 

In \cite{BT06} it is shown that the relative Hopkins-Singer group $\check{H}^k(\varphi;\Z)$ is a subquotient of the group of relative differential characters $\widehat H^k(\varphi;\Z)$.
More precisely, it is a quotient of the subgroup $\widehat H^k_0(\varphi;\Z) := \{h \in \widehat H^k(\varphi;\Z) \,|\, \varphi^*\vds_\varphi(h) =0 \}$.
From the results of Section~\ref{subsec:rel_diff_par}, we easily obtain the following identification:

\begin{prop}\label{prop:cov_int_periods}
Let $\varphi:A \to X$ be a smooth map.
Then the subgroup $\widehat H^k_0(\varphi;\Z) \subset \widehat H^k(\varphi;\Z)$ coincides with the subgroup of characters with covariant derivative in $\Omega^{k-1}_0(A)$, i.e.~a closed form with integral periods.  
In particular, we have the exact sequence:
\begin{equation}\label{eq:H0_sequ}
\xymatrix{
0 \ar[r] & \widehat H^k_0(\varphi;\Z) \ar[r] & \widehat H^k(\varphi;\Z) \ar^\cov[r] & \frac{\pr_2(\Omega^{k-1}_0(\varphi))}{\Omega^{k-1}_0(A)} \ar[r] & 0
}  
\end{equation}
Here $\pr_2(\Omega^{k-1}_0(\varphi)) \subset \Omega^{k-1}(A)$ denotes the image of the projection $\pr_2:\Omega^{k-1}_0(\varphi) \to \Omega^{k-1}(A)$, $(\omega,\vartheta) \mapsto \vartheta$.
\end{prop}
 
\begin{proof}
The identification of $\widehat H^k_0(\varphi;\Z) = \ker(\varphi^* \circ \vds_\varphi)$ follows from the commutative diagram \eqref{eq:diag:triv_cov} and the exact sequence \eqref{eq:rel_sequences} for topological trivializations.
The exact sequence follows from this identification.
\end{proof}

The sequence \eqref{eq:H0_sequ} appeared in \cite[Prop.~4.1]{BT06}.
As mentioned above, it is shown in \cite{BT06} that the relative Hopkins-Singer group $\check{H}^k(\varphi;\Z)$ is a quotient of $\widehat H^k_0(\varphi;\Z)$.
More precisely, we have the exact sequence \cite[Thm.~4.2]{BT06}:
\begin{equation}\label{eq:H_0_H_check}
\xymatrix{
0 \ar[r] & \frac{\Omega^{k-1}_0(X)}{\widetilde \Omega^{k-1}(X)} \ar[r] & \widehat H^k_0(\varphi;\Z) \ar[r] & \check{H}^k(\varphi;\Z) \ar[r] & 0 \,.
}
\end{equation}
Here we have $\widetilde \Omega^{k-1}(X):= \{\omega \in \Omega^{k-1}_0(X) \,|\, (\omega,0) \in \Omega^{k-1}_0(\varphi)\}$ with the homomorphism \mbox{$\frac{\Omega^{k-1}_0(X)}{\widetilde \Omega^{k-1}(X)} \to \widehat H^k_0(\varphi;\Z)$}, $\omega \mapsto \iota_\varphi(\omega,0)$.

It remains to determine the relation between the relative Hopkins-Singer group $\check{H}^k(i_A;\Z)$ and the group $\widehat H^k(X,A;\Z)$.
It turns out that the latter is a subgroup of the former: 

\begin{prop}
Let $i_A:A \to X$ be the embedding of a smooth submanifold.
Then the following sequences are exact:
\begin{equation*}
\xymatrix{
0 \ar[r] & \widehat H^k(X,A;\Z) \ar[r] & \widehat H^k_0(i_A;\Z) \ar^\cov[r] & \Omega^{k-1}_0(A) \ar[r] & 0 \\
0 \ar[r] & \widehat H^k(X,A;\Z) \ar[r] & \check{H}^k(i_A;\Z) \ar^\cov[r] & \frac{\Omega^{k-1}_0(A)}{i_A^*\Omega^{k-1}_0(X)} \ar[r] & 0 \,. 
} 
\end{equation*}
\end{prop}

\begin{proof}
Exactness of the first sequence follows from the results of Section~\ref{subsec:rel_diff_par}:
Clearly, $\widehat H^k(X,A;\Z)$ the kernel of $\cov:\widehat H^k_0(i_A;\Z) \to \Omega^{k-1}_0(A)$ by Theorem~\ref{thm:rel_rel} and Proposition~\ref{prop:cov_int_periods}.
The latter is surjective since for any $\vartheta \in \Omega^{k-1}_0(A)$, we may choose a differential character $h \in \widehat H^{k-1}(A;\Z)$ with $\curv(h) = \vartheta$.
Then $\ti_{i_A}(-h) \in \widehat H^k_0(i_A;\Z)$ and $\cov(\ti_{i_A}(-h)) = \vartheta$.

The second sequence is obtained from the first by dividing out the action of $\Omega^{k-1}_0(X)$ on $\widehat H^k_0(i_A;\Z)$.
Vanishing of the composition $\widehat H^k(X,A;\Z) \to \check{H}^k(i_A;\Z) \to \frac{\Omega^{k-1}_0(A)}{i_A^*\Omega^{k-1}_0(X)}$ follows from the first sequence.

We show exactness at $\widehat H^k(X,A;\Z)$:
For $\varphi = i_A$, we have $\widetilde \Omega^{k-1}(X) = \Omega^{k-1}_0(X,A)$.
Let $h \in \widehat H^k(X,A;\Z)$ which maps to $0$ in $\check{H}^k(i_A;\Z)$.
By Theorem~\ref{thm:rel_rel} we may consider $h$ as a parallel character in $\widehat H^k_0(i_A;\Z)$.
From the exact sequence \eqref{eq:H_0_H_check} we conclude $h = \iota_{i_A}(\omega,0)$ for some $\omega \in \Omega^{k-1}_0(X)$.
Now we have $0 = \cov(h) = \cov(\iota_{i_A}(\omega,0)) = \varphi^*\omega$.
Thus $(\omega,0) \in \Omega^{k-1}_0(X,A)$, hence $h=\iota_{i_A}(\omega,0) =0 \in \widehat H^k_0(i_A;\Z)$.
From the first sequence we conclude $h=0$.

Next we show exactness at $\check{H}^k(i_A;\Z)$:
Let $[h] \in \ker \big(\check{H}^k(i_A;\Z) \to \frac{\Omega^{k-1}_0(A)}{i_A^*\Omega^{k-1}_0(X)}\big)$.
Choose a representant $h \in \widehat H^k_0(i_A;\Z)$ of the equivalence class $[h] \in \check{H}^k(i_A;\Z)$. 
Then there exists a differential form $\omega \in \Omega^{k-1}_0(X)$ such that $\cov(h) = \varphi^*\omega$.
Now put $h' := h - \iota_{i_A}(\omega,0)$.
Then we have $\cov(h') = \cov(h) - \varphi^*\omega =0$, thus $h' \in \widehat H^k(X,A;\Z)$.
From the exact sequence \eqref{eq:H_0_H_check} we conclude that $[h]$ is the image of $h'$ under the map $\widehat H^k(X,A;\Z) \to \check{H}^k(i_A;\Z)$.

Finally, exactness at $\frac{\Omega^{k-1}_0(A)}{i_A^*\Omega^{k-1}_0(X)}$ is clear from exactness of the first sequence and the sequence \eqref{eq:H_0_H_check}.
\end{proof}

\section{Internal and external products}\label{sec:products}

In this chapter we discuss internal and external products in differential cohomology.
The internal product of differential characters and the induced ring structure on differential cohomology $\widehat H^*(X;\Z)$ has first been constructed in \cite{CS83}.
Uniqueness of the ring structure was proved in \cite{SS08} and \cite{BB13}.
The proof in \cite[Ch.~6]{BB13} starts from an axiomatic definition of internal and external products, similar to the one in \cite{SS08}, and ends up with an explicit formula.
In this sense the proof is constructive.
Simple formulas for the ring structure are obtained in models of differential cohomology based on differential forms with singularities \cite{C73}, de Rham-Federer currents as in \cite[Sec.~3]{HLZ03}, and differential cocycles \cite{HS05}, \cite{BKS10}. 

Uniqueness of (the external product and) the ring structure has been shown in \cite{SS08} and \cite[Ch.~6]{BB13}.
Our proof in there is constructive in the sense that it directly yields a formula for the external product, starting from an abstract definition. 
In Section~\ref{subsec:ring_unique}, we derive that formula directly from the original construction of the ring structure in \cite{CS83}.
In Section~\ref{subsec:module} we use the methods of \cite[Ch.~6]{BB13} to construct a cross product between relative and absolute differential characters.
This in turn provides the graded abelian group $\widehat H^*(\varphi;\Z)$ with the structure of a right module over the ring $\widehat H^*(X;\Z)$.
The module structure is constructed from the cross product by pull-back along a version of the diagonal map.

In \cite{BB13} we focussed only on \emph{uniqueness} of the cross product and ring structure of absolute differential characters since existence of the products is well-known. 
In the present paper, we only prove \emph{existence} of the cross product between relative and absolute characters and the module structure on $\widehat H^*(\varphi;\Z)$. 
We do not prove uniqueness of the cross product.
However, the uniqueness proof from \cite[Ch.~6]{BB13} for absolute differential cohomology carries over directly to the notion of relative differential cohomology considered here.

\subsection{The ring structure on differential cohomology}\label{subsec:ring_unique}
In this section we briefly recall the original construction of the internal product
$$
*: \widehat H^k(X;\Z) \times \widehat H^{k'}(X;\Z) \to \widehat H^{k+k'}(X;\Z)%
$$
from \cite{CS83}.
We derive another formula for the induced external product based on representation of smooth singular cohomology by geometric cycles.
This new formula was proved in \cite[Ch.~6]{BB13} by means of an abstract definition of internal and external products of differential characters.
Here we derive the new formula for the external product from the original Cheeger-Simons formula for the internal product. 

\subsubsection{The Cheeger-Simons internal product.}
Let $h \in \widehat H^k(X;\Z)$ and $h' \in \widehat H^{k'}(X;\Z)$ be differential characters.
Choose real lifts $\tilde h \in C^{k-1}(X;\R)$ and $\tilde h' \in C^{k'-1}(X;\R)$.
Denote by $B: C_*(X;\Z) \to C_*(X;\Z)$ the barycentric subdivision and by $H:C_*(X;\Z) \to C_{*+1}(X;\Z)$ a chain homotopy from $B$ to the identity, hence
\begin{equation}\label{eq:chain_homo}
\mathrm{id} - B = \partial \circ H + H \circ \partial.
\end{equation}
By construction, the image of the characteristic class $c(h)$ in $H^k(X;\R)$ coincides with the de Rham cohomology class of $\curv(h)$ under the de Rham isomorphism.
The wedge product of closed differential forms, regarded as smooth singular cocycles, descends to the cup product on $H^*_\mathrm{dR}(X) \cong H^*(X;\R)$.
Thus the differential form $\curv(h) \wedge \curv(h')$, regarded as a smooth singular cocycle, differs from the cocycle $\curv(h) \cup \curv(h')$ by a real coboundary.
An explicit construction of a cochain $E(h,h') \in C^{k+k'-1}(X;\R)$ such that $\delta E(h,h') = \curv(h) \wedge \curv(h') -\curv(h) \cup \curv(h')$ is given in \cite[p.~55f.]{CS83} by:
\begin{equation*} 
E(h,h')(x) := - \sum_{j=0}^\infty (\curv(h) \cup \curv(h'))(H(B^j x)) \,.
\end{equation*}
Here $x \in C_{k+k'-1}(X;\R)$.

Now put 
$\nu(\tilde h,\tilde h'):= \big(\tilde h \cup \mu^{\tilde h'} + (-1)^k \curv(h) \cup \tilde h'\big) \in C^{k+k'-1}(X;\R)$.
Then the differential character $h*h' \in \widehat H^{k+k'}(X;\Z)$ is defined by 
\begin{equation}\label{eq:def_ring_CS}
(h * h')(z)
:=  \exp \Big( 2 \pi i \Big( \nu(\tilde h,\tilde h') + E(h,h')\Big)(z) \Big) \,,
\end{equation}
where $z \in Z_{k+k'-1}(X;\Z)$.
\index{differential character!ring structure}
\index{differential character!internal product}
\index{+Star@$*$, internal product}
\index{internal product!of absolute characters}
\index{product!internal}
As observed in \cite{CS83}, the internal product $*$ is well-defined, i.e.~it does not depend upon the choice of real lifts $\tilde h,\tilde h'$ and chain homotopy $H$.
Moreover, the product $*$ is associative and graded commutative, and it is natural with respect to smooth maps.
It is compatible with the exact sequences in \eqref{eq:rel_sequences} in the sense that curvature and characteristic class are multiplicative and
\begin{equation}\label{eq:prod_iota}
\iota(\varrho) * h' 
= \iota(\varrho \wedge \curv(h')). 
\end{equation}

By definition, the internal product is $\Z$-bilinear.
In particular, if $h =0$ or $h'=0$, then $h*h'=0$.

Using these properties, we derive an expression for the internal product that no longer involves the cochain $E(h,h')$.
For a similar formula, see \cite[p.~57]{CS83}.

\begin{prop}\label{prop:ring_BB}
Let $h \in \widehat H^k(X;\Z)$ and $h' \in \widehat H^{k'}(X;\Z)$ be differential characters on $X$.
Let $z \in Z_{k+k'-1}(X;\Z)$ be cycle. 
Choose a geometric cycle $M\xrightarrow{f}X$ that represents the homology class $[z]$.
Let $y \in Z_{k+k'-1}(M;\Z)$ be a fundamental cycle of $M$.
Choose a chain $a(z) \in C_{k+k'}(X;\Z)$ such that $z = f_*y + \partial a(z)$.
Then we have 
\begin{equation}\label{eq:def_ring_BB}
(h*h')(z) 
= \lim_{j \to \infty} \exp \Big( 2 \pi i \Big( f^*\nu(\tilde h,\tilde h')(B^j y) + \int_{a(z)} \curv(h) \wedge \curv(h') \Big) \Big) \,. 
\end{equation}
\end{prop}

\begin{proof}
Since $z = f_*y + \partial a(z)$, we have:
\begin{align*}
(h*h')(z) 
&=
f^*(h*h')(y) \cdot \exp \Big( 2\pi i \int_{a(z)} \curv(h*h') \Big) \\
&=
(f^*h*f^*h')(y) \cdot \exp \Big( 2 \pi i \int_{a(z)} \curv(h) \wedge \curv(h') \Big) \,.
\end{align*}
We compute $f^*(h*h')([M])$ using~\eqref{eq:def_ring_CS} on the stratifold $M$.
The characteristic class $c(h)$ is represented by the cocycle $\mu^{\tilde h} := \curv(h) - \delta \tilde h \in C^k(M;\Z)$, and similarly for $h'$.
Now we have:
\begin{align*}
\mu^{\tilde h} \cup \mu^{\tilde h'}
&= (\curv(h) - \delta \tilde h) \cup (\curv(h') - \delta \tilde h') \\
&= \curv(h) \cup \curv(h') - \delta( \tilde h \cup (\curv(h') - \delta \tilde h') + (-1)^k \curv(h) \cup \tilde h' ) \\
&= \curv(h) \cup \curv(h') - \delta \nu(\tilde h,\tilde h'). 
\end{align*}
Since $M$ is $(k+k'-1)$-dimensional, the cocycle $f^*(\mu^{\tilde h} \cup \mu^{\tilde h'})$ is an integral coboundary for dimensional reasons.
Thus we have $f^*(\mu^{\tilde h} \cup \mu^{\tilde h'}) = \delta t$ for some $t \in C^{k+k'-1}(M;\Z)$ and hence $f^*(\curv(h) \cup \curv(h')) = \delta t + \delta \nu(f^*\tilde h,f^*\tilde h')$.
Evaluating the cochain $E(f^*h,f^*h')$ on the fundamental cycle $y$ of $M$, we obtain
\begin{align}
E(f^*h,f^*h')(y)
&= - \sum_{j=0}^\infty (f^*\curv(h) \cup f^*\curv(h'))(H(B^j y)) \notag \\
&= - \sum_{j=0}^\infty \delta (t + \nu(f^* \tilde h,f^*\tilde h'))(H(B^j y)) \notag \\
&= - \sum_{j=0}^\infty (t + \nu(f^* \tilde h,f^* \tilde h'))(\partial H(B^j y)) \notag \\
&\stackrel{\eqref{eq:chain_homo}}{=} - \sum_{j=0}^\infty (t + \nu(f^*\tilde h,f^*\tilde h'))((\mathrm{id} -B) B^j y - H(\partial B^j y)) \notag \\
&= - (t + \nu(f^*\tilde h,f^*\tilde h'))(y)  + \lim_{j\to\infty}(t + \nu(f^*\tilde h,f^*\tilde h'))(B^j y) \notag \\
&= \underbrace{-t(y) + \lim_{j\to\infty} t(B^j y)}_{\in \Z} - \nu(f^* \tilde h,f^*\tilde h')(y) + \lim_{j\to\infty}\nu(f^*\tilde h,f^*\tilde h')(B^j y). \label{eq:ring_2}
\end{align}
We thus have:
\begin{align*}
f^*(h * h')(y) 
&\stackrel{\eqref{eq:def_ring_CS}}{=} \exp \Big( 2 \pi i \Big( f^*\nu(\tilde h,\tilde h') + E(f^*h,f^*h') \Big)(y) \Big) \notag \\
&\stackrel{\eqref{eq:ring_2}}{=} \exp \Big( 2 \pi i \Big( (f^*\nu(\tilde h,\tilde h') - \nu(f^*\tilde h,f^*\tilde h'))(y) + \lim_{j\to\infty}\nu(f^*\tilde h,f^*\tilde h')(B^j y)\Big)\Big) \notag \\
&= \exp \Big( 2 \pi i \Big( \lim_{j\to\infty}\nu(f^*\tilde h,f^*\tilde h')(B^j y) \Big)\Big). \qedhere
\end{align*}
\end{proof}

\subsubsection{The external or cross product.}
Similar to singular cohomology, the internal product of differential characters on a manifold $X$ gives rise to an external or cross product
$$
\times: \widehat H^k(X;\Z) \times \widehat H^{k'}(X';\Z) \to \widehat H^{k+k'}(X\times X';\Z)\,, \quad (h,h') \mapsto (\pr_1^*h) * (\pr_2^*h') \,.   
$$
Here $\pr_1, \pr_2$ denote the projection on the first and second factor of $X \times X'$, respectively.   
Conversely, the internal product can be recovered from the external product by pull-back along the diagonal $\Delta_X: X \to X \times X$, $x \mapsto (x,x)$: for characters $h,h' \in \widehat H^k(X;\Z)$, we have:
$$
\Delta_X^*(h \times h') 
= \Delta_X^* (\pr_1^*h * \pr_2^*h') 
= (\pr_1 \circ \Delta_X)^*h * (\pr_2 \circ \Delta_X)^*h' 
= h*h' \,.
$$
\index{differential character!cross product}
\index{differential character!external product}
\index{+X@$\times$, cross product}
\index{external product!of absolute characters}
\index{product!external}

The external product is $\Z$-bilinear.
Moreover, since curvature and characteristic class are multiplicative for the internal product, the same holds for the external product: 
\begin{align}
\curv(h \times h') 
&= \curv(h) \times \curv(h') \in \Omega^{k+k'}(X\times X') ,  \label{eq:curvx_commute} \\
c(h \times h')
&= c(h) \times c(h') \in H^{k+k'}(X \times X';\Z) . \label{eq:cx_commute} 
\end{align}

\subsubsection{A formula for the cross product}
To understand the external product $h \times h'$ of differential characters $h \in \widehat H^k(X;\Z)$ and $h'\in \widehat H^{k'}(X';\Z)$, the following special case is crucial.
It is the key step in the uniqueness proof in \cite[Ch.~6]{BB13}.
We give another proof here, based on Proposition~\ref{prop:ring_BB}.

\begin{lemma}\label{lem:x_M_M'}
Let $M$ and $M'$ be closed oriented regular stratifolds with $\dim(M) + \dim(M') = k+k'-1$.
Let $h \in \widehat H^k(M;\Z)$ and $h' \in \widehat H^{k'}(M';\Z)$.
Then we have:
\begin{equation}\label{eq:MM'}
(h \times h')([M \times M']) 
= \begin{cases}
  h([M])^{\langle c(h'),[M']\langle}  & \quad \mbox{if $(\dim(M),\dim(M')) = (k-1,k')$} \\ 
  h'([M'])^{(-1)^k \langle c(h),[M]\rangle}  & \quad \mbox{if $(\dim(M),\dim(M')) = (k,k'-1)$} \\ 
  1 & \quad \mbox{otherwise} \,.
  \end{cases} 
\end{equation}
\end{lemma}

\begin{proof}
If $\dim(M) < k-1$, we have $\widehat H^k(M;\Z) = \{0\}$.
If $\dim(M') < k'-1$, we have $\widehat H^{k'}(M';\Z) = \{0\}$. 
Thus if $(\dim(M),\dim(M')) \notin \{(k-1,k'),(k,k'-1)\}$ then either $h=0$ or $h'=0$.
Hence $h \times h'=0$.

Now let $y \in Z_{k+k'-1}(M \times M';\Z)$ be a fundamental cycle of $M \times M'$.
By \eqref{eq:def_ring_BB}, we have
\begin{align*}
(h \times h')([M \times M'])
&= (\pr_1^*h * \pr_2^*h')(y) \\
&\stackrel{\eqref{eq:def_ring_BB}}{=} \lim_{j \to \infty} \exp \big( 2\pi i \cdot \nu(\pr_1^*\widetilde h,\pr_2^*\widetilde{h'})(B^jy) \big) \\
&= \lim_{j \to \infty} \exp \Big( 2\pi i \cdot \big(\widetilde h \times \mu^{\widetilde{h'}} + (-1)^k \curv(h) \times \widetilde{h'} \big)(B^jy)\Big) 
\end{align*}
Since $(h \times h')([M \times M'])$ does not depend upon the choice of fundamental cycle, we may choose $y = x \times x'$, where $x$ and $x'$ are fundamental cycles of $M$ and $M'$, respectively.
Moreover, we have $B^jy = y + \partial b_j$ for some $b_j \in C_{k+k'}(M \times M';\Z) = S_{k+k'}(M \times M';\Z)$.

If $\dim(M) = k-1$, we may choose $\widetilde h$ to be a cocycle.
Then $\widetilde h \times \mu^{\widetilde{h'}}$ is a cocycle, too. 
Moreover, $\curv(h) =0$ in this case, and we obtain:
\begin{align*}
(h \times h')([M \times M'])
&= \lim_{j \to \infty} \exp \Big( 2\pi i \cdot \big(\widetilde h \times \mu^{\widetilde{h'}} + (-1)^k \curv(h) \times \widetilde{h'} \big)(x \times x' + \partial b)\Big) \\
&= \exp \Big( 2\pi i \cdot \widetilde h(x) \cdot \mu^{\widetilde{h'}}(x')\Big) \\
&= \big(h([M])\big)^{\langle c(h'),[M']\rangle} \\
\intertext{Similarly, for $\dim(M) = k$, we have $\dim(M')=k'-1$, hence $\curv(h') =0$.
We may choose $\widetilde{h'}$ to be a cocycle.
Then $\curv(h) \times \widetilde{h'}$ is a cocycle and $\mu^{\widetilde{h'}}=0$.
This yields:}
(h \times h')([M \times M'])
&= \lim_{j \to \infty} \exp \Big( 2\pi i \cdot \big(\widetilde h \times \mu^{\widetilde{h'}} + (-1)^k \curv(h) \times \widetilde{h'} \big)(x \times x' + \partial b)\Big) \\
&= \exp \Big( 2\pi i \cdot (-1)^k \cdot \widetilde{h'}(x') \cdot \int_M \curv(h) \Big) \\ 
&= \big(h'([M'])\big)^{(-1)^k \cdot \langle c(h),[M]\rangle}. \qedhere
\end{align*}
\end{proof}

From Lemma~\ref{lem:x_M_M'} and Remark~\ref{rem:ev_torsion} we easily obtain a formula for the external product as in \cite[Ch.~6]{BB13}.
The K\"unneth sequence allows us to decompose cycles in $X \times X'$ into a sum of cross products of cycles in $X$ and $X'$ respectively, and a torsion cycle.
This is done by carefully constructing splittings of the K\"unneth sequence at the level of cycles as explained in detail in \cite[Ch.~6]{BB13}.
\index{K\"unneth splitting}
\index{splitting!K\"unneth $\sim$}
The construction is briefly reviewed in the appendix.
The construction for the relative K\"unneth sequence is given in detail there.
We use it to construct the module structure on $\widehat H^*(\varphi;\Z)$ in the following section.  

Let $h \in \widehat H^k(X;\Z)$ and $h' \in \widehat H^{k'}(X';\Z)$ be differential characters.
To evaluate the character $h \times h' \in \widehat H^{k+k'}(X \times X';\Z)$ on a cycle $z \in Z_{k+k'-1}(X \times X';\Z)$ we use the K\"unneth decomposition of $z$ into a sum of cross products of cycles on $X$ and $X'$ and a torsion cycle. 
Then the two types of summands are treated separately.
This yields:

\begin{cor}[Formula for the cross product]\label{cor:cross}
Let $h \in \widehat H^k(X;\Z)$ and $h' \in \widehat H^{k'}(X';\Z)$, and let $z \in Z_{k+k'-1}(X \times X';\Z)$.
Decompose it according to the K\"unneth sequence into a sum of cross products of cycles in $X$ and $X'$, respectively, and torsion cycles.
Then $h \times h'$ is evaluated on the two types of summands separately:
\index{formula!cross product}
\index{differential character!cross product}

If $z$ represents an $N$-torsion class, choose a chain $b \in C_{k+k'}(X\times X';\Z)$ such that $\partial b = N \cdot z$.
Then we have:
\begin{equation}\label{eq:cross_ev_on_torsion}
(h \times h')(z)
=
\exp \Big( \frac{2\pi i}{N} \Big( \int_b \curv(h) \times \curv(h') - \langle c(h) \times c(h'), b \rangle \Big) \Big) \,. 
\end{equation}
If $z = y_i \times y'_j$ with $y_i \in Z_i(X;\Z)$ and $y'_j \in Z_j(X';\Z)$ we have:
\begin{equation}\label{eq:xx}
(h \times h')(y_i \times y'_j) 
= \begin{cases}
  h(y_i)^{\langle c(h'),y'_j \rangle}  & \quad \mbox{if $(i,j) = (k-1,k')$} \\ 
  h'(y'_j)^{(-1)^k \langle c(h),y_i\rangle}  & \quad \mbox{if $(i,j) = (k,k'-1)$} \\ 
  1 & \quad \mbox{otherwise} \,.
  \end{cases} 
\end{equation}
\end{cor}

\begin{proof}
The representation \eqref{eq:cross_ev_on_torsion} on torsion cycles follows from Remark~\ref{rem:ev_torsion} and \eqref{eq:curvx_commute}, \eqref{eq:cx_commute}.

Now let $z = y_i \times y'_j$ where $y_i \in Z_i(X;\Z)$ and $y'_j \in Z'_j(X';\Z)$. 
Choose geometric cycles $\zeta(y_i) \in \ZZ_i(X)$ and $\zeta'(y'_j) \in \ZZ_j(X')$ and chains $a(y_i) \in C_{i+1}(X;\Z)$ and $a(y'_j) \in C_{j+1}(X';\Z)$ such that $[y_i-\partial a(y_i)]_{\partial S_{i+1}}=[\zeta(y_i)]_{\partial S_{i+1}}$ and
$[y'_j-\partial a(y'_j)]_{\partial S_{j+1}}=[\zeta(y'_j)]_{\partial S_{j+1}}$.

Now apply Lemma~\ref{lem:x_M_M'} to the fundamental cycles of $\zeta(y_i)$ and $\zeta'(y'_j)$:
For degrees $(i,j)$ different from $(k-1,k')$ and $(k,k'-1)$ we have $(h \times h')(y_i \times y'_j)=1$. 
For $(i,j)=(k-1,k')$ we obtain:
\begin{align*}
(h \times h')(y_{k-1} \times y'_{k'})
&=
h([\zeta(y_{k-1})]_{\partial S_k})^{\langle c(h'),y'_{k'}\rangle} 
\cdot \exp \Big( 2\pi i \int_{(a(y_i) \times y'_{k'})} \curv(h \times h') \Big) \notag \\
&=
h([\zeta(y_{k-1})]_{\partial S_k})^{\langle c(h'),y'_{k'}\rangle} 
\cdot \exp \Big( 2\pi i \int_{a(y_{k-1})} \curv(h) \cdot \int_{y'_{k'}} \curv(h') \Big) \notag \\
&=
h([\zeta(y_{k-1})]_{\partial S_k})^{\langle c(h'),y'_{k'}\rangle} 
\cdot h(\partial a(y_{k-1}))^{\langle c(h'),y'_{k'}\rangle} \notag \\
&=
h(y_{k-1})^{\langle c(h'),y'_{k'}\rangle} \,.
\end{align*}
Similarly for $(i,j)=(k-1,k')$ we obtain:
\begin{equation*}
(h \times h')(y_k \times y'_{k'-1})
=
h(y'_{k'-1})^{(-1)^k \cdot \langle c(h),y_{k'}\rangle} \,. \qedhere
\end{equation*}
\end{proof}

\subsection{The module structure on relative differential cohomology}\label{subsec:module}
In this section, we use the method developed in \cite[Ch.~6]{BB13} to construct an external and internal product between relative and absolute differential characters.
This provides the graded abelian group $\widehat H^*(\varphi;\Z)$ of relative differential characters with the structure of a right module over the ring $\widehat H^*(X;\Z)$. 
The module structure is natural with respect to smooth maps.
It is compatible with the module structures on relative cohomology and the relative de Rham complex in the sense that the structure maps (i.e.~curvature, covariant derivative, characteristic class and topological trivializations) are multiplicative.
Moreover, the module structure is compatible with the maps $\vds$ and $\ti$ between absolute and relative differential characters groups.

\subsubsection{The cross product.}
As above let $\varphi:A \to X$ be a smooth map.
We consider the induced map $\varphi \times \id_{X'}:A \times X' \to X \times X'$.
The cup product on smooth singular cochains induces an internal product between relative and absolute cochains 
$$
\cup: C_*(\varphi;\Z) \otimes C_*(X;\Z) \to C_*(\varphi;\Z)\,, \quad (\mu,\nu) \cup \sigma := (\mu \cup \sigma,\nu \cup \sigma) \,.
$$
Likewise, the cup product induces an external or cross product
$$
\times: C_*(\varphi;\Z) \otimes C_*(X';\Z) \to C_*(\varphi \times \id_{X'};\Z)\,, \quad (\mu,\nu) \times \sigma := (\mu \times \sigma,\nu \times \sigma) \,.
$$
Since cross and cup products are natural chain maps, so are the induced internal and external products between absolute and relative cochains.
Clearly, the products are invariant under the boundary operator of the mapping cone complex and thus descend to the cup and cross products on cohomology.
Th cup product in particular provides the mapping cone cohomology $H^*(\varphi;\Z)$ with the structure of a right module over the cohomology ring $H^*(X;\Z)$.

Likewise, the wedge product\footnote{We avoid the familiar term ``exterier product'' to avoid confusion with the external product. The wedge product clearly defines an \emph{internal} rather than an \emph{external} product on the de Rham complex.} of differential forms induces an internal product between relative differential forms $(\omega,\vartheta) \in \Omega^k(\varphi)$ and differential forms $\omega' \in \Omega^{k'}(X)$: 
$$
(\omega,\vartheta) \wedge \omega'
:=
(\omega \wedge \omega',\vartheta \wedge \varphi^*\omega') \in  \Omega^{k+k'}(\varphi) 
$$
This provides the mapping cone de Rham complex $\Omega^*(\varphi)$ with the structure of a right module over the ring $\Omega^*(X)$ of differential forms on $X$.
Similarly, we have the external product of $(\omega,\vartheta) \in \Omega^k(\varphi)$ with $\omega' \in \Omega^{k'}(X')$:
$$
(\omega,\vartheta) \times \omega'
:=
(\omega \times \omega',\vartheta \times \varphi^*\omega')  \in \Omega^{k+k'}(\varphi \times \id_{X'})\,. 
$$
The internal and external products on the de Rham complex $\Omega^*(\varphi)$ and the cochain complex $C^*(\varphi;\R)$ coincide in cohomology under the de Rham isomorphism.
Hence they induce the same module structure on $H_\mathrm{dR}^*(\varphi) \cong H^*(\varphi;\R)$. 
\index{external product!of differential forms}

Now we construct the external product between a relative character $h \in \widehat H^*(\varphi;\Z)$ and an absolute character $h' \in \widehat H^*(X';\Z)$.
The construction is completely analogous to the one for absolute characters reviewed in the previous section.

We have the relative K\"unneth sequence
$$
0 \to 
\big[H_*(\varphi;\Z) \otimes H_*(X';\Z)\big]_n \xrightarrow{\times}
H_n(\varphi\times \id_{X'};\Z) \to 
\Tor(H_*(\varphi;\Z),H_*(X';\Z))_{n-1} \to
0 \,.
$$
As is well-known, the sequence splits on the level of cycles. 
A construction of a splitting $S:Z(C_*(\varphi\times\id_{X'};\Z) \to Z_*(\varphi;\Z) \otimes Z_*(X';\Z)$ is given in the appendix.
Denote the complement of the image of $Z_*(\varphi;\Z) \otimes Z_*(X';\Z) \xrightarrow{\times} Z_{k+k'-1}(\varphi\times\id_{X'};\Z)$, obtained from the K\"unneth splitting, by $T_{k+k'-1}(\varphi\times\id_{X'};\Z)$.
It will be referred to as the \emph{K\"unneth complement}.
\index{K\"unneth complement}
\index{+TkkvarphiidZ@$T_{k+k'-1}(\varphi\times\id_{X'};\Z)$, K\"unneth complement}

Now let $(s,t) \in Z_{k+k'-1}(\varphi\times \id_{X'};\Z)$ be a cycle.
The K\"unneth splitting allows us to decompose $(s,t)$ into a sum of cross products of cycles $(x,y) \in Z_*(\varphi;\Z)$ and $y' \in Z_*(X';\Z)$ and torsion cycles in $T_{k+k'-1}(\varphi\times\id_{X'};\Z)$.
Analogously to Corollary~\ref{cor:cross} we define the external product as follows:

\begin{definition}[Cross product]\label{def:cross}
The \emph{cross product} of differential characters $h \in \widehat H^k(\varphi;\Z)$ and $h' \in H^{k'}(X';\Z)$ is the homomorphism $h \times h':Z_{k+k'-1}(\varphi \times \id_{X'};\Z) \to \Ul$ defined as follows:
\index{cross product!of absolute and relative characters}
\index{relative differential character!cross product}
\index{external product!of absolute and relative characters}
\index{relative differential character!external product}
\index{product!external $\sim$}

For cycles $(x,y) \in Z_i(\varphi;\Z)$ and $y' \in Z_j(X';\Z)$, put
\begin{equation}
(h \times h')((x,y) \times y')
= \begin{cases}
   h(x,y)^{\langle c(h'),y'\rangle} & \mbox{if $(i,j) = (k-1,k')$} \\
   h'(y')^{(-1)^k \cdot \langle c(h),(x,y)\rangle} & \mbox{if $(i,j)=(k,k'-1)$} \\
   1 & \mbox{otherwise.}
   \end{cases} \label{eq:def_cross_1}
\end{equation}
For an $N$-torsion cycle $(s,t) \in T_{k+k'-1}(\varphi\times\id_{X'};\Z)$ in the K\"unneth complement choose a chain $(a,b) \in C_{k+k'}(\varphi \times \id_{X'};\Z)$ such that $N \cdot (s,t)=\partial_{\varphi\times\id_{X'}}(a,b)$. Then put
\begin{equation}
(h \times h')(s,t)
:=
\exp \Big( \frac{2\pi i}{N} \Big( \int_{(a,b)} (\curv,\cov)(h) \times \curv(h') - \langle c(h) \times c(h'),(a,b) \rangle \Big) \Big) \,. \label{eq:def_cross_2}
\end{equation}
These two cases uniquely determine the homomorphism $h \times h':Z_{k+k'-1}(\varphi \times \id_{X'};\Z) \to \Ul$. 
\end{definition}

Some comments on the notations in Definition~\ref{def:cross} are in order.
First of all, we write 
\begin{align*}
\langle c(h),(x,y)\rangle 
&= \langle c(h),[x,y]\rangle = \int_{(x,y)} (\curv,\cov)(h) \\
\mbox{and} \qquad 
\langle c(h'),y'\rangle 
&= \langle c(h'),[y']\rangle = \int_{y'} \curv(h')
\end{align*}
for the Kronecker pairing between (relative) cohomology and homology in \eqref{eq:def_cross_1}. 

Secondly, the term $\langle c(h) \times c(h'),(a,b) \rangle$ in \eqref{eq:def_cross_2} is not well-defined.
Replacing the cohomology class $c(h) \times c(h')$ by a cocycle $\mu \in C^{k+k'}(\varphi \times \id_{X'};\Z)$ representing it, the term $\exp \frac{2\pi i}{N} \langle \mu,(a,b)\rangle$ is independent of the choice of cocycle.
This is because $\frac{1}{N} \langle \delta_{\varphi\times\id_{X'}} \ell,(a,b) \rangle = \langle \ell,(s,t)\rangle \in \Z$ holds for any cochain $\ell \in C^{k+k'-1}(\varphi\times\id_{X'};\Z)$. 

Thirdly, the value of $h \times h'$ on a torsion cycle $(s,t)$ obtained from the K\"unneth splitting is independent of the choice of chain $(a,b)$ satisfying $N \cdot (s,t)=\partial_{\varphi \times \id_{X'}}(a,b)$.
For if we change $(a,b)$ by adding a cycle $(v',w') \in Z_{k+k'}(\varphi\times\id_{X'};\Z)$, the result in \eqref{eq:def_cross_2} changes by multiplication with 
$$
\exp \Big(\frac{2\pi i}{N} \Big(\underbrace{\int_{(v',w')} (\curv,\cov)(h) \times \curv(h') - \langle c(h) \times c(h'),(v',w')\rangle}_{=0} \Big)\Big) 
= 
1 \,.
$$

Finally, the K\"unneth complement $T_{k+k'-1}(\varphi\times\id_{X'};\Z) \subset Z_{k+k'-1}(\varphi\times\id_{X'};\Z)$ is the sum over $N \in \N$ of its subgroups of $N$-torsion cycles. 
This sum is of course not direct.
However, it is easy to see that the homomorphism $h \times h'$ in \eqref{eq:def_cross_2} is well-defined:
for a cycle $(s,t)$ in the complement choose $N'$ minimal such that $N' \cdot (s,t) = \partial_{\varphi\times\id_{X'}}(a,b)$.
Then the homology class $[s,t]$ has order $N'$ in $H_{k+k'-1}(\varphi\times\id_{X'};\Z)$ and all other possible choices of $N$ divide $N'$.
Thus the values in \eqref{eq:def_cross_2} for all such choices coincide.

\subsubsection{Well-definedness.}
Clearly, the map $h \times h':Z_{k+k'-1}(\varphi\times\id_{X'};\Z) \to \Ul$ defined by \eqref{eq:def_cross_1} and \eqref{eq:def_cross_2} is a homomorphism.
We show that it satisfies condition~\eqref{eq:def_rel_diff_charact_2} and thus defines a differential character in $\widehat H^{k+k'}(\varphi\times\id_{X'};\Z)$.

\begin{prop}\label{prop:h_times_h'}
Let $h \in \widehat H^k(\varphi;\Z)$ and $h' \in \widehat H^{k'}(X';\Z)$ be differential characters.
Then the homomorphism $h \times h':Z_{k+k'-1}(\varphi\times\id_{X'};\Z) \to \Ul$ in Definition~\ref{def:cross} is a differential character in $\widehat H^{k+k'}(\varphi \times \id_{X'};\Z)$ with $(\curv,\cov)(h \times h')=(\curv,\cov)(h) \times \curv(h')$.
\end{prop}

\begin{proof}
We check condition~\eqref{eq:def_rel_diff_charact_2} for the two cases separately.
Since the cross product is injective on cohomology, a cross product of cycles is a boundary if and only if one of the factors is a boundary. 
For $(x,y) = \partial_\varphi (a,b) \in Z_{k-1}(\varphi;\Z)$ and $y' \in Z_{k'}(X';\Z)$, we have:
\begin{align*}
(h \times h')(\partial_{\varphi \times \id_{X'}}((a,b) \times y')))
&=
(h \times h')(\partial_\varphi (a,b) \times y') \\
&\stackrel{\eqref{eq:def_cross_1}}{=}
(h \times h')(\partial_\varphi (a,b))^{\langle c(h'),y'\rangle} \\
&=
\exp \Big( 2\pi i \Big( \int_{(a,b)} (\curv(h),\cov(h)) \cdot \int_{y'} \curv(h') \Big) \Big) \\
&=
\exp \Big( 2\pi i \int_{(a,b) \times y'} (\curv,\cov)(h) \times \curv(h') \Big) \,.
\end{align*}
For $(x,y) \in Z_{k-1}(\varphi;\Z)$ and $y'=\partial b' \in Z_{k'}(X';\Z)$, we have:
\begin{align*}
(h \times h')(\partial_{\varphi \times \id_{X'}}((x,y) \times b')))
&= 
(h \times h')((-1)^{k'-1}(x,y) \times \partial b') \\
&\stackrel{\eqref{eq:def_cross_1}}{=}
h((-1)^{k'-1}(x,y))^{\langle c(h'),\partial b'\rangle} \\
&=
1 \\
&=
\exp \Big( 2\pi i \int_{(x,y) \times b'} (\curv,\cov)(h) \times \curv(h') \Big) \,.
\end{align*}
The last equality follows from the fact that the differential form $\curv(h') \in \Omega^{k'}_0(X')$ and the chain $b' \in C_{k'+1}$ have different degrees (and similarly for the other factor).

Similarly, for $(x,y) \in Z_k(\varphi;\Z)$ and $y' = \partial b' \in Z_{k'-1}(X';\Z)$, we have:
\begin{align*}
(h \times h')(\partial_{\varphi\times\id_{X'}} (x,y) \times b')
&=
(h \times h')((-1)^k(x,y) \times \partial b') \\
&\stackrel{\eqref{eq:def_cross_1}}{=}
h'(\partial b')^{\langle c(h),(x,y)\rangle} \\
&=
\exp \Big( 2 \pi i \int_{b'} \curv(h') \cdot \int_{(x,y)} (\curv,\cov)(h) \Big) \\
&=
\exp \Big( 2 \pi i \int_{(x,y) \times b'} (\curv,\cov)(h) \times \curv(h') \Big) \,. 
\end{align*}
Finally, for $(x,y) = \partial_{\varphi}(a,b) \in Z_k(\varphi;\Z)$ and $y' \in Z_{k'-1}(X';\Z)$, we have:
\begin{align*}
(h \times h')(\partial_{\varphi\times\id_{X'}} (a,b) \times y')
&=
(h \times h')(\partial_\varphi (a,b) \times \partial y') \\
&\stackrel{\eqref{eq:def_cross_1}}{=}
h'(\partial b')^{\langle c(h),\partial_\varphi (a,b)\rangle} \\
&=
1 \\
&=
\exp \Big( 2 \pi i \int_{(a,b) \times y'} (\curv,\cov)(h) \times \curv(h') \Big) \,. 
\end{align*}
If $(s,t) = \partial_{\varphi\times\id_{X'}}(v,w) \in T_{k+k'-1}(\varphi\times\id_{X'};\Z)$ is a boundary, we may choose $N=1$ in \eqref{eq:def_cross_2}.
This yields
\begin{align*}
(h \times h')(\partial_{\varphi \times \id_{X'}}&(v,w)) \\
&\stackrel{\eqref{eq:def_cross_2}}{=}
\exp \Big( 2\pi i \,\Big( \int_{(v,w)} (\curv,\cov)(h) \times \curv(h') - \underbrace{\langle c(h) \times c(h'),(v,w)\rangle}_{\in \Z} \Big) \Big) 
\\ 
&=
\exp \Big( 2\pi i \int_{(v,w)} (\curv,\cov)(h) \times \curv(h') \Big) \,.  
\end{align*}
Thus the homomorphism $h \times h':Z_{k+k'-1}(\varphi\times\id_{X'};\Z) \to \Ul$ is a relative differential character in $\widehat H^{k+k'}(\varphi \times \id_{X'};\Z)$ with $(\curv,\cov)(h \times h')=(\curv,\cov)(h) \times \curv(h')$.   
\end{proof}
 
\subsubsection{Naturality and compatibilities.}
We show that the cross product of relative and absolute differential characters is natural with respect to smooth maps.
Moreover, it is compatible with the structure maps (curvature, covariant derivative, characteristic class and topological trivializations) and with the homomorphisms $\vds$ and $\ti$ between absolute and relative characters groups.

\begin{thm}[Cross product: naturality and compatibilities]\label{thm:cross_prop}
The cross product between relative and absolute differential characters
$$
\times:\widehat H^k(\varphi;\Z) \times \widehat H^{k'}(X';\Z) \to \widehat H^{k+k'}(\varphi\times\id_{X'};\Z) \,, \quad (h,h') \mapsto h \times h' \,,
$$
is $\Z$-bilinear and associative with respect to absolute characters: for a relative character $h \in \widehat H^k(\varphi;\Z)$ and absolute characters $h' \in \widehat H^{k'}(X';\Z)$ and $h'' \in \widehat H^{k''}(X'';\Z)$, we have 
\begin{equation}
(h \times h') \times h''
=
h \times (h' \times h'') \in \widehat H^{k+k'+k''}(\varphi\times\id_{X' \times X''}) \,. \label{eq:cross_ass} 
\end{equation}
The cross product is natural: for smooth maps $(Y,B)\xrightarrow{(f,g)}(X,A)$ and $Y'\xrightarrow{f'}X'$, we have:
\begin{equation}
((f,g) \times f')^*(h \times h') 
=
(f,g)^*h \times {f'}^*h' \,. \label{eq:cross_nat}  
\end{equation}
Curvature, covariant derivative, characteristic class and topological trivializations are multiplicative:
\begin{align}
(\curv,\cov)(h \times h') 
&= 
(\curv,\cov)(h) \times \curv(h') \,. \label{eq:curvcov_mult_cross} \\
c(h \times h')
&=
c(h) \times c(h') \,. \label{eq:c_mult_cross} \\
\iota_\varphi(\omega,\vartheta) \times h'
&=
\iota_{\varphi\times\id_{X'}}((\omega,\vartheta) \times \curv(h')) \,, \label{eq:iota_mult_cross}
\end{align}
where $(\omega,\vartheta) \in \Omega^{k-1}(\varphi)$, $h' \in \widehat H^{k'}(X';\Z)$ and hence $(\omega,\vartheta) \times \curv(h') \in \Omega^{k+k'-1}(\varphi\times\id_{X'})$.

\noindent
The homomorphism $\ti$ is multiplicative: for $h \in \widehat H^{k-1}(A;\Z)$ and $h' \in \widehat H^{k'}(X';\Z)$, we have:\begin{align}
\ti_\varphi(h) \times h'
&=
\ti_{\varphi\times\id_{X'}}(h \times h') \,. \label{eq:ti_mult}
\intertext{Likewise, the map $\vds$ is multiplicative: for $h \in \widehat H^k(\varphi;\Z)$ and $h' \in \widehat H^{k'}(X';\Z)$, we have:} 
\vds_{\varphi\times\id_{X'}}(h \times h')
&=
\vds_\varphi(h) \times h' \,.  \label{eq:vds_mult}
\end{align} 
\end{thm}
\index{Theorem!Cross product: naturality and compatibilities}

\begin{proof}
The cross product is obviously $\Z$-bilinear.
Moreover, \eqref{eq:curvcov_mult_cross} was observed in the proof of Proposition~\ref{prop:h_times_h'}.
The other properties have to be checked now.

a)
We compute the characteristic class $c(h \times h')$.
Let $\tilde h$ and $\tilde h'$ be real lifts of $h$ and $h'$ and $\mu^{\tilde h}:= (\curv,\cov)(h) - \delta_\varphi \tilde h \in C^k(\varphi;\Z)$ and $\mu^{\tilde h'} := \curv(h') - \delta \tilde h' \in C^{k'}(X';\Z)$ be the corresponding cocyles for the characteristic classes.
We first compute a real lift $\widetilde{h \times h'}$ for the character $h \times h'$.

We have the K\"unneth splitting:\footnote{More precisely, the first factor is the image of $\big(Z_*(\varphi;\Z) \otimes Z_*(X';\Z)\big)_{k+k'-1} \xrightarrow{\,\times\,} Z_{k+k'-1}(\varphi\times\id_{X'};\Z)$. 
Therefor, we write cross products instead of tensor products for the real lifts on this factor.} 
\begin{equation}
Z_{k+k'-1}(\varphi\times\id_{X'};\Z)
= 
\big(Z_*(\varphi;\Z) \otimes Z_*(X';\Z)\big)_{k+k'-1} \oplus T_{k+k'-1}(\varphi\times\id_{X'};\Z) \,. \label{eq:split_Z_kk}
\end{equation}
On the first factor we obtain from \eqref{eq:def_cross_1} the real lift 
\begin{equation*}
\tilde h \times \mu^{\tilde h'} + (-1)^{k} (\curv,\cov)(h) \times \tilde h' \,. 
\end{equation*}
On $N$-torsion cycles in the second factor we obtain from \eqref{eq:def_cross_2} the real lift
\begin{equation*}
\frac{1}{N} \big( (\curv,\cov)(h) \times \curv(h') - \mu^{\tilde h} \times \mu^{\tilde h'} \big) \circ \partial_{\varphi\times\id_{X'}}^{-1} \circ \big(N \cdot (\cdot) \big)\,.
\end{equation*}
In particular, the coboundary of the lift on the second factor is given by 
\begin{equation}
\delta_{\varphi\times\id_{X'}} \big( \widetilde{h \times h'}|_{T_{k+k'-1}(\varphi\times\id_{X'};\Z)} \big)
= (\curv,\cov)(h) \times \curv(h') - \mu^{\tilde h} \times \mu^{\tilde h'} \,. \label{eq:delta_T}
\end{equation}
Now we compute the cocycle $\mu^{\widetilde{h \times h'}} \in C^{k+k'}(\varphi\times\id_{X'};\Z)$ that represents the characteristic class $c(h \times h') \in H^{k+k'}(\varphi\times\id_{X'};\Z)$.
We use the decomposition
\begin{align}\label{eq:split_C_kk}
C_{k+k'}&(\varphi\times\id_{X'};\Z) \notag \\
&=
Z_{k+k'}(\varphi\times\id_{X'};\Z) \oplus \im \Big( C_{k+k'}(\varphi\times\id_{X'};\Z) \xleftarrow{s_{\varphi\times\id_{X'}}} B_{k+k'-1}(\varphi\times\id_{X'};\Z) \Big)
\end{align}
obtained from a splitting of the boundary map $\partial_{\varphi\times\id_{X'}}$.
On the first factor in \eqref{eq:split_C_kk}, we have 
\begin{align*}
\mu^{\widetilde{h \times h'}}\big|_{Z_{k+k'}(\varphi\times\id_{X'};\Z)}
&:= (\curv,\cov)(h \times h') - \delta_{\varphi\times\id_{X'}}\widetilde{h \times h'} \\
&= (\curv,\cov)(h) \times \curv(h') \\
&= \big(\mu^{\tilde h} \times \mu^{\tilde h'}\big)\big|_{Z_{k+k'}(\varphi\times\id_{X'};\Z)}
\end{align*}
The second factor in \eqref{eq:split_C_kk} inherits a further splitting from \eqref{eq:split_Z_kk}.
With respect to this splitting, we obtain:
\begin{align*}
\mu^{\widetilde{h \times h'}} 
&:= 
(\curv,\cov)(h \times h') - \delta_{\varphi\times\id_{X'}}\widetilde{h \times h'} \\
&\stackrel{\eqref{eq:delta_T}}{=}
(\curv,\cov)(h) \times \curv(h') \\
& \quad - \delta_{\varphi\times\id_{X'}} \Big(\tilde h \times \mu^{\tilde h'} + (-1)^{k} (\curv,\cov)(h) \times \tilde h'\Big) \oplus
\Big( (\curv,\cov)(h \times h') - \mu^{\tilde h} \times \mu^{\tilde h'} \Big) \\
&=
\Big( (\curv,\cov)(h) \times \curv(h') - \delta_\varphi \tilde h \times \mu^{\tilde h'} - (\curv,\cov)(h) \times \delta \tilde h' \Big) \oplus \Big(\mu^{\tilde h}  \times \mu^{\tilde h'} \Big) \\
&= \Big( (\curv,\cov)(h) \times \mu^{\tilde h'} - \delta_\varphi \tilde h \times \mu^{\tilde h'}\Big) \oplus \Big(\mu^{\tilde h}  \times \mu^{\tilde h'} \Big) \\
&= \mu^{\tilde h} \times \mu^{\tilde h'} \,.
\end{align*}
In conclusion, we have $\mu^{\widetilde{h \times h'}} = \mu^{\tilde h} \times \mu^{\tilde h'}$ and thus \eqref{eq:c_mult_cross} holds.

b)
Next we prove associativity.
To apply Definition~\ref{def:cross} we need to first derive an appropriate K\"unneth splitting of $Z_{k+k'+k''-1}(\varphi\times\id_{X'\times X''};\Z)$.
The cross product of cycles and the classical Alexander-Whitney and Eilenberg-Zilber maps are associative.
This implies that the induced Alexander-Whitney and Eilenberg-Zilber maps for mapping cone complexes are also associative.
More explicitly, we have the following commutative diagram:
\begin{equation*}
\xymatrix{
Z_*(\varphi;\Z) \otimes Z_*(X';\Z) \otimes Z_*(X'';\Z) \ar@<2pt>[rrr]^{\times \otimes \id} \ar@<2pt>[dd]^{\id\otimes \times} &&& Z_*(\varphi\times\id_{X'};\Z) \otimes Z_*(X'';\Z) \ar@{-->}@<2pt>[lll]^{S_\varphi\otimes\id} \ar@<2pt>[dd]^{\times} \\
&&& \\
Z_*(\varphi;\Z) \otimes Z_*(X'\times X'';\Z) \ar@<2pt>[rrr]^{\times} \ar@{-->}@<2pt>[uu]^{\id\otimes S} 
&&& Z_*(\varphi\times\id_{X'\times X''};\Z) \ar@{-->}@<2pt>[uu]^{S_{\varphi\times\id_{X'}\times\id_{X''}}} \ar@{-->}@<2pt>[lll]^{S_{\varphi\times\id_{X \times X'}}} \ar@{-->}[uulll]^{\bf S} \,.
}
\end{equation*}
The induced splitting ${\bf S}:Z_*(\varphi\times\id_{X'\times X''};\Z) \to Z_*(\varphi;\Z) \otimes Z_*(X';\Z) \otimes Z_*(X'';\Z)$ of the concatenation $\times \circ (\id \otimes \times) = \times \circ (\times \otimes \id)$ yields a direct sum decomposition
\begin{equation}
Z_*(\varphi\times\id_{X'\times X''};\Z)
=
Z_*(\varphi;\Z) \otimes Z_*(X';\Z) \otimes Z_*(X'';\Z) \oplus \ker ({\bf S}) \,.
\end{equation}
By the relative K\"unneth theorem, the cycles in $\ker (\bf S)$ represent torsion classes in $H_{k+k'+k''}(\varphi\times\id_{X'\times X''};\Z)$.

The cross products of relative and absolute differential forms and cohomology classes are associative.
Thus for a relative differential character $h \in \widehat H^k(\varphi;\Z)$ and absolute characters $h' \in \widehat H^{k'}(X';\Z)$ and $h'' \in \widehat H^{k''}(X'';\Z)$, the cross products $(h \times h') \times h''$ and $h \times (h' \times h'')$ have the same curvature, covariant derivative and characteristic class.
Hence they coincide on torsion cycles, and in particular on cycles in $\ker (\bf S)$.

Now we compare the two characters $(h \times h') \times h''$ and $h \times (h' \times h'')$ on cross products $(x,y) \times y' \times y''$, where $(x,y) \in Z_i(\varphi;\Z)$ and $y' \in Z_j(X';\Z)$ and $y'' \in Z_l(X'';\Z)$.
By \eqref{eq:def_cross_1} both characters vanish on cross products with $(i,j,l)$ different from $(k-1,k',k'')$, $(k,k'-1,k'')$ and $(k,k',k''-1)$.
Now we compute the remaining cases.
For $(i,j,l) = (k-1,k',k'')$, we have:
\begin{align*}
((h \times h') \times h'')((x,y)\times y' \times y'')
&=
(h \times h')((x,y)\times y')^{\langle c(h''),y''\rangle} \\
&=
h(x,y)^{\langle c(h'),y'\rangle \cdot \langle c(h''),y''\rangle} \\
&\stackrel{\eqref{eq:cx_commute}}{=}
h(x,y)^{\langle c(h'\times h''),y'\times y''\rangle} \\
&=
(h \times (h' \times h''))((x,y)\times y' \times y'')  \,.
\intertext{Similarly, for $(i,j,l) = (k-1,k',k'')$, we have:} 
((h \times h') \times h'')((x,y)\times y' \times y'')
&=
(h \times h')((x,y)\times y')^{\langle c(h''),y''\rangle} \\
&=
h'(y')^{(-1)^k \cdot \langle c(h),(x,y)\rangle \cdot \langle c(h''),y''\rangle} \\
&=
(h' \times h'')(y'\times y'')^{(-1)^k \cdot \langle c(h),(x,y)\rangle} \\
&=
(h \times (h'\times h''))((x,y) \times y' \times y'') \,.
\intertext{Finally, for $(i,j,l) = (k,k',k''-1)$, we have:}
((h \times h') \times h'')((x,y)\times y' \times y'')
&=
h''(y'')^{(-1)^{k+k'} \cdot \langle c(h \times h'),(x,y) \times y'\rangle} \\
&\stackrel{\eqref{eq:c_mult_cross}}{=}
h''(y'')^{(-1)^k \cdot \langle c(h),(x,y) \rangle \cdot (-1)^{k'} \cdot \langle c(h'),y'\rangle} \\
&=
(h \times (h'\times h''))((x,y) \times y' \times y'') \,.
\end{align*}
Thus $(h \times h') \times h'' = h \times (h' \times h'')$.
 
c)
Now we consider topological trivializations. 
Let $(\omega,\vartheta) \in \Omega^{k-1}(\varphi)$.
We compare the characters $\iota_\varphi(\omega,\vartheta) \times h'$ and $\iota_{\varphi\times\id_{X'}}((\omega,\vartheta)\times \curv(h'))$.
For curvature and covariant derivative, we have:
\begin{align*} 
(\curv,\cov)(\iota_\varphi(\omega,\vartheta) \times h') 
&=
(\curv,\cov)(\iota_\varphi(\omega,\vartheta)) \times \curv(h') \\ 
&=
d_\varphi (\omega,\vartheta) \times \curv(h') \\
&=
d_{\varphi\times\id_{X'}} \big((\omega,\vartheta) \times \curv(h')\big) \\
&=
(\curv,\cov)\big(\iota_{\varphi\times\id_{X'}}((\omega,\vartheta)\times \curv(h'))\big) \,.
\end{align*}
For the characteristic class, we have:
$$
c(\iota_\varphi(\omega,\vartheta) \times h')
=
\underbrace{c(\iota_\varphi(\omega,\vartheta))}_{=0} \times c(h') 
=
0
=
c\big(\iota_{\varphi\times\id_{X'}}((\omega,\vartheta)\times \curv(h'))\big) \,.
$$
By \eqref{eq:def_cross_2} the characters $\iota_\varphi(\omega,\vartheta) \times h'$ and $\iota_{\varphi\times\id_{X'}}((\omega,\vartheta)\times \curv(h'))$ thus  coincide on the factor $T_{k+k'-1}(\varphi\times\id_{X'};\Z)$ in $Z_{k+k'-1}(\varphi\times\id_{X'};\Z)$.

Let $(s,t)=(x,y) \times y'$ be a cross product of cycles $(x,y) \in Z_i(\varphi;\Z)$ and $y' \in Z_j(X';\Z)$.
By \eqref{eq:def_cross_2}, we have $(\iota_\varphi(\omega,\vartheta)\times h')((x,y) \times y')=1$ if $(i,j) \neq (k-1,k')$. 
The same holds for the character $\iota_{\varphi\times\id_{X'}}((\omega,\vartheta)\times \curv(h'))$, since the differential form $(\omega,\vartheta)\times \curv(h')$ vanishes upon integration over cross products of cycles of degrees different from $(k-1,k')$.
For $(i,j)=(k-1,k')$, we have:
\begin{align*}
(\iota_\varphi(\omega,\vartheta) \times h')((x,y)\times y')
&=
(\iota_\varphi(\omega,\vartheta)(x,y))^{\langle c(h'),y'\rangle} \\
&=
\exp \Big( 2\pi i \int_{(x,y)} (\omega,\vartheta) \int_{y'} \curv(h') \Big) \\
&=
\iota_{\varphi\times\id_{X'}}((\omega,\vartheta)\times \curv(h'))((x,y) \times y') \,.
\end{align*}
This proves \eqref{eq:iota_mult_cross}.

d)
Now we prove naturality.
Let $\psi:Y \to B$ and $f':Y'\to X'$ be smooth maps.
Let $(Y,B)\xrightarrow{(f,g)}(X,A)$ be a smooth map.
Let $h \in \widehat H^k(\varphi;\Z)$ and $h' \in \widehat H^{k'}(X';\Z)$.
The relative classical Alexander-Whitney and Eilenberg-Zilber maps are natural with respect to smooth maps.
Thus so are the K\"unneth splittings \eqref{eq:split_Z_kk} constructed in the appendix.
More explicity, the map induced by $(Y,B)\times Y'\xrightarrow{(f,g) \times f'}(X,A)\times X'$ maps the splitting
\begin{align*}
Z_{k+k'-1}(\psi\times\id_{Y'};\Z)
&=
\big( Z_*(\psi;\Z) \otimes Z_*(Y';\Z) \big)_{k+k'-1} \oplus T_{k+k'-1}(\psi\times\id_{Y'};\Z) 
\intertext{to the splitting}
Z_{k+k'-1}(\varphi\times\id_{X'};\Z)
&=
\big( Z_*(\varphi;\Z) \otimes Z_*(X';\Z) \big)_{k+k'-1} \oplus T_{k+k'-1}(\varphi\times\id_{X'};\Z) \,.
\end{align*}
Since curvature, covariant derivative and characteristic class are natural, by \eqref{eq:def_cross_2} we have for any cycle $(s,t) \in T_{k+k'-1}(\psi\times\id_{Y'};\Z)$:
\begin{align*}
\big(((f,g)\times f')^*(h \times h')\big)(s,t)
&=
(h \times h')(((f,g)\times f')_*(s,t)) \\
&=
((f,g)^*h \times {f'}^*h')(s,t) \,.
\intertext{Similarly, for cross product cycles $(s,t) = (x,y) \times y'$ with $(x,y) \in Z_i(\psi;\Z)$ and $y' \in Z_j(Y';\Z)$, we obtain from \eqref{eq:def_cross_2}:}
\big(((f,g)\times f')^*(h \times h')\big)((x,y) \times y')
&=
(h \times h')\big(((f,g)\times f')_*((x,y) \times y')\big) \\
&=
(h \times h')((f,g)_* (x,y)\times f'_* y')) \\
&=
((f,g)^*h \times {f'}^*h')(s,t) \,.
\end{align*}

e)
Finally we consider compatibility with the maps $\ti$ and $\vds$ that relate absolute and relative differential characters groups. 
For $h \in \widehat H^{k-1}(A;\Z)$ and $h' \in \widehat H^{k'}(X';\Z)$, we have:
\begin{align*}
(\curv,\cov)(\ti_\varphi(h) \times h')
&=
(0,-\curv(h) \times \curv(h')) \\
&=
(0,-\curv(h \times h')) \\
&=
(\curv,\cov)(\ti_{\varphi\times\id_{X'}}(h \times h')) \,.
\end{align*}
Moreover, the characteristic classes of both $\ti_\varphi(h) \times h'$ and $\ti_{\varphi\times\id_{X'}}(h \times h')$ equal the image of the class $c(h \times h')$ under the map $H^{k+k'-1}(A \times X';\Z) \to H^{k+k'}(\varphi \times \id_{X'};\Z)$.
Thus the characters $\ti_\varphi(h) \times h'$ and $\ti_{\varphi\times\id_{X'}}(h \times h')$ coincide on the second factor in \eqref{eq:split_Z_kk}.

Let $(s,t) = (x,y) \times y'$ be a cross product of cycles $(x,y) \in Z_i(\varphi;\Z)$ and $y' \in Z_j(X';\Z)$.
By \eqref{eq:xx} and \eqref{eq:def_cross_1} both $\ti_\varphi(h) \times h'$ and $\ti_{\varphi\times\id_{X'}}(h \times h')$ vanish on cross products with $(k-1,k') \neq (i,j) \neq (k,k'-1)$.
For $(i,j) = (k-1,k')$ or $(i,j) = (k,k'-1)$, we have:
$$
(\ti(h) \times h')((x,y) \times y')
=
(h \times h')(y \times y') 
=
\ti(h \times h')((x,y) \times y') \,.
$$
Thus \eqref{eq:ti_mult} holds.
The proof of \eqref{eq:vds_mult} is completely analogous.
\end{proof}

\subsubsection{The module structure.}
As is well-known, the cup product provides relative cohomology with the structure of a right module over the absolute cohomology ring.
In the same way, the mapping cone cohomology $H^*(\varphi;\Z)$ of a (smooth) map $\varphi:A \to X$ is a right module over the cohomology ring $H^*(X;\Z)$. 
Similarly, we have an internal product on the mapping cone de Rham complex $\Omega^*(\varphi)$ defined by
\begin{equation}\label{eq:ext_prod_forms}
(\omega,\vartheta) \wedge \omega' 
:=
(\omega \wedge \omega',\vartheta \wedge \varphi^*\omega') \,, 
\end{equation}
where $(\omega,\vartheta) \in \Omega^*(\varphi)$ and $\omega' \in \Omega^*(X)$.
Thus the abelian group $\Omega^*(\varphi)$ is a right module over the ring $\Omega^*(X)$ of differential forms on $X$.

From the external product between relative and absolute differential characters we derive an internal product by pull-back along a version of the diagonal map.
By the analogue of Theorem~\ref{thm:cross_prop}, the internal product provides the graded abelian group $\widehat H^*(\varphi;\Z)$ with a natural structure of a right module over the ring $\widehat H^*(X;\Z)$ such that the structure maps (curvature, covariant derivative and characteristic class) become ring homomorphisms. 

Denote by $\Delta_X:X \to X \times X$, $x\mapsto (x,x)$, the diagonal map, and similarly for $A$.
\index{diagonal map}
\index{+DeltaX@$\Delta_X$, diagonal map}
Let $\varphi:A \to X$ be a smooth map and $\varphi \times \id_{X}:A \times X \to X \times X$.
Similar to the diagonal map, let
\index{+DeltaXA@$\Delta_{(X,A)}$, relative diagonal map}
\begin{align*}
\Delta_{(X,A)} := (\Delta_X, (\id_A \times \varphi) \circ \Delta_A):(X,A) &\to (X,A) \times X = (X \times X,A \times X) \,, \\
(x,a) &\mapsto ((x,x),(a,\varphi(a))) \,.   
\end{align*}
Since the external product of relative and absolute differential characters is natural, we may use $\Delta_{(X,A)}$ to pull-back cross products $h \times h' \in \widehat H^{k+k'}(\varphi\times\id_{X'};\Z)$ to $(X,A)$.

\begin{definition}[Internal product]
Let $h \in \widehat H^k(\varphi;\Z)$ and $h' \in \widehat H^{k'}(X;\Z)$ be differential characters.
Their internal product the character $h*h' \in \widehat H^{k+k'}(\varphi;\Z)$, defined by
\begin{equation*}
h * h'
:=
\Delta_{(X,A)}^*(h \times h')  \,.
\end{equation*}  
\end{definition}
\index{internal product!of relative and absolute characters}
\index{product!internal}
\index{relative differential character!internal product}

The properties of the external product proved in Theorem~\ref{thm:cross_prop} directly translate into properties of the internal product.
This establishes the module structure on $\widehat H^*(\varphi;\Z)$ and turns the curvature, covariant derivative and characteristic into module homomorphisms:
\index{module structure!on relative differential cohomology}
\index{relative differential character!module structure}

\begin{thm}[Module structure: naturality and compatibilities]\label{cor:int_prop}
The internal product between relative and absolute differential characters
$$
*:\widehat H^k(\varphi;\Z) \times \widehat H^{k'}(X';\Z) \to \widehat H^{k+k'}(\varphi\times\id_{X'};\Z) \,, \quad (h,h') \mapsto h \times h' \,,
$$
is $\Z$-bilinear and associative with respect to absolute characters: for a relative character $h \in \widehat H^k(\varphi;\Z)$ and absolute characters $h' \in \widehat H^{k'}(X;\Z)$ and $h'' \in \widehat H^{k''}(X;\Z)$, we have 
\begin{equation*}
(h * h') * h''
=
h*(h'*h'') \,. \label{eq:*_ass}
\end{equation*}
The internal product is natural: for a smooth map $(Y,B)\xrightarrow{(f,g)}(X,A)$ we have:
\begin{equation*}
(f,g)^*(h * h') 
=
(f,g)^*h * f^*h' \,. \label{eq:*_nat}  
\end{equation*}
Curvature, covariant derivative, characteristic class and topological trivializations are multiplicative:
\begin{align*}
(\curv,\cov)(h * h') 
&= 
(\curv,\cov)(h) \wedge \curv(h') \,. \\
c(h * h')
&=
c(h) \cup c(h') \,.  \\
\iota_\varphi(\omega,\vartheta) * h'
&=
\iota_\varphi((\omega,\vartheta) \wedge \curv(h')) \,,
\end{align*}
where $(\omega,\vartheta) \in \Omega^{k-1}(\varphi)$.

\noindent
The group homomorphism $\ti_\varphi: \widehat H^{*-1}(A;\Z) \to \widehat H^*(\varphi;\Z)$ is multiplicative: for characters $h \in \widehat H^{k-1}(A;\Z)$ and $h' \in \widehat H^{k'}(X';\Z)$, we have:
\begin{align}
\ti_\varphi(h) * h'
&=
\ti_\varphi(h * \varphi^*h') \,. \label{eq:ti_mult_*}
\intertext{Likewise, the homomorphism $\vds: \widehat H^*(\varphi;\Z) \to \widehat H^*(X;\Z)$ is multiplicative: for characters $h \in \widehat H^k(\varphi;\Z)$ and $h' \in \widehat H^{k'}(X';\Z)$, we have:} 
\vds_\varphi(h * h')
&=
\vds_\varphi(h) * h' \,.  \label{eq:vds_mult_*}
\end{align} 
\end{thm}
\index{Theorem!Module structure}

\begin{proof}
To prove associativity, we need to keep track of the various pull-backs:
\begin{align*}
(h*h')*h''
&=
\Delta_{(X,A)}^*\Big( \big(\Delta_{(X,A)}^*h\times h'\big) \times h''\Big) \\
&\stackrel{\eqref{eq:cross_nat}}{=}
\Delta_{(X,A)}^* \Big( \big(\Delta_{(X,A)} \times \id_{X}\big)^*(h \times h') \times h'' \Big) \\
&=
\Big( (\Delta_{(X,A)} \times \id_{X}) \circ \Delta_{(X,A)}\Big)^*\big((h \times h') \times h'')\big) \\
&\stackrel{\eqref{eq:cross_ass}}{=}
\Big( (\id_{(X,A)} \times \Delta_X) \circ \Delta_{(X,A)}\Big)^*\big(h \times (h' \times h'')\big) \\
&\stackrel{\eqref{eq:cross_nat}}{=}
\Delta_{(X,A)}^*\big( h \times \Delta_X^*(h' \times h'')\big) \\
&=
h*(h'*h'') \,.
\end{align*}
In the third last equation we used equality of the maps
\begin{align*}
(\Delta_{(X,A)} \times \id_{X}) \circ \Delta_{(X,A)}
= 
(\id_{(X,A)} \times \Delta_X) \circ \Delta_{(X,A)}: (X,A) &\to (X,A) \times (X \times X) \,, \\ 
(x,a) \mapsto ((x,x,x),(a,\varphi(a),\varphi(a))) \,. 
\end{align*}
Naturality of the internal product follows from naturality of the cross product together with the equality of maps
$$
\Delta_{(X,A)} \circ (f,g) =((f,g)\times f) \circ \Delta_{(Y,B)}
: (Y,B) \to (X,A) \times X \,.
$$
Thus for characters $h \in \widehat H^k(\varphi;\Z)$ and $h' \in \widehat H^{k'}(X';\Z)$, we have:
\begin{align*}
(f,g)^*h* h'
&=
(f,g)^*\Delta_X^*(h \times h') \\
&=
(\Delta_X \circ (f,g))^*(h \times h') \\
&=
((f,g)\times f) \circ \Delta_Y)^*(h \times h') \\
&\stackrel{\eqref{eq:cross_nat}}{=}
\Delta_Y^*((f,g)^*h \times f^*h') \\
&=
(f,g)^*h * f^*h' \,.
\end{align*}
For curvature and covariant derivative, we have:
\begin{align*}
(\curv,\cov)(h * h')
&\stackrel{\eqref{eq:curvcov_mult_cross}}{=}
(\curv,\cov)(\Delta_{(X,A)}^* (h \times h')) \\
&=
\Delta_{X,A}^* ((\curv,\cov)(h) \times \curv(h')) \\
&=
(\Delta_X^* \curv(h) \times \curv(h'), \Delta_A^* (\id_A \times \varphi)^* \cov(h) \times \curv(h') ) \\
&\stackrel{\eqref{eq:cross_nat}}{=}
(\curv(h) \wedge \curv(h'), \cov(h) \wedge \varphi^*\curv(h') ) \\
&\stackrel{\eqref{eq:ext_prod_forms}}{=}
(\curv,\cov)(h) \wedge \curv(h') \,.
\end{align*}
Likewise, for topological trivializations we have:
\begin{align*}
\iota_\varphi(\omega,\vartheta) * h'
&=
\Delta_{(X,A)}^*(\iota_\varphi(\omega,\vartheta) \times h') \\
&\stackrel{\eqref{eq:iota_mult_cross}}{=}
\Delta_{(X,A)}^*\iota_\varphi((\omega,\vartheta) \times \curv(h')) \\
&=
\iota_\varphi((\omega,\vartheta) \wedge \curv(h')) \,.
\end{align*}
Multiplicativity of the characteristic class follows from \eqref{eq:c_mult_cross} and the fact that the cup product is the pull-back along $\Delta_{(X,A)}$ of the cross product.

It remains to prove multiplicativity of the homomorphisms $\ti$ and $\vds$.
For characters $h \in \widehat H^{k-1}(A;\Z)$ and $h' \in \widehat H^{k'}(X';\Z)$, and a cycle $(s,t) \in Z_{k+k'-1}(\varphi;\Z)$ we have:
\begin{align*}
(\ti_\varphi(h) * h')(s,t)
&=
\big(\Delta_{(X,A)}^*(\ti_\varphi(h) \times h')\big)(s,t) \\
&\stackrel{\eqref{eq:ti_mult}}{=}
\big(\Delta_{(X,A)}^*\ti_\varphi(h \times h')\big)(s,t) \\
&=
(h \times h')\big((\id_A \times \varphi) \circ \Delta_A\big)_*(t) \\
&\stackrel{\eqref{eq:cross_nat}}{=}
\Delta_A^*(h \times \varphi^*h')(t) \\
&=
\ti_\varphi(h*\varphi^*h')(s,t) \,.
\intertext{Likewise, for characters $h \in \widehat H^k(\varphi;\Z)$ and $h' \in \widehat H^{k'}(X';\Z)$ and a cycle $z \in Z_{k+k'-1}(X;\Z)$ we have:}
(\vds_\varphi(h) * h')(z)
&=
\Delta_X^*(\vds_\varphi(h) \times h')(z) \\
&\stackrel{\eqref{eq:vds_mult}}{=}
(h \times h')({\Delta_X}_*z,0) \\
&=
\big(\Delta_{(X,A)}^*(h \times h')\big)(z,0) \\
&=
\vds_\varphi(h * h')(z) \,. \qedhere
\end{align*}
\end{proof}

\subsubsection{Uniqueness of the cross product and module structure.}
In \cite[Ch.~6]{BB13} we have shown uniqueness of the external and internal product between absolute differential characters. 
This in particular implies uniqueness of the ring structure on differential cohomology.
The proof starts from an axiomatic definition of the cross product. 
The axioms essentially coincide with the properties in Theorem~\ref{thm:cross_prop} (for the absolute case).

The methods of proof used in \cite[Ch.~6]{BB13} directly apply to the external product between relative and absolute differential characters defined in the present work.
Thus we could have defined the external product axiomatically by the properties in Theorem~\ref{thm:cross_prop}.
Then we could have derived the formulae \eqref{eq:def_cross_1} and \eqref{eq:def_cross_2} from this axiomatic decription.
This would have proved uniqueness of the external and internal product and hence uniqueness of the right $\widehat H^*(X;\Z)$-module structure on the relative differential cohomology $\widehat H^*(\varphi;\Z)$ of a smooth map $\varphi:A \to X$. 

Thus we note without explicit proof here:
\begin{cor}
The relative differential cohomology $\widehat H^*(\varphi;\Z)$ of a smooth map $\varphi:A \to X$ carries the structure of a right module over the ring $\widehat H^*(X;\Z)$. 
The module structure is uniquely determined by the properties in Theorem~\ref{cor:int_prop}.
\end{cor}
\index{uniqueness!of cross product and module structure}

\subsubsection{The module structure on parallel characters.}
In Theorem~\ref{thm:rel_rel} we have shown that the graded abelian group $\widehat H^*(X,A;\Z)$ defined by characters on the groups of relative cycles coincides with the subgroup of parallel characters in $\widehat H^*(i_A;\Z)$, where $i_A:A \to X$ is the embedding of a smooth submanifold.
By Theorems~\ref{thm:cross_prop} and \ref{cor:int_prop}, the external and internal products of relative and absolute differential characters are multiplicative with respect to the covariant derivative.
Thus products of flat characters with absolute characters are again flat characters.
In other words, we have:
\index{module structure!on relative differential cohomology}
\index{relative differential character!module structure}

\begin{cor}
Let $i_A:A \to X$ be the inclusion of a smooth submanifold.
Then there exist unique natural internal and external products
\begin{align*}
\times: \widehat H^k(X,A;\Z) \times \widehat H^{k'}(X';\Z) &\to \widehat H^{k+k'}(X \times X',A \times X';\Z) \,, \quad (h,h') \mapsto h \times h' \,, \\
*: \widehat H^k(X,A;\Z) \times \widehat H^{k'}(X;\Z) &\to \widehat H^{k+k'}(X,A;\Z) \,, \quad (h,h') \mapsto h * h' \,, 
\end{align*}
such that curvature, characteristic class and topological trivializations are multiplicative.
Moreover, the products are associative with respect to absolute characters.

In particular, the graded abelian group $\widehat H^*(X,A;\Z)$ carries a unique structure of a right module over the ring $\widehat H^*(X;\Z)$. 
\end{cor}
 
\begin{remark}
By the identification of the relative Hopkins-Singer group $\check{H}^*(\varphi;\Z)$ as a subgroup of the group $\widehat H^*(\varphi;\Z)$ of relative differential characters, we have induced external and internal products
\begin{align*}
\times: \check H^k(\varphi;\Z) \times \widehat H^{k'}(X';\Z) &\to \check H^{k+k'}(\varphi \times \id_{X'};\Z) \,, \quad (h,h') \mapsto h \times h' \,, \\
*: \check H^k(\varphi;\Z) \times \widehat H^{k'}(X;\Z) &\to \check H^{k+k'}(\varphi;\Z) \,, \quad (h,h') \mapsto h * h' \,, 
\end{align*}
This is well-defined, since for $h \in \check H^k(\varphi;\Z) \subset \widehat H^k(\varphi;\Z)$ and $h' \in \widehat H^{k'}(X';\Z)$ we have $\cov(h \times h') = \cov(h) \times \varphi^*\curv(h') \in \Omega^{k+k'-1}_0(A \times X')$.
Hence $h \times h'$ lies in the subgroup $\check{H}^{k+k'}(\varphi;\Z) \subset \widehat H^{k+k'}(\varphi;\Z)$ of characters with covariant derivative a closed form with integral periods.  

This yields a right $\widehat H^*(X;\Z)$-module structure also on $\check{H}^*(\varphi;\Z)$. 
\index{module structure!on relative Hopkins-Singer group}
\index{relative Hopkins-Singer group!module structure}
\end{remark}

\section{Fiber integration and transgression}\label{sec:fiber_int_trans}
Let $\pi:E \to X$ be a fiber bundle with closed oriented fibers.
There are natural fiber integration maps $\fint_F: \Omega^k(E) \to \Omega^{k-\dim F}(X)$ for differential forms and $\pi_!:H^k(E;\Z) \to H^{k-\dim F}(X;\Z)$ for integral cohomology.
Thus it is natural to expect that there exists a fiber integration map $\widehat H^*(E;\Z) \to \widehat H^{k-\dim F}(X;\Z)$ that induces the well-known maps on the curvature and charactristic class.

Such fiber integration maps have been constructed for several models of differential cohomology, see \cite{HS05}, \cite{BKS10} for differential cocycles, \cite{DL05}, \cite{L06} for simplicial forms, \cite{HLZ03} for de Rham-Feder currents and \cite{BB13} for the original model of differential characters.
In \cite[Ch.~7]{BB13} we proved that fiber integration is uniquely determined by the requirements to be compatible with pull-back diagrams and with fiber integration for differential forms (i.e.~with curvature and topological trivializations).
The proof is constructive in that it yields an explicit formula for the fiber integration map.
In particular, the various constructions in the different models for differential cohomology yield the same fiber integration map.

In this section we use the method from \cite{BB13} to construct fiber integration and transgression maps for relative differential characters.
In particular, we make use of the pull-back operation for geometric relative cycles and the transfer maps constructed in Sections~\ref{sec:PB} and \ref{sec:transfer}.
We show that fiber integration for relative characters is compatible via the homomorphisms $\ti$ and $\vds$ with fiber integration for absolute characters.
As a corollary, we obtain fiber integration and transgression maps for parallel characters.
Moreover, fiber integration and transgression commute with the long exact sequences \eqref{eq:long_ex_sequ} and \eqref{eq:long_ex_sequ_par}. 

\subsection{Fiber integration}\label{subsec_fiber_int}
Let $\pi:E \to X$ be a fiber bundle and $\varphi:A \to X$ a smooth map.
We have the pull-back diagram
\begin{equation*}
\xymatrix{
\varphi^*E \ar^\Phi[r] \ar_\pi[d] & E \ar^\pi[d] \\
A \ar_\varphi[r] & X \,.
}
\end{equation*}
In the following we construct fiber integration for relative differential characters.
We discuss its compatibility with curvature, covariant derivative, topological trivializations and characteristic class and with fiber integration for absolute differential characters.

\subsubsection{Construction of the fiber integration map.}
For convenience of the reader, we recall the formula for fiber integration of (absolute) differential characters obtained in \cite[Ch.~7]{BB13}.
For a differential character $h \in \widehat H^k(E;\Z)$ and a smooth singular cycle $z \in Z_{k-\dim F-1}(X;\Z)$, we have:
\index{differential character!fiber integration}
\begin{equation}\label{eq:FI_absolute}
\widehat\pi_! h(z)
=
h(\lambda(z)) \cdot \exp \Big( 2\pi i \int_{a(z)} \fint_F \curv(h) \Big) \,. 
\end{equation}
We now adapt this formula to relative differential characters.  

Let $h \in \widehat H^k(\Phi;\Z)$.
To evaluate the integrated character $\widehat \pi_!h \in \widehat H^{k-\dim F}(\varphi;\Z)$ on a cycle $(s,t) \in Z_{k-\dim F-1}(\varphi;\Z)$, we use the homomorphism $(a,b)_\varphi$ from Section~\ref{sec:PB} and the transfer map $\lambda_\varphi$ defined in Section~\ref{sec:transfer}.

\begin{definition}
Let $\varphi:A \to X$ be a smooth map.
Let $F \hookrightarrow E \stackrel{\pi}{\twoheadrightarrow} X$ be a fiber bundle with closed oriented fibers.
Let $k \geq \dim F +2$.
Fiber integration for relative differential characters is the group homomorphism $\widehat\pi_!: \widehat H^k(\Phi;\Z) \to \widehat H^{k-\dim F}(\varphi;\Z)$ defined by 
\begin{equation}\label{eq:def_FI_rel}
(\widehat\pi_!h)(s,t)
:=
h(\lambda_\varphi(s,t)) \cdot \exp \Big( 2\pi i \, \int_{(a,b)_\varphi(s,t)} \fint_F (\curv,\cov)(h) \; \Big) \,.
\end{equation} 
Here $(s,t) \in Z_{k-\dim F-1}(\varphi;\Z)$.
\end{definition}
\index{fiber integration!of relative characters}
\index{relative differential character!fiber integration}
\index{+Pihat@$\widehat\pi_\ausruf$, fiber integration}
\index{integration!fiber $\sim$ for differential characters}

Clearly, the mapping $h \mapsto \widehat\pi_!h$ is a additive in $h$, thus $\widehat\pi_!$ is a group homomorphism.
Moreover, the map $(s,t) \mapsto (\widehat\pi_!h)(s,t)$, defined by the right hand side of \eqref{eq:def_FI_rel}, is a group homomorphism $Z_{k-1-\dim F}(\varphi;\Z) \to \Ul$, since the maps $\lambda_\varphi$ and $(a,b)_\varphi$ are group homomorphisms.
In order to show that this homomorphism is indeed a differential character in $\widehat H^{k-\dim F}(\varphi;\Z)$, we need to evaluate it on a boundary $\partial_\varphi (v,w)$, where $(v,w) \in C_{k-\dim F}(\varphi;\Z)$.
This will be done in the proof of Theorem~\ref{thm:FI_props} below.

\subsubsection{Well-definedness.}
Before discussing its properties, we show that fiber integration is well-defined, i.e.~its definition is independent of the choice of geometric representative $(\zeta,\tau)_\varphi(s,t)$ and chain $(a,b)_\varphi(s,t)$:

\begin{lemma}
Let $\varphi:A \to X$ be a smooth map and $\pi:E \to X$ a fiber bundle with closed oriented fibers.
Let $k \geq 2$.
Let $h \in \widehat H^k(\Phi;\Z)$ and $(s,t) \in Z_{k-1-\dim F}(\varphi;\Z)$.
Let $(\zeta',\tau') \in \ZZ_{k-1-\dim F}(X)$ and $(a',b') \in C_{k-\dim F}(\varphi;\Z)$ be any geometric cycle and singular chain such that $[\zeta',\tau']_{\partial_\varphi S_{k-\dim F}} = [(s,t) - \partial_\varphi (a',b')]_{\partial_\varphi S_{k-\dim F}}$.
Then we have:
\begin{equation}
(\widehat\pi_!h)(s,t)
=
h([\PB_{E,\varphi^*E}(\zeta,\tau)]_{\partial_\varphi S_k})  \cdot \exp \Big( 2\pi i \int_{(a',b')} \fint_F (\curv,\cov)(h) \Big) \,. 
\end{equation}
\end{lemma}

\begin{proof}
Since the geometric cycles $(\zeta,\tau)$ and $(\zeta',\tau')$ represent the same homology class in $H_{k-1-\dim F}(\varphi;\Z)$, they are bordant.
A bordism $(W,M)\xrightarrow{(F,G)}(X,A)$ from $(\zeta,\tau)$ to $(\zeta',\tau')$ yields a bordism $(F^*E,G^*(\varphi^*E))\xrightarrow{({\bf F},{\bf G})}(E,\varphi^*E)$ from $\PB_{E,\varphi^*E}(\zeta,\tau)$ to $\PB_{E,\varphi^*E}(\zeta',\tau')$.
By equation \eqref{eq:ffggwm} and the assumption, we have 
\begin{align*}
\partial_\varphi \Big( (F,G)_*[W,M]_{S_{k-\dim F}} \Big)
&\stackrel{\eqref{eq:ffggwm}}{=}
[\zeta',\tau']_{\partial_\varphi S_{k-\dim F}} - [\zeta,\tau]_{\partial_\varphi S_{k-\dim F}} \\
&=
[\partial_\varphi(a,b) - \partial_\varphi(a',b')]_{\partial_\varphi S_{k-\dim F}} \,.
\end{align*}
In particular, we find a cycle $(x,y) \in Z_{k-\dim F}(\varphi;\Z)$ such that 
\begin{equation}\label{eq:wmaabb}
(F,G)_*[W,M]_{S_{k-\dim F}} 
= 
[(a,b)-(a',b')-(x,y)]_{S_{k-\dim F}} \,.
\end{equation}
Using \eqref{eq:bord_fund_cycle}, we obtain:
\begin{align*}
h([\PB_{E,\varphi^*E}&(\zeta',\tau')]_{\partial_\varphi S_k}) \cdot \big(h([\PB_{E,\varphi^*E}(\zeta,\tau)]_{\partial_\varphi S_k})\big)^{-1} \\
&\stackrel{\eqref{eq:bord_fund_cycle}}{=} 
h(\partial_\varphi ({\bf F},{\bf G})_*[F^*E,G^*(\varphi^*E)]_{S_k}) \\
&= 
\exp \Big( 2\pi i \int_{[F^*E,G^*(\varphi^*E)]_{S_k}} ({\bf F},{\bf G})^*(\curv,\cov)(h) \Big) \\
&= 
\exp \Big( 2\pi i \int_{[W,M]_{S_{k-\dim F}}} (F,G)^* \fint_F (\curv,\cov)(h) \Big) \\
&\stackrel{\eqref{eq:wmaabb}}{=}
\exp \Big[ 2\pi i \Big(\int_{(a,b)-(a',b')} \fint_F (\curv,\cov)(h) + \underbrace{\int_{(x,y)} \fint_F (\curv,\cov)(h)}_{\in \Z} \Big) \Big] \\
&=
\exp \Big( 2\pi i \int_{(a,b)-(a',b')} \fint_F (\curv,\cov)(h) \Big) \,. \qedhere
\end{align*}
\end{proof}

\begin{remark}
As a consequence of the preceding lemma, we note that if the cycle $(s,t)$ is a fundamental cycle of a relative geometric cycle $(\zeta',\tau')$, we do not need the chain $(a',b')$.
In this case, the formula \eqref{eq:def_FI_rel} for fiber integration thus simplifies to
\begin{equation}\label{eq:pi_lambda}
(\widehat\pi_!h)(s,t)
=
h(\lambda_\varphi(s,t)) \,.
\end{equation}
\end{remark}

\subsubsection{Naturality and compatibilities.}
In order to discuss naturality of fiber integration, we consider pull-back diagrams for smooth maps in the base:
Let $\psi:B \to Y$ be a smooth map.
Let $(Y,B)\xrightarrow{(f,g)}(X,A)$ be a smooth map, such that the diagram
\begin{equation*}
\xymatrix{
B \ar^\psi[rrr] \ar^g[dr] &&& Y \ar^f[dr] \\
& A \ar^\varphi[rrr] &&& X \\ 
}
\end{equation*}
commutes.
A fiber bundle $\pi:E \to X$ yields a commutative diagram of pull-back bundles:
\begin{equation}
\xymatrix{
g^*(\varphi^*E) \ar^(0.545)\Psi[rrr] \ar[dd] \ar^G[dr] &&& f^*E \ar'[d][dd] \ar^F[dr] \\
& \varphi^*E \ar^(0.4)\Phi[rrr] \ar[dd] &&& E \ar[dd] \\ 
B \ar'[r][rrr]^(0.2){\psi} \ar^g[dr] &&& Y \ar^f[dr] \\
& A \ar_(0.4)\varphi[rrr] &&& X \\ 
} \label{eq:diag_pull-back_bundles}
\end{equation}
Here $\Psi:g^*(\varphi^*E) \to f^*E$ is the bundle map induced by $\psi:B \to Y$.

Now we prove the main theorem of this section:

\begin{thm}[Fiber integration: naturality and compatibilities]\label{thm:FI_props}
Let $\varphi:A \to X$ be a smooth map.
Let $F \hookrightarrow E \stackrel{\pi}{\twoheadrightarrow} X$ be a fiber bundle with closed oriented fibers.
Let $k \geq \dim F+2$.
Then the fiber integration map $\widehat\pi_!:\widehat H^k(\Phi;\Z) \to \widehat H^{k-\dim F}(\varphi;\Z)$ is natural with respect to pull-back diagrams \eqref{eq:diag_pull-back_bundles}, i.e.~for any smooth map $(Y,B)\xrightarrow{(f,g)}(X,A)$ and differential character $h \in \widehat H^k(\Phi;\Z)$, we have: 
\begin{equation}
\widehat\pi_!((F,G)^*h) 
= (f,g)^*\widehat\pi_!(h) \,. \label{eq:fiber_int_natural}
\end{equation}
In other words, the following diagram is commutative for all $k \geq \dim F + 2$:
\begin{equation}\label{eq:fiber_int_natural_diagram}
\xymatrix{
\widehat H^k(\Phi;\Z) \ar^{(F,G)^*}[rr] \ar_{\widehat\pi_!}[d] && \widehat H^k(\Psi;\Z) \ar^{\widehat\pi_!}[d] \\
\widehat H^{k-\dim F}(\varphi;\Z) \ar_{(f,g)^*}[rr] && \widehat H^{k-\dim F}(\psi;\Z) \,.  
}
\end{equation}

\noindent
Fiber integration is compatible with curvature and covariant derivative, i.e., the diagram
\begin{equation}\label{eq:fiber_int_curvature}
\xymatrix{
\widehat H^k(\Phi;\Z) \ar[d]_{\widehat\pi_!} \ar[rr]^{(\curv,\cov)} 
&& \Omega_0^k(\Phi) \ar[d]^{\fint} \\
\widehat H^{k-\dim F}(\varphi;\Z)  \ar[rr]_{(\curv,\cov)} 
&& \Omega_0^{k-\dim F}(\varphi) 
}
\end{equation}
commutes.

\noindent
Fiber integration is compatible with topological trivializations, i.e., the diagram  
\begin{equation}\label{eq:fiber_int_iota}
\xymatrix{
\Omega^{k-1}(\Phi) \ar[d]_{\fint} \ar[r]^{\iota}
& \widehat H^k(\Phi;\Z) \ar[d]^{\widehat\pi_!} \\
\Omega^{k-1-\dim F}(\varphi) \ar[r]_{\iota}
& \widehat H^{k-\dim F}(\varphi;\Z)
}
\end{equation}
commutes.

\noindent
Fiber integration for relative differential characters commutes with the maps $\ti$ and $\vds$.
We thus have the commutative diagram
\begin{equation}\label{eq:diag_FI_t_p_commute}
\xymatrix{
\widehat H^{k-1}(\varphi^*E;\Z) \ar^{\ti_\Phi}[rr] \ar^{\widehat\pi_!}[d] && \widehat H^k(\Phi;\Z) \ar^{\vds_\Phi}[rr] \ar^{\widehat\pi_!}[d] && \widehat H^k(E;\Z) \ar^{\widehat\pi_!}[d] \\ 
\widehat H^{k-1-\dim F}(A;\Z) \ar^{\ti_\varphi}[rr] && \widehat H^{k-\dim F}(\varphi;\Z) \ar^{\vds_\varphi}[rr] && \widehat H^{k-\dim F}(X;\Z) \,.
}
\end{equation}
\end{thm}
\index{Theorem!Fiber integration: naturality and compatibilities}

\begin{proof}
a)
We first show naturality:
Let $\psi:B \to Y$ and $(Y,B)\xrightarrow{(f,g)}(X,A)$ be a smooth maps.
Let $h \in \widehat H^k(\Phi;\Z)$ and $(s,t) \in Z_{k-\dim F-1}(\psi;\Z)$.
Choose $(\zeta,\tau)_\psi(s,t) \in \ZZ_{k-\dim F-1}(\psi)$ and $(a,b)_\psi(s,t) \in C_{k-\dim F}(\psi;\Z)$ such that $[(s,t) - \partial_\psi(a,b)]_{\partial_\psi S_{k-\dim F}} = [\zeta,\tau]_{\partial_\psi S_{k-\dim F}}$.

Now put 
\begin{align*}
(\zeta,\tau)_\varphi((f,g)_*(s,t))&
:=(f,g)_*\big((\zeta,\tau)_\psi(s,t)\big) \in \ZZ_{k-\dim F-1}(\varphi) \\
(a,b)_\varphi((f,g)_*(s,t)) 
&:= (f,g)_*\big((a,b)_\psi(s,t)\big) \in C_{k-\dim F}(\varphi;\Z) \,. 
\end{align*}
In particular, we may choose the transfer maps naturally such that $\lambda_\varphi((f,g)_*(s,t)) := (F,G)_*\lambda_\psi(s,t)$.
This yields:
\begin{align*}
(\widehat\pi_!((F,G)^*h))(s,t)
&=
\big((F,G)^*h\big)(\lambda_\psi(s,t)) \\
&\qquad\qquad\cdot \exp \Big( 2\pi i \int_{(a,b)_\psi(s,t)} \fint_F (\curv,\cov)((F,G)^*h)) \, \Big) \\
&=
h((F,G)_*\lambda_\psi(s,t)) \cdot \exp \Big( 2\pi i \int_{(a,b)_\psi(s,t)} (f,g)^* \fint_F (\curv,\cov)(h) \Big) \\
&=
h(\lambda_\varphi((f,g)_*(s,t))) \cdot \Big( 2\pi i \int_{(a,b)_\varphi((f,g)_*(s,t))} \fint_F (\curv,\cov)(h) \Big) \\
&=
(\widehat\pi_!h)((f,g)_*(s,t)) \\
&=
\big((f,g)^*\widehat\pi_!h\big)(s,t) \,.
\end{align*}

b)
Next we compute the curvature and covariant derivative of the character $\widehat\pi_!h$.
To this end we evaluate $\widehat\pi_!h$ on a boundary $\partial_\varphi(v,w)$, where $(v,w) \in C_{k-\dim F}(\varphi;\Z)$:
\begin{align*}
(\widehat\pi_!h)(\partial_\varphi(v,w))
&= 
h(\lambda_\varphi(\partial_\varphi(v,w))) \cdot \exp \Big( \int_{(a,b)_\varphi(\partial_\varphi(v,w))} \fint_F (\curv,\cov)(h) \Big) \\
&\stackrel{\eqref{eq:del_lambda}}{=} 
h(\partial_\Phi \lambda_\varphi(v,w)) \cdot \exp \Big( \int_{(a,b)_\varphi(\partial_\varphi(v,w))} \fint_F (\curv,\cov)(h) \Big) \\
&=
\exp \Big( 2\pi i \Big( \int_{\lambda_\varphi(v,w)} (\curv,\cov)(h) + \int_{(a,b)_\varphi(\partial_\varphi(v,w))} \fint_F (\curv,\cov)(h) \Big) \Big) \\
&\stackrel{\eqref{eq:lambda_fint}}{=}
\exp \Big( 2 \pi i \int_{(v,w)} \fint_F (\curv,\cov)(h) \Big) \,.
\end{align*}
Thus the homomorphism $Z_{k-\dim F-1}(\varphi;\Z) \to \Ul$ defined by the right hand side of \eqref{eq:def_FI_rel} satisfies condition \eqref{eq:def_rel_diff_charact_2}.
We conclude that $\widehat\pi_!h$ is a differential character in $\widehat H^{k-\dim F}(\varphi;\Z)$ with curvature $\curv(\widehat\pi_!h) = \fint_F \curv(h)$ and covariant derivative $\cov(\widehat\pi_!h) = \fint_F \cov(h)$.

c)
Now we prove compatibility with topological trivializations:
Let $(\omega,\vartheta) \in \Omega^{k-1}(\Phi)$ and $(s,t) \in Z_{k-1-\dim F}(\varphi;\Z)$.
Then we have:
\begin{align*}
\widehat\pi_!\big(\iota_\Phi(\omega,\vartheta)\big)(s,t)
&=
\big(\iota_\Phi(\omega,\vartheta)\big)(\lambda_\varphi(s,t)) \cdot \exp \Big( 2\pi i \int_{(a,b)_\varphi(s,t)} \fint_F \underbrace{(\curv,\cov)(\iota_\varphi(\omega,\vartheta))}_{=d_\Phi(\omega,\vartheta)} \Big) \\
&=
\exp \Big( 2\pi i \Big( \int_{\lambda_\varphi(s,t)} (\omega,\vartheta) + \int_{(a,b)_\varphi(s,t)} \fint_F d_\Phi(\omega,\vartheta) \Big) \Big) \\
&\stackrel{\eqref{eq:fiber_d_commute}}{=}
\exp \Big( 2\pi i \Big( \int_{\lambda_\varphi(s,t)} (\omega,\vartheta) + \int_{(a,b)_\varphi(s,t)} d_\varphi \fint_F (\omega,\vartheta) \Big) \Big) \\
&\stackrel{\eqref{eq:lambda_fint_1}}{=}
\exp \Big( 2\pi i \int_{(s,t)} \fint_F (\omega,\vartheta) \Big) \\
&=
\iota_\varphi\Big(\fint_F(\omega,\vartheta)\Big)(s,t) \,. 
\end{align*}

d)
Finally we prove compatibility with the maps $\ti$ and $\vds$. 
It follows from diagram \eqref{eq:diag_lambda_commute}, i.e.~from compatibility of the transfer maps with the sequence \eqref{eq:ex_seq_Z}.
Let $h \in \widehat H^{k-1}(\varphi^*E;\Z)$ and $(s,t) \in Z_{k-\dim F-1}(\varphi;\Z)$.
Let $\sigma:Z_{k-\dim F-2}(A;\Z) \to Z_{k-\dim F-1}(\varphi;\Z)$ be a splitting as in~\eqref{eq:ex_seq_Z}. 
Write $(s,t) = (z,0) + \sigma(t)$. 
Then we have:
\begin{align*}
\widehat\pi_!(\ti_\Phi (h))(s,t)
&\stackrel{\eqref{eq:def_FI_rel}}{=}
\ti_\Phi(h)(\lambda_\varphi(s,t)) \cdot \exp \Big( 2\pi i \int_{(a,b)_\varphi(s,t)} \fint_F \;\underbrace{(\curv,\cov)(\ti_\Phi(h))}_{=(0,-\curv(h))} \Big) \\ 
&=
h(p(\lambda_\varphi(s,t))) \cdot \exp \Big( 2\pi i \int_{-a(t)} \fint_F  -\curv(h) \Big) \\
&\stackrel{\eqref{eq:diag_lambda_commute}}{=}
h(\lambda(t)) \cdot \exp \Big( 2\pi i \int_{a(t)} \curv(h) \Big) \\
&\stackrel{\eqref{eq:FI_absolute}}{=}
(\widehat\pi_!h)(t) \\
&=\ti_\varphi(\widehat\pi_!h)(s,t) \,. \\
\intertext{Similarly, for a character $h \in \widehat H^k(\Phi;\Z)$ and a cycle $z \in Z_{k-\dim F-1}(X;\Z)$, we have:}
\vds_\varphi(\widehat\pi_! h)(z)
&=
(\widehat\pi_!h) (z,0) \\
&=
h(\lambda_\varphi (z,0)) \cdot \exp \Big( 2\pi i \int_{(a,b)_\varphi (z,0)} \fint_F (\curv,\cov)(h) \Big) \\
&\stackrel{\eqref{eq:diag_lambda_commute}}{=} 
h((\lambda(z),0)) \cdot \exp \Big( 2\pi i \int_{(a(z),0)} \fint_F (\curv,\cov)(h) \Big) \\ 
&=
\vds_\Phi(h)(\lambda(z)) \cdot \exp \Big( 2\pi i \int_{a(z)} \fint_F \curv(\vds_\Phi(h)) \Big) \\
&\stackrel{\eqref{eq:FI_absolute}}{=}
\widehat\pi_!(\vds_\Phi(h))(z) \,. \qedhere
\end{align*}
\end{proof}

As a corollary of Theorem~\ref{thm:FI_props}, we obtain compatibility of fiber integration with all the maps in the long exact sequence \eqref{eq:long_ex_sequ} for relative and absolute differential characters groups. 

\begin{cor}[Compatibility with long exact sequence]
Let $\pi:E \to X$ be a fiber bundle with oriented closed fibers and $\varphi:A \to X$ a smooth map.
Then the fiber integration map $\widehat\pi_!$ on (relative and absolute) differential characters commutes with all maps in the long exact sequence \eqref{eq:long_ex_sequ}, and with the usual fiber integration maps $\pi_!$ on cohomology with $\Ul$- and $\Z$-coefficients, respectively. 
\end{cor}
\index{fiber integration!compatibility with long exact sequence}

\begin{proof}
Theorem~\ref{cor:int_prop} shows that $\widehat\pi_!$ commutes with the maps $\ti$ and $\vds$.
The rest follows from \cite[Prop.~7.10, Prop.~7.11]{BB13}.  
\end{proof}

\subsubsection{Compatibility with characteristic class.}
The Leray-Serre spectral sequence of a fiber bundle $\pi:E \to X$ has an obvious modification that converges to the mapping cone cohomology $H^*(\Phi;\Z)$.
Using this modified spectral sequence, fiber integration on mapping cone cohomology groups $\pi_!:H^k(\Phi;\Z) \to H^{k-\dim F}(\varphi;\Z)$ can be defined in the same way as in \cite[§~8]{BH53} for absolute cohomology.
Compatibility of fiber integration with the characteristic class is discussed in detail for absolute differential characters in \cite[Ch.~7]{BB13}.   
The crucial point is that fiber integration for cohomology classes can be realized by pre-composition of cocycles with the transfer map:
$$
\pi_!:H^k(E;\Z) \to H^{k-\dim F}(X;\Z) \,, \quad  [\mu] \mapsto [\mu \circ \lambda] \,.
$$
In the same way, we obtain compatibility of fiber integration of relative characters with the characteristic class:

\begin{remark}[Compatibility with characteristic class]
Choosing an extension of the transfer map $\lambda_\varphi$ to a homomorphism of chains as in \eqref{eq:del_lambda}, one can show that fiber integration of differential characters is compatible with the characteristic class.
Thus for any relative character $h \in \widehat H^k(\Phi;\Z)$, we have:
\begin{equation}\label{eq:fint_classes}
c(\widehat\pi_!h)
=
\pi_!c(h) \,. 
\end{equation}
\end{remark}
\index{fiber integration!compatibility with characteristic class}

\subsubsection{Fiber integration of parallel characters.}
By \eqref{eq:fiber_int_curvature} we have $\cov(\widehat \pi_!h) = \fint_F \cov(h)$.
Thus fiber integrals of parallel characters are again parallel.
This way we obtain fiber integration on $\widehat H^*(E,E|_A;\Z)$: 

\begin{cor}[Fiber integration of parallel characters]
Let $\pi:E \to X$ be a fiber bundle with closed oriented fibers.
Let $i_A:A \to X$ be the embedding of a smooth submanifold.
Denote by $I_A:E|_A \to E$ the induced bundle map.
Let $k \geq \dim F + 2$.
Then the inclusion $\widehat H^k(E,E|_A;\Z) \hookrightarrow \widehat H^k(I_A;\Z)$ induces a natural fiber integration map
$$
\widehat\pi_!:\widehat H^k(E,E|_A;\Z) \to \widehat H^{k-\dim F}(X,A;\Z)
$$ 
that commutes with curvature, characteristic class and topological trivializations.
Moreover, it commutes with the long exact sequence \eqref{eq:long_ex_sequ_par} and fiber integration for cohomology with $\Ul$- and $\Z$-coefficients, respectively.
\end{cor}
\index{fiber integration!of parallel characters}

By the fiberwise Stokes theorem, the fiber integral of a closed form with integral periods is again closed with integral periods (as long as the fibers are closed).
Thus by the idenfication from \cite{BT06} of the relative Hopkins-Singer group $\check H^k(\Phi;\Z)$ as the quotient of the subgroup $\widehat H^k_0(\Phi;\Z) \subset \widehat H^k(\Phi;\Z)$ by closed forms with integral periods on $E$, the fiber integration map $\widehat\pi_!$ descends to the relative Hopkins-Singer group:  

\begin{cor}[Fiber integration of relative differential cocycles]
Let $\pi:E \to X$ be a fiber bundle with closed oriented fibers.
Let $\varphi:A \to X$ be a smooth map and $\Phi:\varphi^*E \to E$ the induced bundle map.
Let $k \geq \dim F + 2$.
Then fiber integration of relative differential characters descends to a fiber integration map 
$$
\widehat\pi_!:\check H^k(\Phi;\Z) \to \check H^{k-\dim F}(\varphi;\Z)
$$ 
that commutes with the long exact sequence \eqref{eq:long_ex_sequ_HS}.
\end{cor}
\index{fiber integration!of relative differential cocycles}

\subsection{Fibers with boundary}\label{subsec:fiber_boundary}
Let $\pi:E \to X$ be a fiber bundle with compact oriented fibers $F$ with boundary $\partial F$.
We have the induced fiber bundle $\pi^{\partial E}:\partial E \to X$ with closed oriented fibers.
Fiber integration of differential forms satisfies the fiberwise Stokes theorem:
\index{fiberwise Stokes theorem}
\index{Stokes theorem!fiberwise $\sim$}
\begin{equation*}
d \fint_F \omega = \fint_F d \omega + (-1)^{k + \dim F} \fint_{\partial F} \omega \,,
\end{equation*}
where $\omega \in \Omega^k(E)$.
Likewise, we have the fiberwise Stokes theorem for mapping cone differential forms:
\index{Stokes theorem!mapping cone $\sim$}
\index{mapping cone!Stokes theorem}
\begin{equation}\label{eq:stokes_fiber}
d_\varphi \fint_F (\omega,\vartheta) 
=
\fint_F d_\Phi (\omega,\vartheta) + (-1)^{k + \dim F} \fint_{\partial F} (\omega,\vartheta) \,,
\end{equation}
where $(\omega,\vartheta) \in \Omega^k(\Phi)$.
In particular, fiber integration over the boundary $\partial F$ maps $d_\Phi$-closed forms to $d_\varphi$-exact forms. 
Likewise, fiber integration in the bundle $\pi:\partial E \to X$ of mapping cone cohomology classes for the bundle map $\Phi:\varphi^*E \to E$ yields the trivial map $\pi^{\partial E}:H^k(\Phi;\Z) \to H^{k-\dim \partial E}(\varphi;\Z)$.

In \cite[Ch.~7]{BB13} we show that integration of differential characters on $E$ over the fibers of $\pi:\partial E \to X$ yields topologically trivial characters on $X$.
More precisely, for a character $h \in \widehat H^k(E;\Z)$ with $k \geq \dim F$, we have:
\begin{equation}\label{eq:fint_bound_abs}
\widehat\pi^{\partial E}_!h
=
\iota\Big((-1)^{k-\dim F} \fint_F \curv(h) \Big) \,.
\end{equation}
Here we consider two generalizations of this result.

First we consider fiber integration of relative differential characters in the fiber bundle $\pi^{\partial E}:\partial E \to X$.

\begin{prop}
Let $E \to X$ be a fiber bundle with compact oriented fibers with boundary. 
Let $\varphi:A \to X$ be a smooth map and $\Phi: \varphi^*E \to E$ the induced bundle map.
Then for any character $h \in \widehat H^k(\Phi;\Z)$, the integrated character $\widehat\pi_!^{\partial E}h \in \widehat H^{k-\dim\partial F}(\varphi;\Z)$ is topologically trivial, and we have:
\begin{equation}\label{eq:fint_bound_triv}
\widehat\pi_!^{\partial E} h
=
\iota\Big((-1)^{k-\dim F} \fint_F (\curv,\cov)(h)\Big) \,. 
\end{equation}
\end{prop}

\begin{proof}
The integrated character $\widehat\pi_!^{\partial E}h$ is topologically trivial since $c(\widehat\pi_!^{\partial E}h)=\pi_!^{\partial E}(c(h))$ and $\pi^{\partial E}_!:H^k(\Phi;\Z) \to H^{k-\dim\partial F}(\varphi;\Z)$ is the trivial map.
We compute the curvature and covariant derivative of $\widehat\pi_!^{\partial E}h$ using the fiberwise Stokes theorem \eqref{eq:stokes_fiber}: 
\begin{align*}
(\curv,\cov)(\widehat\pi_!^{\partial E}h) 
&\stackrel{\eqref{eq:fiber_int_curvature}}{=}
\fint_{\partial F}(\curv,\cov)(h) \\
&\stackrel{\eqref{eq:stokes_fiber}}{=}
(-1)^{k-\dim F} d_\varphi \Big(\int_{F} (\curv,\cov)(h)\Big) \,. 
\end{align*}
Now the exact sequences \eqref{eq:rel_sequences} together with the commutative diagram 
\begin{equation*}
\xymatrix{
\Omega^{*-1}(\varphi) \ar^{d_\varphi}[rr] \ar^\iota[d] && d_\varphi(\Omega^{*-1}(\varphi)) \ar[d] \\
\widehat H^*(\varphi;\Z) \ar_{(\curv,\cov)}[rr] && \Omega^k_0(\varphi) 
}
\end{equation*}
yield \eqref{eq:fint_bound_triv}.
\end{proof}

We obtain another generalization of \eqref{eq:fint_bound_abs} by weakening the condition on the fiber bundle:
instead of a fiber bundle with fibers that bound we consider a fiber bundle $\pi:E \to X$ and a smooth map $\varphi:A \to X$ such that the pull-back bundle $\pi:\varphi^*E \to A$ is the fiberwise boundary of a fiber bundle $\pi':E'\to A$.
For this situation we introduce the following notation:

\begin{definition}
Let $\pi:E \to X$ be a fiber bundle with closed oriented fibers $F$ and $\varphi:A \to X$ a smooth map.
We say that $\pi:E \to X$ \emph{bounds along} $\varphi$ if there exists a fiber bundle $\pi':(E',\partial E')\to A$ with compact oriented fibers with boundary $(F',\partial F')$ and a fiber bundle isomorphism $\partial E' \to \varphi^*E$ over the identity $\id_A$. 
For short, we say that $\pi:E \to X$ bounds $\pi':E' \to A$ along $\varphi:A \to X$.
\end{definition}

With this notation, we obtain a generalization of \eqref{eq:fint_bound_abs} for bundles that bound along a smooth map:

\begin{prop}\label{prop:fiber_bound_rel}
Let $\pi:E \to X$ be a fiber bundle with closed oriented fibers that bounds a fiber bundle $\pi':E' \to A$ with compact oriented fibers with boundary $(F',\partial F')$ along a smooth map $\varphi:A \to X$.
Then for any differential character $h \in \widehat H^k(E;\Z)$ the integrated character $\widehat\pi_!h \in \widehat H^{k-\dim F}(X;\Z)$ has a section along $\varphi$ with covariant derivative $(-1)^{k+\dim F'} \cdot \int_{F'} \curv(\Phi^*h)$.
\end{prop}

\begin{proof}
As above, the fiber integration map ${\pi'}^{\partial E'}_!:H^k(E';\Z) \to H^{k-\dim \partial F'}(A;\Z)$ is trivial.
Thus the integrated character $\widehat\pi_!h$ is topologically trivial along $\varphi$, since $\varphi^*c(\widehat\pi_!h) = \widehat{\pi'}^{\partial E'}(c(\Phi^*h)) =0$.
By the fiberwise Stokes theorem, the curvature satisfies:
$$
\varphi^*\curv(\widehat\pi_!h)
=
\fint_F \Phi^* \curv(h) 
=
\fint_{\partial F'} \curv(\Phi^*h)
=
(-1)^{k+\dim F'} \cdot d \Big(\int_{F'} \curv(\Phi^*h) \Big) \,.
$$ 
Thus the integrated character $\widehat\pi_!h$ admits a section along $\varphi$ with covariant derivative $(-1)^{k+\dim F'} \cdot \int_{F'} \curv(\Phi^*h)$.
\end{proof}

\subsection{The up-down formula}\label{subsec:up-down}

Fiber integration for differential forms satisfies the following up-down formula:
Let $(\omega,\vartheta) \in \Omega^k(\varphi)$ and $\omega' \in \Omega^{k'}(E)$.
Then we have the equality 
\begin{align*}
\fint_F \pi^*(\omega,\vartheta) \wedge \omega' 
&=
(\omega,\vartheta) \wedge \fint_F \omega' 
\intertext{of differential forms in $\Omega^{k+k'-\dim F}(\varphi)$.
Likewise, for cohomology classes $c \in H^k(\varphi)$ and $c' \in H^{k'}(E;\Z)$, we have}
\pi_!(\pi^*c \cup c')
&=
c \cup \pi_!c' \,. 
\end{align*}
More generally, fiber integration in cross products is compatible with cross products of differential forms and cohomology classes in the following sense:
Let $\pi:E \to X$ and $\pi':E' \to X'$ be fiber bundles with closed oriented fibers $F$ and $F'$.
Let $\varphi:A \to X$ be a smooth map and $\Phi:\varphi^*E \to E$ the induced bundle map.
Let $(\omega,\vartheta) \in \Omega^k(\Phi)$ and $\omega' \in \Omega^{k'}$ be differential forms and $c \in H^k(\Phi;\Z)$ and $c' \in H^{k'}(E';\Z)$ cohomology classes. 
Then we have:
\begin{align}
\fint_{F \times F'} (\omega,\vartheta) \times \omega' 
&=
(-1)^{(k'-\dim F')\cdot \dim F} \cdot \fint_F (\omega,\vartheta) \times \fint_{F'} \omega' \label{eq:cross_fint_forms} \\
\pi_!(c \times c')
&= 
(-1)^{(k'-\dim F')\cdot \dim F} \cdot \pi_!c \times \pi'_!c' \,. \label{eq:cross_fint_classes}
\end{align}
In \cite[Ch.~7]{BB13} we proved the following up-down formula for absolute differential characters:
for $h \in \widehat H^k(X;\Z)$ and $h' \in \widehat H^{k'}(E;\Z)$, we have the equality
$$
\widehat\pi_!(\pi^*h * h')
=
h * \widehat\pi_!h'
$$
of differential characters in $\widehat H^{k+k'-\dim F}(X;\Z)$.
Here we prove the relative version of this up-down formula and the relative version of the compatibility of fiber integration with cross products. 
The method of proof is the same as in \cite[Ch.~7]{BB13}.

\subsubsection{Compatibility with cross products.} 
We start with compatibility of cross products with fiber integration in fiber products:

\index{fiber integration!compatibility with fiber and cross products}
\begin{thm}[Fiber integration: compatibility with fiber and cross products]
Let $\pi:E \to X$ and $\pi':E' \to X'$ be fiber bundles with closed oriented fibers and $\varphi:A \to X$ a smooth map. 
Let $\Phi:\varphi^*E \to E$ be the induced bundle map.
Let $\pi \times\pi':E \times E' \to X \times X'$ denote the fiber product with fiber orientation the product orientation of $F \times F'$.  
Then fiber integration of differential characters is compatible with the fiber product and the cross product in the sense that the following diagram is graded commutative:
\begin{equation*}
\xymatrix{
\widehat H^k(\Phi;\Z) \otimes \widehat H^{k'}(E';\Z) \ar@<-27pt>[d]^{\widehat\pi_!} \ar@<+24pt>[d]^{\widehat\pi'_!} \ar^{\times}[rr] && \widehat H^{k+k'}(\Phi \times \id_{E'};\Z) \ar^{\widehat{\pi \times \pi'}_!}[d] \\
\widehat H^{k-\dim F}(\varphi;\Z) \otimes \widehat H^{k'-\dim F'}(E';\Z) \ar^{\times}[rr] && \widehat H^{k+k'-\dim(F\times F')}(\varphi \times \id_{X'};\Z) \,. \\
} 
\end{equation*}
More explicitly, for characters $h \in \widehat H^k(\Phi;\Z)$ and $h' \in \widehat H^{k'}(E';\Z)$ we have: 
\begin{equation}\label{eq:cross_fiber}
\widehat\pi_!h \times \widehat\pi'_!h'
=
(-1)^{(k'-\dim F')\cdot \dim F}\cdot \widehat{\pi \times \pi'}_!(h \times h') \,.
\end{equation}
\end{thm}
\index{Theorem!Fiber integration: compatibility with fiber and cross products}

\begin{proof}
Let $h \in \widehat H^k(\Phi;\Z)$ and $h' \in \widehat H^{k'}(E';\Z)$.
We compare the two sides of \eqref{eq:cross_fiber} by evaluating them on cycles in $Z_{k+k'-\dim(F \times F')-1}(\varphi \times \id_{X'};\Z)$.
By definition of the cross product, we consider the evaluation on cross products of cycles and on cycles in the K\"unneth complement $T_{k+k-\dim(F \times F')-1}(\varphi\times \id_{X'};\Z)$ separately.
More specifically, we may assume the cycles $(x,y) \in Z_i(\varphi;\Z)$ and $y' \in Z_j(X';\Z)$ to be fundamental cycles of appropriately chosen stratifolds.
For the correction terms are boundaries which may be added to the torsion cycles in $T_{k+k-\dim(F \times F')-1}(\varphi\times \id_{X'};\Z)$.

The transfer map on the fiber product can be chosen multiplicatively as in \eqref{eq:diag_lambda_mult}.
Let $(x,y) \in Z_i(\varphi;\Z)$ and $y' \in Z_j(X';\Z)$.
If $(i,j)$ neither equals $(k-1-\dim F,k'-\dim F')$ nor $(k-\dim F,k'-1-\dim F')$, then both sides of \eqref{eq:cross_fiber} vanish on $(s,t) = (x,y) \times y'$.
For $(i,j) = (k-1-\dim F,k'-\dim F')$, we have: 
\begin{align*}
\big(\widehat\pi_!h \times \widehat\pi'_!h'\big)((x,y) \times y') 
&=
(\widehat\pi_!h(x,y))^{\langle c(\widehat\pi'_!h'),y'\rangle} \\
&\stackrel{\eqref{eq:pi_lambda}}{=}
h(\lambda_\varphi(x,y))^{\langle h',\lambda'(y') \rangle} \\
&\stackrel{\eqref{eq:lambda_cross}}{=}
(-1)^{(k'-\dim F') \cdot \dim F} \cdot (h \times h')(\lambda_{\varphi\times\id_{X'}}(x,y)\times y') \\
&=
(-1)^{(k'-\dim F') \cdot \dim F} \big(\widehat{\pi \times \pi'}_!(h \times h')\big)((x,y) \times y') \,. 
\intertext{Similarly, for $(i,j) = (k-\dim F,k'-1-\dim F')$, we find:}
\big(\widehat\pi_!h \times \widehat\pi'_!h'\big)((x,y) \times y') 
&=
(\widehat\pi'_!h'(y'))^{(-1)^{k-\dim F} \cdot \langle c(\widehat\pi_!h),(x,y)\rangle} \\
&=
h'(\lambda'(y'))^{(-1)^{k-\dim F} \cdot \langle c(h),\lambda_\varphi(x,y) \rangle} \\
&=
(h \times h')((-1)^{\dim F} \cdot \lambda_\varphi(x,y) \times \lambda'(y')) \\
&\stackrel{\eqref{eq:lambda_cross}}{=}
(-1)^{(k'-1-\dim F') \cdot \dim F} \cdot (-1)^{\dim F} \cdot (h \times h')(\lambda_{\varphi\times\id_{X'}}(x,y) \times y') \\
&=
(-1)^{(k'-\dim F') \cdot \dim F} \cdot \big(\widehat{\pi \times \pi'}_!(h \times h') \big)((x,y) \times y') \,.
\end{align*}

It remains to verify \eqref{eq:cross_fiber} on the K\"unneth complement $T_{k+k'-\dim(F \times F')-1}(\varphi\times\id_{X'};\Z)$ and on the correction terms obtained from replacing cycles $(x,y) \in Z_i(\varphi;\Z)$ and $y' \in Z_j(X';\Z)$ by fundamental cycles of appropriately chosen stratifolds.
It suffices that (more generally) the two sides of \eqref{eq:cross_fiber} coincide on all torsion cycles.
By Remark~\ref{rem:ev_torsion} this follows from the fact that curvature, covariant derivative and characteristic class of the two sides of \eqref{eq:cross_fiber} coincide.
The latter follows from multiplicativity \eqref{eq:curvcov_mult_cross}, \eqref{eq:c_mult_cross}, compatibility of fiber integration with curvature, covariant derivative and characteristic class \eqref{eq:fiber_int_curvature}, \eqref{eq:fint_classes} and compatibility of fiber integration in fiber products with cross products of differential forms and cohomology classes \eqref{eq:cross_fint_forms}, \eqref{eq:cross_fint_classes}. 
\end{proof}

\subsubsection{The up-down formula.}
As a corollary of the compatibility of the cross product with fiber integration in fiber products we obtain the following up-down formula:

\index{up-down formula}
\index{formua!up-down $\sim$}
\index{relative differential character!up-down formula}
\index{differential character!up-down formula}
\begin{cor}[Up-down formula]
Let $\pi:E \to X$ be fiber bundle with closed oriented fibers.
Let $\varphi:A \to X$ be a smooth map.
Let $h \in \widehat H^k(\varphi;\Z)$ and $h' \in \widehat H^{k'}(E;\Z)$.
Then we have the equality
\begin{equation}
\widehat\pi_!(\pi^*h * h')
=
h * \widehat\pi_!h' 
\end{equation}
of differential characters in $\widehat H^{k+k'-\dim F}(\varphi;\Z)$.
\end{cor}

\begin{proof}
The method of proof is the same as for absolute differential characters in \cite[Ch.~7]{BB13}.
We first consider the fiber product $E \times E \to X \times X$.
Write this as the composite fiber bundle $E\times E \xrightarrow{\pi \times \id_E} X \times E \xrightarrow{\id_X \times \pi} X \times X$.
Let 
$$
\Delta_{(E,\varphi^*E)}:= (\Delta_E,(\id_{\varphi^*E} \times \Phi)\circ \Delta_{\varphi^*E}):(E,\varphi^*E) \to (E,\varphi^*E) \times E
$$
and 
$$
\Delta_{(X,A)} = (\Delta_X,(\id_A \times \varphi) \circ \Delta_A):(X,A) \to (X,A) \times X  
$$ 
be the relative diagonal maps as in Section~\ref{subsec:module}.
Then we have the pull-back diagram
\begin{equation}\label{eq:cross_bundle_map}
\xymatrix{
&&& (E,\varphi^*E) \times E \ar[d]^{\pi \times \id_E} \\
(E,\varphi^*E) \ar[rrr]_{(\pi \times \id_E) \circ \Delta_{(E,\varphi^*E)}} \ar[d]_\pi \ar[urrr]^{\Delta_{(E,\varphi^*E)}} &&& (X,A) \times E \ar[d]^{\id_{(X,A)} \times \pi} \\
(X,A) \ar[rrr]_{\Delta_{(X,A)}} &&& (X,A) \times X 
} 
\end{equation}
The map $(\pi \times \id_E) \circ \Delta_{(E,\varphi^*E)}:(E,\varphi^*E) \to (X,A) \times E$ is the bundle map induced by the diagonal map $\Delta_{(X,A)}:(X,A) \to (X,A) \times X$.
Now we compute:
\begin{align*}
\widehat\pi_!(\pi^*h * h')
&=
\widehat\pi_!(\Delta_{(E,\varphi^*E)}^*(\pi^*h \times h')) \\
&\stackrel{\eqref{eq:cross_nat}}{=}
\widehat\pi_!\big( \Delta_{(E,\varphi^*E)}^*((\pi \times \id_E)^*(h \times h'))\big) \\
&\stackrel{\eqref{eq:fiber_int_natural}}{=}
\Delta_{(X,A)}^*(\widehat{\id \times \pi})_!(h \times h') \\
&\stackrel{\eqref{eq:cross_fiber}}{=}
\Delta_{(X,A)}^*(h \times \widehat\pi_!h') \\
&=
h * \widehat\pi_!h' \,.
\end{align*}
In the second last equation, the sign from \eqref{eq:cross_fiber} drops out since we are considering the fiber product $\id_{(X,A)} \times \pi:(X,A) \times E \to (X,A) \times X$, where the bundle on the first factor has point fibers.
\end{proof}

\subsection{Transgression}
Transgression of differential characters along oriented closed manifolds $\Sigma$ was constructed in \cite[Ch.~9]{BB13}.
Here we adapt this construction to relative differential characters.
Transgression of relative characters is used in \cite{BB13}.

\subsubsection{Transgression along closed manifolds.}
Let $\Sigma$ be a closed oriented manifold.
Denote by $\ev_\Sigma: C^\infty(\Sigma,X) \times \Sigma \to X$, $(f,m) \mapsto f(m)$, the evaluation map.
Composition with a smooth map $\varphi:A \to X$ induces a smooth map $\overline{\varphi}:C^\infty(\Sigma,A) \to C^\infty(\Sigma,X)$, $f \mapsto \varphi \circ f$.
Moreover, the evaluation map yields a smooth map $\ev_\Sigma:C^\infty(\Sigma,(X,A)) \times \Sigma \to (X,A)$, $((f,g),m) \mapsto (f(m),g(m))$.    

Transgression of relative differential characters in $\widehat H^k(\varphi;\Z)$ is defined by pull-back along the evaluation map $\ev_\Sigma$ followed by integration over the fiber $\Sigma$ of the trivial bundle $\pi:C^\infty(\Sigma,(X,A)) \times \Sigma \to C^\infty(\Sigma,(X,A))$:

\index{transgression}
\index{+TauSigma@$\tau_\Sigma$, transgression along closed manifold}
\begin{definition}[Transgression along $\Sigma$]\label{def:tau}
Let $\varphi:A \to X$ be a smooth map.
Let $\Sigma$ be a closed oriented manifold.
Transgression along $\Sigma$ is the group homomorphism
\begin{equation*}
\tau_\Sigma: \widehat H^k(\varphi;\Z) \to \widehat H^{k-\dim \Sigma}({\overline{\varphi}};\Z) \,,
\quad h \mapsto \widehat\pi_!(\ev_\Sigma^*h) \,.  
\end{equation*} 
\end{definition}

From the commutative diagram \eqref{eq:diag_FI_t_p_commute} we conclude that transgression for relative characters commutes with the maps $\ti$ and $\vds$ and transgression for absolute characters constructed in \cite[Ch.~9]{BB13}:

\begin{prop}\label{prop:tau_t_p_commute}
Let $\varphi:A \to X$ be a smooth map.
Let $\Sigma$ be a closed oriented surface.
Then the transgression maps $\tau_\Sigma$ for absolute and relative characters groups commute with the maps $\ti$ and $\vds$.
Thus we have the commutative diagram:
\begin{equation}\label{eq:ti_vds_trans_commute}
\xymatrix{
\widehat H^{k-1}(A;\Z) \ar[d]_{\tau_\Sigma} \ar^(0.55){\ti_\varphi}[rr] && \widehat H^k(\varphi;\Z) \ar[d]_{\tau_\Sigma} \ar^(0.45){\vds_\varphi}[rr] && \widehat H^k(X;\Z) \ar[d]_{\tau_\Sigma} \\
\widehat H^{k-1-\dim \Sigma}(C^\infty(\Sigma,A);\Z) \ar^(0.55){\ti_{\overline{\varphi}}}[rr] && \widehat H^{k-\dim \Sigma}(\overline{\varphi};\Z) \ar^(0.45){\vds_{\overline{\varphi}}}[rr] && \widehat H^{k-\dim \Sigma}(C^\infty(\Sigma,X);\Z) \,. \\
} 
\end{equation}
\end{prop}

\begin{proof}
The claim follows from naturality of the homomorphisms $\ti$ and $\vds$ and the commutative diagram \eqref{eq:diag_FI_t_p_commute}:
Let $h \in \widehat H^{k-1}(A;\Z)$.
For the left square in \eqref{eq:ti_vds_trans_commute} we have: 
$$
\tau_\Sigma\ti_\varphi(h) 
= 
\widehat \pi_!(\ev_\Sigma^*(\ti_\varphi(h))) 
\stackrel{\eqref{eq:ti_nat}}{=}
\widehat \pi_!(\ti_{\overline{\Phi}}(\ev_{\Sigma}^*h))
\stackrel{\eqref{eq:diag_FI_t_p_commute}}{=}
\ti_{\overline{\varphi}} (\widehat \pi_!(\ev_\Sigma^*h))
=
\ti_{\overline{\varphi}}(\tau_{\Sigma}h) \,.
$$
Let $h \in \widehat H^k(\varphi;\Z)$.
Similarly to the above we find for the right square in \eqref{eq:ti_vds_trans_commute}:
\begin{equation*}
\tau_\Sigma\vds_\varphi(h)
=
\widehat\pi_!(\ev_\Sigma^*\vds_\varphi(h))
\stackrel{\eqref{eq:vds_nat}}{=}
\widehat\pi_!(\vds_{\overline{\Phi}}(\ev_\Sigma^*h))
\stackrel{\eqref{eq:diag_FI_t_p_commute}}{=}
\vds_{\overline{\varphi}}(\widehat\pi_!(\ev_\Sigma^*h))
=
\vds_{\overline{\varphi}}(\tau_\Sigma h) \,. \qedhere
\end{equation*}
\end{proof}

\subsubsection{Transgression along manifolds with boundary.}
Let $W$ be a compact oriented manifold with boundary.
We consider the space of smooth maps $(W,\partial W) \xrightarrow{(f,g)} (X,A)$ and the restriction map 
$$
r:C^\infty(W,(X,A)) \to C^\infty( \partial W,(X,A))\,, \quad (f,g) \mapsto (f,g)|_{\partial W} \,.
$$
The trivial fiber bundle $\pi:C^\infty(\partial W,(X,A)) \times \partial W \to C^\infty(\partial W,(X,A))$ bounds the trivial fiber bundle $\pi:C^\infty(W,(X,A)) \times W \to C^\infty(W,(X,A))$ along the restriction map $r$.
Let $h \in \widehat H^k(X;\Z)$.
By Proposition~\ref{prop:fiber_bound_rel} the transgressed character $\tau_{\partial W}h$ admits sections along the restriction map $r$ with prescribed covariant derivative $(-1)^{k-\dim W} \fint_W \ev_W^*\curv(h)$.
Similarly, for any relative character $h \in \widehat H^k(\varphi;\Z)$, the transgressed character $\tau_{\partial W}h$ becomes topologically trivial upon pull-back along the restriction map.
A topological trivialization of $r^*\tau_{\partial W}h$ is given by $(-1)^{k-\dim W} \fint_W \ev_W^*(\curv,\cov)(h)$.

\appendix

\section{K\"unneth splittings}\label{app:Kuenneth}
In the appendix we recall the construction of splittings of the K\"unneth sequence on the level of cycles by using the classical Alexander-Whitney and Eilenberg-Zilber maps.
We use these well-known splittings to construct an analogous splitting of the mapping cone K\"unneth sequence on the level of cycles.
In the main text, we refer to these splittings as \emph{K\"unneth splittings}.      

\subsection*{Alexander-Whitney and Eilenberg-Zilber maps.}
We denote the well-known Alexander-Whitney and Eilenberg-Zilber maps by 
\index{Alexander-Whitney map}
\index{Eilenberg-Zilber map}
$$
\xymatrix{
C_*(X \times X';\Z) \ar@<2pt>[rr]^(0.45){AW} && C_*(X;\Z) \otimes C_*(X';\Z) \ar@<2pt>[ll]^(0.55){EZ} \,.
}
$$
These are chain homotopy inverses of each other with $AW \circ EZ = \id_{C_*(X;\Z) \otimes C_*(X';\Z)}$ and $EZ \circ AW$ chain homotopic to the identity on $C_*(X \times X';\Z)$, see \cite[p.~167]{McC01}.
They are used in \cite[Ch.~6]{BB13} to construct a splitting of the K\"unneth sequence on the level of cycles.
Moreover, the Alexander-Whitney map relates the cross product of cochains to the tensor product.

In \cite[Ch.~6]{BB13}, we construct a splitting of the K\"unneth sequence
\begin{equation*}
0
\to \big[ H_*(X;\Z) \otimes H_*(X';\Z) \big]_n
\xrightarrow{\times} H_n(X \times X';\Z) 
\to \Tor( H_*(X;\Z) , H_*(X';\Z))_{n-1}
\to 0
\end{equation*}
on the level of cycles as follows:
Let $s:C_*(X;\Z) \to Z_*(X;\Z)$ be a splitting of the sequence $0 \to Z_*(X;\Z) \xrightarrow{i} C_*(X;\Z) \xrightarrow{\partial} B_{*-1}(X;\Z) \to 0$.
Similarly, we have the inclusion $i'$ and a splitting $s'$ on $X'$. 
Set $S := (s \otimes s')\circ AW$ and $K:= EZ \circ (i \otimes i')$.
Denote the cycles in the tensor product complex by $Z(C_*(X;\Z) \otimes C_*(X';\Z))$. 
Then we obtain:
\begin{equation*}
\xymatrix{
0 \ar[r] & Z_*(X;\Z) \otimes Z_*(X';\Z) \ar@<2pt>[rr]^{i \otimes i'} \ar@<2pt>[drr]^(0.6)K && Z(C_*(X;\Z) \otimes C_*(X';\Z)) \ar@<2pt>[d]^{EZ} \ar[r] \ar@<2pt>[ll]^{s \otimes s'} & \ldots \\
&&& Z_*(X \times X';\Z) \ar@<2pt>[u]^{AW} \ar@<2pt>[ull]^(0.4)S
} 
\end{equation*}
In particular $S \circ K = (s \otimes s') \circ AW \circ EZ \circ (i \otimes i') = \id_{Z_*(X;\Z) \otimes Z_*(X';\Z)}$.  
We refer to the map $S$ as the \emph{K\"unneth splitting} map. 
\index{K\"unneth splitting}

The splitting allows us to decompose any cycle 
$z \in Z_{k+k'-1}(X \times X';\Z)$ according to $z = K \circ S(z) + (z- K \circ S(z))$.
By the K\"unneth sequence, the latter represents a torsion class in $H_{k+k'-1}(X \times X';\Z)$, whereas the former is a sum of cross products of cycles in $X$ and $X'$, respectively.
Thus 
$$
K \circ S(z) = \sum_{(i,j)\atop i+j = k+k'-1}\sum_{\alpha \in I} z_i^\alpha \times {z'}_j^\alpha
$$ 
for appropriate cycles $z_i^\alpha \in Z_i(X;\Z)$ and ${z'}_j^\alpha \in Z_j(X';\Z)$.

\subsection*{The mapping cone K\"unneth splitting.}
Let $\varphi:A \to X$ be a smooth map.
We consider the induced map $\varphi \times \id_{X'}:A \times X' \to X \times X'$.
We use the Alexander-Whitney and Eilenberg-Zilber maps above to define Alexander-Whitney and Eilenberg-Zilber maps for the mapping cone complexes such that the following diagram commutes:
\index{Alexander-Whitney map!mapping cone $\sim$}
\index{Eilenberg-Zilber map!mapping cone $\sim$}
\index{mapping cone!Alexander-Whitney map}
\index{mapping cone!Eilenberg-Zilber map}
$$
\xymatrix{
C_*(\varphi;\Z) \otimes C_*(X';\Z) \ar@<-1pt>[d]^(0.51){EZ_{\varphi \times \id_{X'}}} \ar@{=}[r] & \big(C_*(X;\Z) \otimes C_*(X';\Z) \big) \oplus \big( C_{*-1}(A;\Z) \otimes C_*(X';\Z) \big)  \ar@<37pt>[d]^(0.51){EZ_{A \times X'}} \ar@<-37pt>[d]^(0.51){EZ_{X \times X'}} \\
C_*(\varphi \times \id_{X'};\Z) \ar@<5pt>[u]^{AW_{\varphi \times \id_{X'}}}  \ar@{=}[r] &  C_*(X \times X';\Z) \oplus C_{*-1}(A \times X';\Z) \ar@<-33pt>[u]^{AW_{A \times X'}} \ar@<41pt>[u]^{AW_{X \times X'}} 
}
$$
Explicitly, we set $AW_{\varphi \times \id_{X'}}:= AW_{X \times X'} \otimes AW_{A \times X'}$ and $EZ_{\varphi \times \id_{X'}}:= EZ_{X \times X'} \oplus EZ_{A \times X'}$.
Since the usual Alexander-Whitney and Eilenberg-Zilber maps are natural chain maps, so are the maps on the mapping cone complexes.
Moreover $AW_{\varphi \times \id_{X'}} \circ EZ_{\varphi \times \id_{X'}} = \id_{C_*(\varphi \times \id_{X'};\Z)}$ and $EZ_{\varphi \times \id_{X'}} \circ AW_{\varphi \times \id_{X'}}$ is chain homotopic to $\id_{C_*(\varphi;\Z) \otimes C_*(X';\Z)}$.
The algebraic K\"unneth sequence for the homology of the complexes $C_*(\varphi;\Z)$ and $C_*(X';\Z)$ now reads:
$$
0 \to 
\big[H_*(\varphi;\Z) \otimes H_*(X';\Z)\big]_n \to
H_n(C_*(\varphi;\Z) \otimes C_*(X';\Z)) \to 
\Tor(H_*(\varphi;\Z),H_*(X';\Z))_{n-1} \to
0 \,.
$$
The mapping cone Alexander-Whitney and Eilenberg-Zilber maps yield isomorphisms:
$$
\xymatrix{
H_*(\varphi \times \id_{X'};\Z) \ar@<2pt>[rrr]^(0.45){AW_{\varphi \times \id_{X'}}} &&&
H_*(C_*(\varphi;\Z) \otimes C_*(X';\Z)) \ar@<2pt>[lll]^(0.55){EZ_{\varphi \times \id_{X'}}} \,.
}
$$
We thus obtain the topological K\"unneth sequence:
\begin{equation*}\label{eq:top_Kuenneth}
 0 \to 
\big[H_*(\varphi;\Z) \otimes H_*(X';\Z)\big]_n \to
H_n(\varphi\times \id_{X'};\Z) \to 
\Tor(H_*(\varphi;\Z),H_*(X';\Z))_{n-1} \to
0 \,.
\end{equation*}
We construct a splitting of this K\"unneth sequence at the level of cycles.
Since the group of boundaries $B_*(\varphi;\Z)$ of the mapping cone complex is a free $\Z$-module, we have the split exact sequence
\begin{equation*}
\xymatrix{
0 \ar[r] & 
Z_*(\varphi;\Z) \ar@<2pt>^{i_\varphi}[r] &
C_*(\varphi;\Z) \ar^(0.47){\partial_\varphi}[r] \ar@{-->}@<2pt>[l]^{s_\varphi} & 
B_{*-1}(\varphi;\Z) \ar[r]  &
0 \,.
}
\end{equation*}
Let $i_\varphi:Z_*(\varphi;\Z) \to C_*(\varphi;\Z)$ be the inclusion.
Fix a splitting $s_\varphi:C_*(\varphi;\Z) \to Z_*(\varphi;\Z)$.
Similarly, we denote by $i':Z_*(X';\Z) \to C_*(X';\Z)$ the inclusion and by $s':C_*(X';\Z) \to Z_*(X';\Z)$ a splitting for the smooth singular chain complex on $X'$. 

Now put $S := (s_\varphi \otimes s')\circ AW_{\varphi \times \id_{X'}}$ and $K:= EZ_{\varphi \times \id_{X'}} \circ (i_\varphi \otimes i')$.
Then we have
\begin{equation}\label{eq:SK}
S \circ K 
= (s_\varphi \otimes s') \circ AW_{\varphi \times \id_{X'}} \circ EZ_{\varphi \times \id_{X'}} \circ (i_\varphi \otimes i') 
= \id_{Z_*(\varphi;\Z) \otimes Z_*(X';\Z)} \,.
\end{equation}
For cycles $(x,y) \in Z_*(\varphi;\Z)$ and $z' \in Z_*(X';\Z)$ we have $(x,y) \times z'=K((s,t)\otimes z')$.
Likewise, for chains $(a,b) \in C_*(\varphi;\Z)$ and $c \in C_*(X';\Z)$ we have $(a,b) \times c'= EZ_{\varphi \times \id_{X'}}((a,b)\otimes c')$.

Denote the cycles in the tensor product complex by $Z(C_*(\varphi;\Z) \otimes C_*(X';\Z))$.
By \eqref{eq:SK}, we obtain a splitting of the K\"unneth sequence on the level of cycles as in \cite[Ch.~6]{BB13}:
\begin{equation*}
\xymatrix{
0 \ar[r] & Z_*(\varphi;\Z) \otimes Z_*(X';\Z) \ar@<2pt>[rr]^{i_\varphi \otimes i'} \ar@<-0.1pt>[drr]^(0.6)K && Z(C_*(\varphi;\Z) \otimes C_*(X';\Z)) \ar@<4pt>[d]^{EZ_{\varphi\times\id_{X'}}} \ar[r] \ar@{-->}@<2pt>[ll]^{s_\varphi \otimes s'} & \ldots \\
&&& Z_*(\varphi \times \id_{X'};\Z) \ar@{-->}[u]^{AW_{\varphi\times\id_{X'}}} \ar@{-->}@<3.9pt>[ull]^(0.4)S
} 
\end{equation*}
We refer to this as the \emph{mapping cone K\"unneth splitting}.
\index{mapping cone!K\"unneth splitting}
\index{K\"unneth splitting!mapping cone $\sim$}


\nopagebreak[5]
\printindex


\end{document}